\documentclass[a4paper,11pt,leqno]{article}
\pdfoutput=1
\usepackage{amsmath,amsfonts,amsthm,amssymb}
\usepackage[left=3cm, right=3cm, top=3cm, bottom=3.5cm]{geometry}
\usepackage{graphicx}
\usepackage{mathrsfs}
\usepackage{tikz}
\usepackage{marginnote}
\numberwithin{equation}{section}
\usepackage[final]{pdfpages}

\theoremstyle{plain}
\newtheorem{theorem}{Theorem}[section]

\newtheorem{lemma}[theorem]{Lemma}
\newtheorem{proposition}[theorem]{Proposition}
\newtheorem{corollary}[theorem]{Corollary}
\newtheorem{remark}[theorem]{Remark}
\newtheorem{definition}[theorem]{Definition}

\DeclareMathSymbol{\Rho}{\mathalpha}{operators}{"50}
\DeclareMathSymbol{\Mu}{\mathalpha}{operators}{"4D}
\DeclareMathSymbol{\Alpha}{\mathalpha}{operators}{"41}
\DeclareMathSymbol{\Beta}{\mathalpha}{operators}{"42}
\DeclareMathSymbol{\Epsilon}{\mathalpha}{operators}{"45}
\DeclareMathSymbol{\Zeta}{\mathalpha}{operators}{"5A}
\DeclareMathSymbol{\Eta}{\mathalpha}{operators}{"48}
\DeclareMathSymbol{\Iota}{\mathalpha}{operators}{"49}
\DeclareMathSymbol{\Kappa}{\mathalpha}{operators}{"4B}
\DeclareMathSymbol{\Chi}{\mathalpha}{operators}{"58}

\newcommand{\N}{\mathbb{N}}
\newcommand{\R}{\mathbb{R}}

\newcommand{\SF}{\mathbb{S}}

\usepackage{color}

\begin{document}

\title{Minimizing under relaxed symmetry constraints:\\
Triple and $N$-junctions.}
\author{{ Giorgio Fusco\footnote{Dipartimento di Matematica Pura ed Applicata, Universit\`a degli Studi dell'Aquila, Via Vetoio, 67010 Coppito, L'Aquila, Italy; e-mail:{\texttt{fusco@univaq.it}}}}
}
\maketitle

\begin{abstract}
We consider a nonnegative potential $W:\R^2\rightarrow\R$ invariant under the action of the rotation group $C_N$ of the regular polygon with $N$ sides, $N\geq 3$. We assume that $W$ has $N$ nondegenerate zeros and prove the existence of a $N$-junction solution to the vector Allen-Cahn equation. The proof is variational and is based on sharp lower and upper bounds for the energy and on a new pointwise estimate for vector minimizers.
\end{abstract}
\section{Introduction}
This note concerns entire solutions $u:\R^n\rightarrow\R^m$ of the elliptic system
\begin{equation}
\Delta u=W_u(u),
\label{elliptic}
\end{equation}
where $W:\R^m\rightarrow\R$ is a smooth nonnegative function that satisfies
\begin{equation}
0=W(a)<W(u),\;\;a\in A,\;u\in\R^m\setminus A,
\label{W-p}
\end{equation}
and $A=\{a_1,\ldots,a_N\}\subset\R^m$ is a set of $N$ distinct points.

In phase transition theory, a function $W$ that satisfies \eqref{W-p} can be regarded as a model for the bulk free energy of a system that can exist in $N$ equally preferred phases represented by the zeros $a_1,\ldots,a_N$ of $W$.

We focus on \emph{minimizers} that is solutions $u:\R^n\rightarrow\R^m$ of \eqref{elliptic} that satisfy
\begin{equation}
J_\Omega(u+v)\geq J_\Omega(u),\quad
J_\Omega(u)=\int_\Omega\Big(\frac{\vert\nabla u\vert^2}{2}+W(u)\Big)dx.
\label{energy}
\end{equation}
for every bounded domain $\Omega\subset\R^n$ and any $C^1(\bar{\Omega};\R^m)$ map $v$ that coincides with $u$ on $\partial\Omega$.

In the scalar case $m=1$ there is a relationship \cite{Wei} between minimal surfaces and minimizers of \eqref{elliptic} and many deep interesting results \cite{sav}, \cite{dkw} have been obtained in the process of understanding this relationship.

In the vector case $m\geq 2$ the situation is quite different. The lack of basic tools like the maximum principle makes the description of the set of all bounded solutions of \eqref{elliptic} an almost impossible task. To our knowledge the asymmetric layered solutions constructed in \cite{scha} is probably the only known minimizer that does not assume any symmetry. On the other hand, beginning with the triple junction of \cite{bgs} and the quadruple junction of \cite{gs} various symmetric solutions with complex geometric structure were discovered (see \cite{bfs1} and Ch.6 and Ch.7 in \cite{afs}).

The simplest case is where $W$ is invariant under a finite reflection group $G$
\begin{equation}
W(gu)=W(u),\;\;u\in\R^m,\,g\in G.
\label{W-inv0}
\end{equation}
and $G$ acts both on the domain space $\R^n$ and on the target space $\R^m$.
The minimization is on the set of $G$-equivariant maps
\begin{equation}
u(gx)=gu(x),\;\;x\in\R^n,\,g\in G
\label{equiv}
\end{equation}
that map fundamental region $F$ for the action of $G$ on $\R^n$ into fundamental region $\Phi$ for the action of $G$ on $\R^m$
\begin{equation}
u(\bar{F})\subset\bar{\Phi}.
\label{constraint}
\end{equation}
For the triple junction case $G=Z_3$, the group of the symmetries of the equilateral triangle, for the qudruple junction $G$ is the group of the symmetries of a regular tetrahedron.

Restricting to $G$-equivariant maps that satisfy \eqref{constraint} means minimizing under a constraint but the key point here is that it can be shown that the very fact that $G$ is a finite reflection group implies that the constraint \eqref{constraint} is inactive and does not affect the Euler-Lagrange equation yielding a solution of \eqref{elliptic} that satisfies \eqref{constraint}. This is a basic fact since a main difficulty one has to face when dealing with the non convex minimization of $J_\Omega$ is the lack of a method for determining the regions of the domain where the minimal solution $u$ of \eqref{elliptic} is near to one or another of the zeros of $W$. Knowing that $u$ satisfies \eqref{constraint} allows to overcome this difficulty. Indeed, if $W$ has a unique zero in each $\bar{\Phi}$, \eqref{constraint} implies that, in each fundamental region $F$, the minimizer $u$ remains away from all the zeros of $W$ but one of them. This is a very important fact that, in the analysis of the structure of $u\vert_F$, allows to regard $W$ has having a unique zero and, in this sense, to reduce a non convex problem to a convex one and \eqref{constraint} is the starting point for deriving sharp pointwise estimates that yield precise information on the geometric structure of $u$.

The scope of this notes is to make a first step toward removing the assumption of symmetry. We relax the symmetry requirements and consider the problem of the existence of multi-junction solution in $\R^2$ which are equivariant with respect to the rotation group $C_N\subset Z_N$ of the regular polygon with $N$ sides. $C_N$ has order $N$ and $N=\sharp A$ coincides with the number of the minima of $W$ which is assumed to be invariant with respect to $C_N$.

Working in the context of maps equivariant with respect to $C_N$ still helps a lot with the proof of existence but differently from the full reflection group $Z_N$ of all the symmetries of the regular $N$-gon does not allow to impose \eqref{constraint} and consequently to reduce to a convex problem. Therefore a completely new technique must be devised. Here the fact that we work in $\R^2$ plays an important role.

Our precise assumption are the following
\begin{description}
\item[H$_1$] $W:\R^2\rightarrow\R$ is invariant under $C_N$
\begin{equation}
W(\omega u)=W(u),\;\;u\in\R^2,\;\omega=\left(\begin{array}{l}
\cos{\frac{2\pi}{N}}\;\;-\sin{\frac{2\pi}{N}}\\
\sin{\frac{2\pi}{N}}\;\;\;\cos{\frac{2\pi}{N}}
\end{array}\right),
\label{W-inv}
\end{equation}
where $\omega$ is the generator of $C_N$.
\item[H$_2$]
 $W\geq 0$ and $A=\{a,\omega a,\ldots,a^{N-1}a\}$ for some $a\in\R^2\setminus\{0\}$ and the Hessian matrix $W_{uu}(a)$ is positive definite.  Moreover
 \[W_u(u).u\geq 0,\;\;\text{for}\;\vert u\vert\geq M,\;\text{some}\;M>0.\]
 \end{description}

 These assumptions imply the existence of a minimizer $\bar{u}:\R\rightarrow\R^2$ of the problem
 \begin{equation}
\begin{split}
&J_\R(\bar{u})=\min_{v\in\mathscr{A}_a}J_\R(v),\;\;J_\R(v)=\int_\R\Big(\frac{1}{2}\vert\frac{d v}{ds}\vert^2+W(v)\Big)ds,\\
&\mathscr{A}_a=\{v\in H_{\mathrm{loc}}^1(\R;\R^2): \lim_{s\rightarrow\pm\infty}v(s)=a_\pm,\;a_-=a,\;a_+\in A\setminus\{a\}\}.
\end{split}
\label{sig-def}
\end{equation}
The minimizer $\bar{u}$ satisfies
 \begin{equation}
\ddot{u}=W_u(u).
\label{AC}
\end{equation}
We can assume that $a_+=\omega a$:
 \[\begin{split}
 &\lim_{s\rightarrow-\infty}\bar{u}(s)=a,\\
 &\lim_{s\rightarrow+\infty}\bar{u}(s)=\omega a.
 \end{split}\]
Indeed it can be shown that $a_+\not\in\{\omega a,\omega^{-1}a\}$ contradicts the minimality of $\bar{u}$. On the other hand if $a_+=\omega^{-1}a$ then $\omega\bar{u}(-\cdot)$ connects $a$ to $\omega a$.


We prove
\begin{theorem}
\label{Th2}
Assume that $\mathbf{H}_1$ and $\mathbf{H}_2$ hold.
Then there exists a $C_N$-equivariant solution $U:\R^2\rightarrow\R^m$ of \eqref{elliptic}
\begin{equation}
U(\omega x)=\omega U(x),\;\;x\in\R^2.
\label{u-equiv}
\end{equation}
Moreover there are positive constant $\mathring{C}$, $k$, $K$ and $\mathring{r}\geq\frac{(\mathring{C}N)^2}{\pi^2}$ such that
\begin{equation}
\vert U(x)-a\vert\leq Ke^{-k d(x,\partial Q)},\;\;x\in Q,
\label{exp-infty}
\end{equation}
where
\[Q=\{x=x(r,\theta): r>\mathring{r},\;
\frac{\mathring{C}}{r^\frac{1}{2}}<\theta<\frac{2\pi}{N}-\frac{\mathring{C}}{r^\frac{1}{2}}\},\]
$(r,\theta$ are the polar coordinates of $x)$.
\end{theorem}
The solution $U$ in Theorem \ref{Th2} is a $C_N$-equivariant minimizer in the sense that satisfies \eqref{energy} for each $C_N$-equivariant $v$. This is a consequence of the fact that $U$ is obtained as limit of a sequence of $G$-equivariant minimizers $u^R:B_R\rightarrow\R$ of $J_{B_R}$.

Theorem \ref{Th2} can be generalized to the case where different rotation groups act on the domain space $\R^n$ and on the target space $\R^m$. As before we let $\omega$ be the generater of the group $C_N$, $N\geq 2$ acting on $\R^m$ and let $\omega^\frac{1}{h}$, $h=1,\ldots$ be the generator of the group $C_{hN}$ acting on $\R^n$. Then the same proof of Theorem \ref{Th2}, with minor modification, yields
\begin{theorem}
\label{gen-}
Assume that $\mathbf{H}_1$ and $\mathbf{H}_2$ hold. Then there exists a classical solution $U:\R^n\rightarrow\R^m$ of \eqref{elliptic} with the following properties:
\begin{enumerate}
\item $U$ satisfies the equivariance relation
\[U(\omega^{\frac{1}{h}}x)=\omega U(x),\;\;x\in\R^n.\]
\item There are positive constant $\mathring{C}$, $k$, $K$ and $\mathring{r}\geq\frac{(\mathring{C}hN)^2}{\pi^2}$ such that
\begin{equation}
\vert U(x)-a\vert\leq Ke^{-k d(x,\partial Q)},\;\;x\in Q,
\label{exp-infty1}
\end{equation}
where
\[Q=\{x=x(r,\theta): r>\mathring{r},\;
\frac{\mathring{C}}{r^\frac{1}{2}}<\theta<\frac{2\pi}{hN}-\frac{\mathring{C}}{r^\frac{1}{2}}\},\]

\end{enumerate}
\end{theorem}
This theorem extends the result in \cite{acm}, \cite{bfs} to the case of equivariance with respect to the rotation groups $C_N$ and $C_{hN}$ for general $N$ and $h$.

For the classical bistable potential $W:\R\rightarrow\R$, $W(u)=\frac{1}{4}(1-u^2)^2$  we have $N=2$ and it is well known that assumptions $\mathbf{H}_1-\mathbf{H}_2$ are satisfied. Then the following theorem can be regarded as a particular case of  Theorem \ref{gen-}.
\begin{theorem}
\label{Th1}
For each $h=1,\ldots$ the scalar Allen-Cahn equation
 \begin{equation}
\Delta u=u^3-u,\;\;x\in\R^2,
\label{AC2}
\end{equation}
has a classical solution $U:\R^2\rightarrow\R$ which satisfies
 the equivariance relation
\begin{equation}
u(\omega^{\frac{1}{h}}x)=-u(x),\;x\in\R^2,\;\omega^{\frac{1}{h}}=\left(\begin{array}{l}\cos\frac{\pi}{h}\; -\sin\frac{\pi}{h}\\\sin\frac{\pi}{h}\;\;\; \cos\frac{\pi}{h}\end{array}\right).
\label{equi}
\end{equation}

 Moreover there are positive constant $\mathring{C}$, $k$, $K$ and $\mathring{r}\geq\frac{(2\mathring{C}h)^2}{\pi^2}$ such that
\[
\vert U(x)+1\vert\leq Ke^{-k d(x,\partial Q)},\;\;x\in Q,
\]
where
\[Q=\{x=x(r,\theta): r>\mathring{r},\;
\frac{\mathring{C}}{\sqrt{r}}<\theta<\frac{\pi}{h}-\frac{\mathring{C}}{\sqrt{r}}\}.\]
\end{theorem}

From \eqref{equi} it follows that, if $h=2$, the solution $U$ in Theorem \ref{Th1} is saddle shaped. Existence of saddle solution of \eqref{AC2} in $\R^2$  equivariant with respect to $Z_4$, the reflection group of the symmetries of the square was first established in \cite{DFP} and generalized to the case of equivariance with respect to $Z_{2N}$ in \cite{acm}, see also \cite{SO}. Existence and stability of saddle shaped solutions of \eqref{AC2} in $\R^{2n}$ was discussed in \cite{CT}. Theorem \ref{Th1} shows that minimizing in the larger class of maps obtained by relaxing the symmetry constraint to the mere equivariance with respect to the rotation subgroup of $Z_{2N}$ still yields a saddle shaped solution.

We give some sketchy ideas on the complex minimization process that determines a minimizer $u^R$ of $J_{B_R}$ and the estimates necessary to define $U$ as limit of $u^R$ for $R\rightarrow+\infty$.

We show that, for each $r>0$ sufficiently large, there exists a minimizer $u_r$ of
\[\min J_r(v),\;\;J_r(v)=\int_0^{2\pi r}\Big(\frac{\vert\dot{v}\vert^2}{2}+W(v)\Big)ds,\]
in the class of $2\pi r$-periodic $C_N$-equivariant $v$. This and an accurate estimate on $ J_r(u_r)$
allow the derivation of sharp lower and upper bounds for the energy $J_{B_R}(u^R)$ of $u^R$ (see Section \ref{LandU}). A by product of these bounds is the estimate
\begin{equation}
\int_0^R\int_0^{2\pi}\vert\frac{\partial}{\partial r}u^R\vert^2 r d\theta dr\leq C,
\label{Kin-intr}
\end{equation}
with $C>0$ independent of $R$.

Next, in Section \ref{away-ur}, we derive a lower bound for the energy $J_r(v)$ of a $2\pi r$-periodic map $v$ which does not satisfy certain structure requirements necessary to be a minimizer of $J_r(v)$. From this lower bound and the lower and upper bounds for $J_{B_R}(u^R)$ proved in Section \ref{LandU} we obtain that the measure $\vert\Sigma\vert$ of the set of the $r\in(0,R)$ such that the restriction of $u^R$ to the fiber $C_r$ (the circumference of radius $r$) does not have the structure of $u_r$ is small in the sense that

\begin{equation}
\vert\Sigma\vert\leq C,
\label{smallS}
\end{equation}
for some $C>0$ independent of $R$.

For $r\in(0,R)\setminus\Sigma$ the restriction of $u^R$ to the fiber $C_r$ is near $u_r$ and therefore has a layered structure with $N$ layers that by equivariance are equally spaced. This implies that the layer positions are approximately determined by the value of a single angle $\theta_r$. From \eqref{Kin-intr} and \eqref{smallS} it follows that the map $(0,R)\setminus\Sigma\ni r\rightarrow \theta_r$ has some kind of regularity, a sort of Lipschitz property with jumps (cfr. Section \ref{Lip}). This and the results from Section \ref{away-ur} imply the existence of certain open sets $\mathscr{S}$ such that, see Corollary \ref{esseScr},
\begin{equation}
\begin{split}
& x\in\mathscr{S}\setminus\tilde{\Sigma},\;\;\tilde{\Sigma}=\{x:\vert x\vert\in\Sigma\}\\
&\Rightarrow\\
&\vert u^R(x)-a\vert\leq c\delta^\alpha,
\end{split}
\label{small+-}
\end{equation}
where $c\delta^\alpha$ is a small quantity the particular expression of which comes from the characterization of $u_r$ in Section \ref{away-ur} (statements similar to \eqref{small+-} apply to the images of $\mathscr{S}$ through $\omega$).

The key point of the whole proof is to show that \eqref{small+-} implies
\begin{equation}
\vert u^R(x)-a\vert\leq Ke^{-kd(x,\partial\mathscr{S})},\;\;x\in\mathscr{S}.
\label{vero}
\end{equation}
This is a consequence of a pointwise estimate given by Theorem \ref{uR<} that we prove in Section \ref{pointwise}. This is a delicate point. Indeed we have no control on the behaviour of $u^R(x)$ for $x\in\mathscr{S}\cap\tilde{\Sigma}$ and  a priori we can not even exclude that, for some $x\in\tilde{\Sigma}$,  $u^R(x)$ coincides with one of the minima of $W$ different from $a$. Theorem \ref{uR<} states that, provided $\vert x_0\vert$ and $l>0$ are sufficiently large, if $u^R(x)$ is near $a$ or at least remains at a fixed distance from $A\setminus\{a\}$ for $x\in B_l(x_0)\setminus\tilde{\Sigma}$, then the minimality of $u^R$ and the bound \eqref{smallS} imply that, regardless if $x_0\in\tilde{\Sigma}$ or not, $u^R(x_0)$ must necessarily be near $a$.

Pointwise estimates like \eqref{vero} and the pseudo regularity of the map $(0,R)\setminus\Sigma\ni r\rightarrow\theta_r$  allow for the elaborate construction of a set $\mathscr{I}$, see Section \ref{diff-int}, that plays the role of a diffuse interface in the sense that
\begin{equation}
\begin{split}
&\vert u^R(x)-a\vert\leq\frac{C}{\vert x\vert^p},\;\;x\in\partial\mathscr{I}^+,\\
&\vert u^R(x)-\omega^{-1} a\vert\leq\frac{C}{\vert x\vert^p},\;\;x\in\partial\mathscr{I}^-,
\end{split}
\label{small-R-oR}
\end{equation}
where $\partial\mathscr{I}^-\cup\partial\mathscr{I}^+\simeq\partial\mathscr{I}$ and $p>1$.

In Section \ref{lenghtUB} we associate to $\mathscr{I}$ a curve $\gamma^m$ of minimal length and using \eqref{small-R-oR} show that the energy of $u^R$ is mostly contained in $\mathscr{I}$ and its images under $\omega$ and proportional to the length $\vert\gamma^m\vert$ of $\gamma^m$. This and the upper bound for $J_{B_R}(u^R)$ implies an upper bound for $\vert\gamma^m\vert$. We find $\vert\gamma^m\vert\leq R+C$ with $C>0$ independent of $R$. This estimate gives strong control on the shape of $\mathscr{I}$ that as a result is contained in some kind of neighborhood of one of rays of $B_R$. This and another application of Theorem \ref{uR<} lead to the exponential estimate in Theorem \ref{Th2}.

The paper is organized as follows. In Section \ref{LandU} we derive sharp lower and upper bounds for $J_{B_R}(u^R)$. In Section \ref{away-ur} we give quantitative estimate for the one dimensional energy $J_r$ of maps near $2\pi r$-periodic $C_N$-equivariant minimizers $u_r$ used in the derivation of the lower bound in Section \ref{LandU}.
In Section \ref{basic} we prove two basic lemmas. In Section \ref{pointwise} we prove Theorem \ref{uR<} .
 In Section \ref{structure} we derive detailed information on the structure of the minimizer $u^R$ and conclude the proof of Theorem \ref{Th2}.

\section{Basic lemmas}
The assumption on $W$ imply
\begin{lemma}
\label{W-lemma}
There are constants $\delta_W>0$, and $c_W, C_W>0$ such that
\begin{equation}
\begin{split}
&\vert z-a\vert=\delta,\;\;\delta\leq\delta_W,\\
&\Rightarrow\;\;\frac{1}{2}c_W^2\delta^2\leq W(z)\leq\frac{1}{2}C_W^2\delta^2.
\end{split}
\label{W}
\end{equation}
Moreover, given $M>0$, by reducing the value of $\delta_W$ if necessary, we can also assume
\begin{equation}
\begin{split}
&\delta\in(0,\delta_W]\;\text{ and }\;\vert z\vert\leq M,\;\min_{j=1}^N\vert z-\omega^{j-1}a\vert\geq\delta,\\
&\Rightarrow\;\;\frac{1}{2}c_W^2\delta^2\leq W(z).
\end{split}
\label{outW}
\end{equation}
\end{lemma}
We continue with a lower bound for a one dimensional problem. Set $a_-=a$, $a_+=\omega a$ and $\Gamma_0(a_\pm)=\{a_\pm\}$ and $\Gamma_\delta(a_\pm)=\partial B_\delta(a_\pm)$ for $\delta>0$. If $(s_-,s_+)\subset\R$ is a bounded or unbounded interval and $v:(s_-,s_+)\rightarrow\R^m$ is a map in $H_{\mathrm{loc}}^1$ we set
\[J(v,(s_-,s_+))=\int_{s_-}^{s_+}(\frac{\vert\dot{v}\vert^2}{2}+W(v))ds.\]

\begin{lemma}
\label{lower-sigma}
Let $\delta_\pm\in[0,\delta_W]$ and
let $v:(s_-,s_+)\rightarrow\R^m$ a smooth map such that
\begin{equation}
\lim_{s\rightarrow s_\pm}d(v(s),\Gamma_{\delta_\pm}(a_\pm))=0.
\label{Lim-}
\end{equation}

Then
\begin{equation}
J(v,(s_-,s_+))\geq\sigma-\frac{1}{2}C_W(\delta_-^2+\delta_+^2),
\label{sigmadelta}
\end{equation}
where $\sigma=\int_\R\vert\dot{\bar{u}}\vert^2$ is the energy of the heteroclinic $\bar{u}$ that connects $a_-$ to $a_+$.
\end{lemma}
\begin{proof}
1. For $\delta_-=\delta_+=0$ \eqref{sigmadelta} is just the statement of the minimality of $\bar{u}$. Therefore we can assume that either $\delta_-$ or $\delta_+$ or both are positive. From \eqref{Lim-} and $\delta_+>0$, if $s_+=+\infty$, it follows $\int_{s_-}^{s_+}W(v)ds=+\infty$ and \eqref{sigmadelta} holds trivially. The same is true if $\delta_+>0$, $s_+<+\infty$ and $\lim_{s\rightarrow s_+}v(s)$ does not exist. Indeed in this case we have $\int_{s_-}^{s_+}\vert\dot{v}\vert^2 ds=+\infty$. It follows that, if $\delta_+>0$, we can assume $s_+<+\infty$ and moreover that
\begin{equation}
\lim_{s\rightarrow s_+}v(s)=v_+,
\label{Lim}
\end{equation}
for some $v_+\in\Gamma_{\delta_+}(a_+)$. Analogous conclusion applies to the case $\delta_->0$.

2. If both $\delta_-$ and $\delta_+$ are positive and $w_\pm$ is a test map that connects $v_\pm$ to $a_\pm$, the minimality of $\bar{u}$ implies
\[J(v,(s_-,s_+))\geq\sigma-J(w_-)-J(w_+),\]
where $J(w_\pm)$ is the energy of $w_\pm)$. This yields \eqref{sigmadelta} provided we show that $w_\pm$ can be chosen so that
\[J(w_\pm)\leq\frac{1}{2} C_W\delta_\pm^2.\]

3. We choose
\[w_+=(1-\frac{\gamma(s)}{\delta_+})a_+ +\frac{\gamma(s)}{\delta_+}v_+,\;\;\gamma(s)=\delta_+e^{-C_W(s-s_+)}.\]
it follows, using also \eqref{W}

\[\begin{split}
&\frac{1}{2}\int_{s_+}^{+\infty}\vert\dot{w}_+\vert^2 ds=\frac{C_W^2}{2}\vert v_+-a_+\vert^2
\int_{s_+}^{+\infty}e^{-2C_W(s-s_+)}ds=\frac{1}{4}C_W\delta_+^2,\\
&\int_{s_+}^{+\infty}W(w_+)ds\leq\frac{C_W^2}{2}\vert v_+-a_+\vert^2\int_{s_+}^{+\infty}e^{-2C_W(s-s_+)}ds=\frac{1}{4}C_W\delta_+^2,
\end{split}\]
 This and the analogous computation for $J(w_-)$ establish \eqref{sigmadelta} for $\delta_-$ and $\delta_+$ positive. Clearly \eqref{sigmadelta} is valid also if $\delta_-$ or $\delta_+$ vanishes. The proof is complete.
\end{proof}

\begin{lemma}
\label{basic}
Consider a smooth family of lines that are transversal to two distinct $n-1$ dimensional surfaces $S$ and  $S^\prime$. Consider a point $p\in S$ and let $e$ be a unit vector parallel to the line of the family through $p$. Let $\nu$ a unit vector normal to $S$ at $p$. Let $dS$  a small neighborhood of $p$ in $S$ and let $dS^\prime$  the set of the intersections of the lines through $dS$ with $S^\prime$. Finally let $d\Omega$ the union of all segments determined on the lines of the family by $dS$  and $dS^\prime$.
Let $v:O\subset\R^n\rightarrow\R^m$, $O\supset d\Omega$ open, a smooth map that satisfies
\begin{equation}
\begin{split}
&\vert v-a_-\vert\leq\delta,\;\;x\in S,\\
&\vert v-a_+\vert\leq\delta,\;\;x\in S^\prime.
\end{split}
\label{Con}
\end{equation}
Then
\[J_{d\Omega}(v)\geq (\sigma-C_W\delta^2)\min\{\nu\cdot e dS,\nu^\prime\cdot e dS^\prime\}.\]
This inequality still holds if \eqref{Con} is replaced by the condition that each segment in $\Omega$ contains a point where $\vert v-a\vert\leq\delta$ and a point where $\vert v-a_+\vert\leq\delta$.
\end{lemma}
\begin{proof}
Follows from Fig. \ref{fig1} and Lemma \ref{lower-sigma}.
\end{proof}

\begin{figure}
  \begin{center}
\begin{tikzpicture}
\draw [](-5,-1.5).. controls (-4.5,-1)..(-4,.7);
\node at (-4.5,0){$S$};

\draw [](-1.8,1).. controls (-1,-.5) ..(1,-2);
\node at (-1.4,1.2){$S^\prime$};

\draw [thick,->](-4.45,-.71)--(-3.8,-.69);
\draw [thick,->](-4.45,-.71)--(-3.9,-.9);
\draw [thick,->](-.7,-.65)--(-.3,-.25);

\node at (-3.9,-.43){$e$};

\node at (-4,-1.25){$\nu$};

\node at (-.3,-.05){$\nu^\prime$};

\draw (-7,-1)--(1,-1);
\draw (-7,-.8)--(1,-.6);
\draw (-7,-.6)--(1,-.1);
\node at (-4.7,-.72){$d S$};
\node at (-3.2,-.55){$\Omega$};


\end{tikzpicture}
\end{center}
\caption{The region $\omega$.}
\label{fig1}
\end{figure}
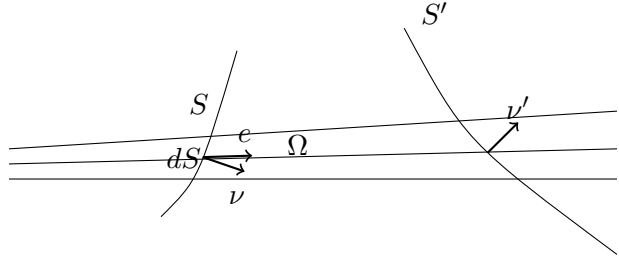

\section{Lower and upper bounds for $J(u^R)$.}\label{LandU}

We construct a lower bound for the energy of $u^R$ by minimizing in each fiber of radius $r\in(\bar{r},R)$, for some fixed $\bar{r}>0$, the energy
\begin{equation}
J_r(v)=\int_{(0,2\pi r)}\Big(\frac{1}{2}\vert\dot{v}\vert^2+W(v)\Big)ds
\label{per-energy}
\end{equation}
in the class $\mathscr{V}_r\subset H^1$ of the map $v$ which are $2\pi r$-periodic
and $C_N$-equivariant:
 \begin{equation}
 u_r(s+\frac{2}{N}\pi r)=\omega u_r(s),\;\;s\in\R.
 \label{s-equi}
 \end{equation}
\begin{proposition}
There exists $\bar{r}>0$ such that $r\geq\bar{r}$ implies the existence of a minimizer $u_r\in\mathscr{V}_r$ of \eqref{per-energy}.
\end{proposition}
\begin{proof}
1.  Given $z_\pm\in\partial B_\delta(a)$, $z_+\neq z_-$, set
\begin{equation}
\tau=\frac{\vert z_+-z_-\vert}{C_W\delta}\leq\frac{2}{C_W},
\label{tau}
\end{equation}
and define $\tilde{u}(t)=z_-+\frac{t}{\tau}(z_+-z_-)$, $t\in[0,\tau]$. We have
\begin{equation}
\begin{split}
&J_{(0,\tau)}(\tilde{u})=\int_0^\tau\Big(\frac{1}{2}\vert\frac{d\tilde{u}}{dt}\vert^2+W(\tilde{u})\big)dt
\leq\frac{1}{2}\int_0^\tau(\frac{1}{\tau^2}\vert z_+-z_-\vert^2+C_W^2\delta^2)dt\\
&=\frac{1}{2}(\frac{\vert z_+-z_-\vert^2}{\tau}+C_W^2\delta^2\tau)=C_W\delta\vert z_+-z_-\vert\leq 2C_W\delta^2
,
\end{split}
\label{JTAU}
\end{equation}
where we have also used  Lemma \ref{W-lemma}.

 2. The minimality of $\bar{u}$ implies that, for small $\delta>0$,  $\bar{u}(\R)\cap\partial B_\delta(\omega^{j-1}a)$, $j=1,2$, is a singleton (see e.g. Lemmas 2.4 and 2.5 in \cite{afs}). It follows that we can choose $z_\pm$ and determine $t_\delta, t^\delta$ by setting:

 \begin{equation}
 \begin{split}
 &z_+=\bar{u}(t_\delta)\in\partial B_\delta(a),\\
 &\omega z_-=\bar{u}(t^\delta)\in\partial B_\delta(\omega a).
 \end{split}
 \label{zeta}
 \end{equation}

 For each $\delta\in(0,\bar{\delta}]$ let $u^\delta$ be the map defined by

 \begin{equation}
 u^\delta(t)=\left\{\begin{array}{l}
 \tilde{u}(t),\;\;t\in[0,\tau)\\
 \bar{u}(t_\delta+t-\tau),\;\;t\in[\tau, t^\delta-t_\delta+\tau).
 \end{array}\right.
 \label{udelta}
 \end{equation}
 Note that the definition of $z_-$ implies $\omega u^\delta(0)=\omega\tilde{u}(0)=\omega z_-=\bar{u}(t^\delta)=u^\delta( t^\delta-t_\delta+\tau)$. It follows that $u^\delta$ can be extended to a $C_N$- equivariant periodic map of period $N( t^\delta-t_\delta+\tau)$.
We have, using also \eqref{JTAU}
\begin{equation}
J(u^\delta)=J_{(t_\delta,t^\delta)}(\bar{u})+J_{(0,\tau)}(\tilde{u})\leq\sigma+2C_W\delta^2.
\label{Jdel-ta}
\end{equation}

3. For $\delta\in(0,\delta_W]$ we can write $\bar{u}=a+\delta n$ with $\delta=\vert\bar{u}-a\vert$ and $n=\frac{\bar{u}-a}{\vert\bar{u}-a\vert}$. This and $\frac{1}{2}\vert\frac{d\bar{u}}{dt}\vert^2=W(\bar{u})$ imply
\[\begin{split}
&\frac{d\delta}{dt}\leq\vert\frac{d\bar{u}}{dt}\vert=\sqrt{2W(\bar{u})}\leq C_W\delta,\\
&\Rightarrow\\
&t_\delta\leq t_{\delta_W}-\frac{1}{C_W}\ln{\frac{\delta_W}{\delta}}.
\end{split}\]
A similar computation yields $t^\delta\geq t^{\delta_W}+\frac{1}{C_W}\ln{\frac{\delta_W}{\delta}}$.
Hence
\[ t^\delta-t_\delta+\tau\geq t^{\delta_W}-t_{\delta_W}+\frac{2}{C_W}\ln{\frac{\delta_W}{\delta}}.\]
From this and \eqref{tau} it follows that, given $r\geq\bar{r}=\frac{N}{2\pi}( t^{\delta_W}-t_{\delta_W}+\frac{2}{C_W})$, there is $\delta\in(0,\delta_W]$ such that
\begin{equation}
r=\frac{N}{2\pi} (t^\delta-t_\delta+\tau).
\label{teta-exs}
\end{equation}

4. From \eqref{teta-exs} and \eqref{Jdel-ta} it follows that the set of $2\pi r$-periodic $C_N$-equivariant map with bounded energy is nonempty and therefore that the existence of a minimizer $u_r$ follows by classical arguments of variational calculus. The minimizer $u_r$ is a solution of \eqref{AC}. The proof is complete.
\end{proof}

\begin{lemma}
 \label{ur}
Given $\delta\in(0,\delta_W]$. Assume $r\geq r_\delta=\frac{4N\sigma}{\pi c_W^2\delta^2}$ and let $u_r$ be a minimizer of \eqref{per-energy}. Then
\[u_r\in\mathscr{V}_r^*,\]
 where $\mathscr{V}_r^*\subset\mathscr{V}_r$ is the set of maps that, after a suitable translation of the independent variable, satisfy
\[\begin{split}
&v(s)\in B_\delta(a),\;\;s\in(0,\frac{2\pi r}{N}-s_{v,\delta}),\;\;\text{some}\;s_{v,\delta}\in(0,\frac{4\sigma}{c_W^2\delta^2}),\\
&v(\frac{2\pi r}{N})\in B_\delta(a^\prime),\;\;\text{some}\;\;a^\prime\in\{\omega a,\omega^{-1}a\},\\
&v(s)\not\in\cup_{\tilde{a}\in A} B_\delta(\tilde{a}),\;\;s\in[\frac{2\pi r}{N}-s_{v,\delta},\frac{2\pi r}{N}].
\end{split}\]
\end{lemma}
\begin{proof}
1. Define $\Lambda_{r,\delta}=\{s\in[0,2\pi r):\min_j\vert u_r(s)-\omega^{j-1}a\vert\geq\delta\}$. From \eqref{Jdel-ta}
 and \eqref{outW} it follows
\[\begin{split}
&\frac{1}{2}c_W^2\delta^2\vert\Lambda_{r,\delta}\vert\leq\int_0^{2\pi r}W(u_r)ds<J_r(u_r)<2N\sigma,\\
&\Rightarrow\quad\vert\Lambda_{r,\delta}\vert\leq\frac{2N(\sigma+2C_W\delta^2)}{c_W^2\delta^2}
<\frac{4N\sigma}{c_W^2\delta^2}=\pi r_\delta.\end{split}\]
Hence $r\geq r_\delta$ implies the existence of $\bar{s}$ such that $u_r(\bar{s})\in B_\delta(\omega^{j-1}a)$ for some $j\in\{1,\ldots,N\}$. By equivariance and modulus a translation of the independent variable we can assume that $j=1$ and $\bar{s}=0$. Then by equivariance we have $u_r(\frac{2\pi}{N}r)\in B_\delta(a^\prime)$ with $a^\prime\in\{\omega a,\omega^{-1}a\}$.

2. It results
\[u_r([0,\frac{2\pi}{N}r])\cap B_\delta(\tilde{a})=\emptyset,\;\;\tilde{a}\not\in\{a,a^\prime\}.\]
Assume instead that $u_r(s)\in B_\delta(\tilde{a})$ for some $\tilde{a}\not\in\{a,a^\prime\}$ and for some $s\in(0,\frac{2\pi}{N}r)$. If this is the case, Lemma \ref{lower-sigma} implies
\[J_{(0,s)}(u_r)\geq\sigma-C_W\delta^2,\quad\quad J_{(s,\frac{2\pi}{N}r)}(u_r)\geq\sigma-C_W\delta^2,\]
hence $J_{(0,\frac{2\pi}{N}r)}(u_r)\geq 2(\sigma-C_W\delta^2)$ in contradiction with \eqref{Jdel-ta}.

3. Finally we observe that the arguments in Lemma 2.4 and Lemma 2.5 in \cite{afs} imply that the minimizer $u_r$, once has entered the ball $B_\delta(a)$ can not reenter in it before entering the ball $B_\delta(a^\prime)$. This and equivariance imply that $\Lambda_{r,\delta}$ is the union of $N$ equal intervals of size $s_\delta<\frac{4\sigma}{c_W^2\delta^2}$. This concludes the proof.
\end{proof}

From Lemmas \ref{lower-sigma} and \ref{ur} it follows that the energy of $u_r$ has a lower bound
\begin{equation}
J_r(u_r)\geq N(\sigma-C_W\delta^2),
\label{lb-ur0}
\end{equation}
where we have also used that $u_r$ is $C_N$-equivariant. From Lemma \ref{ur} and a classical comparison argument \eqref{lb-ur0} can be upgraded to a sharp lower bound that depends on $r$. This is the content of the following lemma.

\begin{lemma}
\label{ur-exp}
There exist $r_{\bar{\delta}}>0$, $\bar{k}>0$ and $\bar{K}>0$ such that, for $r\geq r_{\bar{\delta}}$, it results
\[J_r(u_r)\geq N\sigma-\bar{K} e^{-\bar{k}r},\]
\end{lemma}
\begin{proof}
1. We can assume that the constants $\delta_W>0$ and $c_W>0$ in Lemma \ref{W-lemma} are such that
\begin{equation}
\vert z-a\vert\leq\delta\in(0,\delta_W]\;\;\Rightarrow\; W_u(z)\cdot(z-a)\geq c_W^2\vert z-a\vert^2.
\label{Wuu-a}
\end{equation}

This follows from $W_u(z)=W_{uu}(a)(z-a)+\mathrm{o}(\vert z-a\vert)$ and from the assumption that $W_{uu}(a)$ is positive definite.

2. Set $\varrho=\vert u_r-a\vert$. Since $u_r$ is a solution of \eqref{AC}
Step 1. implies
\begin{equation}
\frac{d^2}{ds^2}\frac{\varrho^2}{2}=\vert\frac{d}{ds}\varrho\vert^2+\frac{d^2}{ds^2}u_r\cdot(u_r-a)
\geq W_u(u_r)\cdot(u_r-a)\geq c_W^2\varrho^2.
\label{Comp}
\end{equation}

3. From Lemma \ref{ur} with $\delta=\delta_W$, for $r\geq r_{\delta_W}=\frac{2 N\sigma}{\pi c_W^2\delta_W^2}$, we have $\varrho(0)\leq\delta_W^2$ and $\varrho(\frac{2\pi}{N}r-s_{\delta_W})\leq\delta_W^2$. This and \eqref{Comp} imply
\begin{equation}
\varrho(s)\leq\delta_W^2
\frac{\cosh{c_W(\frac{\pi}{N}r-\frac{s_{\delta_W}}{2}-s)}}{\cosh{c_W(\frac{\pi}{N}r-\frac{s_{\delta_W}}{2})}},
\label{varro}
\end{equation}
where the right end side is the solution of $\frac{d^2}{ds^2}v=c_W^2v$ that satisfies $v=\delta_W^2$ at the extreme of the interval $(0,\frac{2\pi}{N}r-s_{\delta_W})$.

4. For $s=\frac{\pi}{N}r-\frac{s_{\delta_W}}{2}$ \eqref{varro} yields
\[\varrho(\frac{\pi}{N}r-\frac{s_{\delta_W}}{2})\leq\frac{\delta_W^2}{\cosh{c_W(\frac{\pi}{N}r-\frac{s_{\delta_W}}{2})}}
\leq 2\delta_W^2e^{c_W\frac{s_{\delta_W}}{2}}e^{-c_W\frac{\pi}{N}r}.\]
This and Lemma \ref{lower-sigma} imply
\[J_r(u_r^*)\geq N\sigma-NC_W2\delta_W^2e^{c_W\frac{s_{\delta_W}}{2}}e^{-c_W\frac{\pi}{N}r}\]
that concludes the proof with $\bar{k}=c_W\frac{\pi}{N}$ and $\bar{K}=NC_W2\delta_W^2e^{c_W\frac{s_{\delta_W}}{2}}$.
\end{proof}

From Lemma \ref{ur-exp} we immediately get

\begin{equation}
\begin{split}
&J_{B_R}(u^R)\geq\int_0^R J_r(u^R(r,\frac{\cdot}{r}))dr\geq\int_{r_{\bar{\delta}}}^RJ_r(u_r)dr\\
&\geq N\sigma R-\bar{K}\int_{r_{\bar{\delta}}}^Re^{-\bar{k}r}dr\\
&\geq N\sigma R-\frac{\bar{K}}{\bar{k}}(e^{-\bar{k}r_{\bar{\delta}}}-e^{-\bar{k}R})\geq N\sigma R-C_0,
\end{split}
\label{LB}
\end{equation}
where $C_0>0$ is a constant independent of $R$. We denote by $C,C_0,C_1,\ldots$ generic positive constants that do not depend on $R$.

To derive an upper bound we choose a suitable $C_N$-equivariant test function $u_{\mathrm{test}}:B_R\rightarrow\R^m$
and obtain $J_{B_R}(u^R)\leq J_{B_R}(u_{\mathrm{test}})$.

Set $\theta_N=\frac{2\pi}{N}$. We first define a map $\tilde{u}_{\mathrm{test}}$ in the whole of $\R^2$ and then we identify $u_{\mathrm{test}}$ with the restriction of $\tilde{u}_{\mathrm{test}}$ to $B_R$.
 For $\rho\in(0,+\infty]$  $\alpha<\beta$ we denote by $S_\rho(\alpha,\beta)$ the sector defined by
  \[S_\rho(\alpha,\beta)=\{z(r,\theta):r\in(0,\rho),\;\theta\in[\alpha,\beta]\}.\]
  We begin to define $\tilde{u}_{\mathrm{test}}$ in the sector $S_\infty(-\frac{1}{4}\theta_N,\frac{1}{4}\theta_N)$ . We note that $\omega^{-1}\bar{u}$ connects $\omega^{-1}a$ to $a$ and set
\begin{equation}
\tilde{u}_{\mathrm{test}}=\omega^{-1}\bar{u}(y),\;\;\vert y\vert\leq\tan{\frac{\theta_N}{4}}x,\;x\geq 0.
\label{in-S}
\end{equation}
Then we extend $\tilde{u}_{\mathrm{test}}$ to the sector $S_\infty(j\theta_N-\frac{1}{4}\theta_N,j\theta_N+\frac{1}{4}\theta_N)$ , $j=1,\ldots,N-1$ by equivariance. We now observe that from the above construction $\tilde{u}_{\mathrm{test}}$ is already defined on the boundary of the sector $S_\infty(\frac{1}{4}\theta_N,\frac{3}{4}\theta_N)$ and we have
\begin{equation}
\begin{split}
&\tilde{u}_{\mathrm{test}}(z(r,\frac{1}{4}\theta_N))=\omega^{-1}\bar{u}(\sin{\frac{1}{4}\theta_N}r),\\
&\tilde{u}_{\mathrm{test}}(z(r,\frac{3}{4}\theta_N))=\bar{u}(-\sin{\frac{1}{4}\theta_N}r).
\end{split}
\label{boudry-S}
\end{equation}
Based on this we define $\tilde{u}_{\mathrm{test}}$ in the sector $S_\infty(\frac{1}{4}\theta_N,\frac{3}{4}\theta_N)$ by setting
\[\tilde{u}_{\mathrm{test}}= \omega^{-1}\bar{u}(\sin{\frac{\theta_N}{4}}r)(\frac{3}{2}-2\frac{\theta}{\theta_N})
 +\bar{u}(-\sin{\frac{\theta_N}{4}}r)(2\frac{\theta}{\theta_N}-\frac{1}{2}),\]
 and extend it by equivariance to the sector $S_\infty(j\theta_N+\frac{1}{4}\theta_N,j\theta_N+\frac{3}{4}\theta_N)$, $j=1,\ldots,N-1$. This complete the definition of $\tilde{u}_{\mathrm{test}}$ as a continuous $C_N$-equivariant $H_{\mathrm{loc}}^1$ map.
Since both the functions on the right hand side of \eqref{boudry-S} converge exponentially to $a$ as $r\rightarrow+\infty$ with first derivatives that converge exponentially to $0$ we have
\[J_{S_\infty(\frac{1}{4}\theta_N,\frac{3}{4}\theta_N)}(\tilde{u}_{\mathrm{test}})\leq C,\;\;\text{some}\;C>0.\]
 On the other hand from \eqref{in-S} we have
\[J_{S_R(-\frac{1}{4}\theta_N,\frac{1}{4}\theta_N)}(\tilde{u}_{\mathrm{test}})
<\int_0^R\int_{-\infty}^{+\infty}(\frac{1}{2}\vert\bar{u}^\prime\vert^2+W(\bar{u}))dydx=\sigma R.\]
and we can conclude
\begin{equation}
\label{UB}
J_{B_R}(u^R)\leq J_{B_R}(\tilde{u}_{\mathrm{test}})\leq N\sigma R+C_1,
\end{equation}
for some constant $C_1>0$.

We remark that this and \eqref{LB} imply
\begin{equation}
\int_0^R\int_0^{2\pi}\vert\frac{\partial}{\partial r}u^R\vert^2 r d\theta dr\leq 2(C_0+C_1).
\label{du-dr-bound}
\end{equation}

\section{Quantitative estimate of the energy of maps near $u_r$.} \label{away-ur}

Next we derive a lower bounds for the energy of maps that do not have the structure of minimizers described in Lemma \ref{ur}.
As before let $\mathscr{V}_r\subset H^1$  the class of $2\pi r$-periodic $C_N$-equivariant maps. In the following, without explicit mention, we characterize a map $v\in\mathscr{V}_r$ by the properties of a suitable translation of it.
 \begin{proposition}
 \label{away} There is a constant $c>0$ such that, given $\alpha\in(0,1)$ and a number $\delta\in(0,\delta_W]$ sufficiently small,
 there are $r_\delta>0$, and $s_{v,\delta}\in(0,\frac{4\sigma}{c_W^2\delta^2})$ such that $r\geq r_\delta$ implies
 \[v\in\mathscr{V}_r\setminus\mathscr{V}_r^*\;\Rightarrow\;J_r(v)\geq N\sigma+\delta^{1+\alpha},\]
 where $\mathscr{V}_r^*\subset\mathscr{V}_r$ is the set of maps that after a suitable translation satisfy:
 \begin{equation}
 v(s)\in B_{c\delta^\alpha}(a),\;\;s\in(0,\frac{2\pi r}{N}-s_{v,\delta})
 \label{V*}
 \end{equation}
 \end{proposition}
 \begin{proof}
  We divide $\mathscr{V}_r\setminus\mathscr{V}_r^*$ in four parts $\mathscr{V}_r^i$, $i=1,2,3,4$ which are defined in sequence and satisfy
 \[\begin{split}
 &\mathscr{V}_r=\mathscr{V}_r^{0,c};\;\;\;\mathscr{V}_r^{4,c}=\mathscr{V}_r^*,\\
 &\mathscr{V}_r^{i,c}=\mathscr{V}_r^{i+1}\cup\mathscr{V}_r^{i+1,c},\;\;n=0,1,2,3,\\
 &\mathscr{V}_r^{i+1,c}=\mathscr{V}_r^{i,c}\setminus\mathscr{V}_r^{i+1}.
 \end{split}\]
1. We define
\[
\mathscr{V}_r^1=\{v\in\mathscr{V}_r:v(s)\not\in\cup_{j=1}^N B_\delta(\omega^{j-1}a),\;s\in[0,2\pi r)\}.\]
Therefore $v\in\mathscr{V}_r^{1,c}$ if and only if there are $s\in[0,2\pi r)$ and $1\leq j\leq N$ such that
$v(s)\in B_\delta(\omega^{j-1}a)$. This and equivariance imply
\[\mathscr{V}_r^{1,c}=\{v\in\mathscr{V}_r:v(0)\in\overline{B_\delta}(a)\},\]
where, as remarked before, we actually mean the set of maps that, after a suitable translation, satisfy $v(0)\in B_\delta(a)$.

For $v\in\mathscr{V}_r^1$ from \eqref{W} we have
\begin{equation}
J_r(v)\geq\int_0^{2\pi r}W(v)ds\geq\pi c_W^2\delta^2 r,\;\;v\in\mathscr{V}_r^1.
\label{E-Vsharp}
\end{equation}
2. We define

\[\mathscr{V}_r^2=\{v: v(0)\in\overline{B_\delta}(a),\;v(\bar{s})\in\overline{B_\delta}(\omega^{j-1}a), \text{for some}\,\bar{s}\in(0,\frac{2\pi r}{N}),\text{and some}\,3\leq j\leq N\}.\]
If $v\in\mathscr{V}_r^2$  equivariance and Lemma \ref{lower-sigma} imply
\begin{equation}
J_r(v)>2N(\sigma-C_W\delta^2).
\label{E-Vsharp2}
\end{equation}
Note that
\begin{equation}
\begin{split}
&\mathscr{V}_r^{2,c}=\{v\in\mathscr{V}_r:v(0)\in\overline{B_\delta}(a)),\;v(\frac{2\pi r}{N})=\omega v(0)\in \overline{B_\delta}(\omega a)\,\, \text{and}\\
&\, v(s)\not\in\overline{B_\delta}(\omega^{j-1} a),\;3\leq j\leq N, s\in(0,\frac{2\pi r}{N})\}
\end{split}
\label{Vr2c}
\end{equation}

3.
Let $\delta^\prime\in(\delta,\frac{1}{2}\vert\omega a-a\vert)$ a number to be chosen later. Observe that for each $v\in\mathscr{V}_r^{2,c}$, for the same translation of $v$ considered in \eqref{Vr2c}, we can define
\begin{equation}
\begin{split}
&s_v^-=\inf\{s\in(0,\frac{2\pi r}{N}):\vert v(s)-a\vert=\delta^\prime\},\\
&s_v^+=\sup\{s\in(0,\frac{2\pi r}{N}):\vert v(s)-\omega a\vert=\delta^\prime\}.
\end{split}
\label{s+s-}
\end{equation}
We set
\[\mathscr{V}_r^3=\{v\in\mathscr{V}_r^{2,c}: v(\bar{s})\in\overline{B_\delta}(a)\cup\overline{B_\delta}(\omega a),\,\text{for some}\,\bar{s}\in(s_v^-,s_v^+)\}\]
To derive a lower bound for the energy of $v\in\mathscr{V}_r^3$ we suppose (the case $v(\bar{s})\in\overline{B_\delta}(\omega a)$ is similar) that
\begin{equation}
v(\bar{s})\in\overline{B_\delta}(a)
\label{vinBa}
\end{equation} and define
\[\tilde{s}=\inf\{s<s_v^-:v(t)\not\in\overline{B_\delta}(a), t\in[s,s_v^-)\}.\]
A standard computation and \eqref{W} shows that
\[\begin{split}
&\int_{\tilde{s}}^{s_v^-}(\frac{\vert\dot{v}\vert^2}{2}+W(v))ds\geq\frac{\vert v(s_v^-)-v(\tilde{s})\vert^2}{2(s_v^--\tilde{s})}+\frac{1}{2}c_W^2\delta^2(s_v^--\tilde{s})\\
&\geq\frac{(\delta^\prime-\delta)^2}{2(s_v^--\tilde{s})}+\frac{1}{2}c_W^2\delta^2(s_v^--\tilde{s})\geq c_W\delta(\delta^\prime-\delta).
\end{split}\]
This \eqref{vinBa}, $v(\frac{2\pi r}{N})\in\overline{B_\delta}(\omega a)$ and Lemma \ref{lower-sigma} that imply
\[J_r(v,(\bar{s},\frac{2\pi r}{N}))\geq\sigma-C_W\delta^2,\]
yield, provided we choose $\delta^\prime=c\delta^\alpha$ with $c=1+\frac{1}{c_W N}+\frac{C_W}{c_W}$,
\begin{equation}
J_r(v)\geq N(\sigma-C_W\delta^2+c_W\delta(\delta^\prime-\delta))
\geq N\sigma+\delta^{1+\alpha},
\label{E-Vsharp3}
\end{equation}

4.
Observe that
\[\begin{split}
&\mathscr{V}_r^{3,c}=\{v\in\mathscr{V}_r^{2,c}:v(s)\in B_{c\delta^\alpha}(a),\,s\in[0,s_v^-),\,
v(s)\in B_{c\delta^\alpha}(\omega a),\,s\in[s_v^+,\frac{2\pi r}{N}),\\
&v(s)\not\in\overline{B_\delta}(\omega^{j-1}a),\,s\in[s_v^-,s_v^+],\,1\leq j\leq N\},
\end{split}\]
and define
\[\mathscr{V}_r^4=\{v\in\mathscr{V}_r^{3,c}:s_v^+-s_v^-\geq\frac{4\sigma}{c_W^2\delta^2}\}.\]
Then \eqref{outW} implies
\begin{equation}
J_r(v)\geq\frac{N}{2}c_W^2({s_v^+}-{s_v^-})\delta^2=2N\sigma.
\label{Jr4}
\end{equation}

5. If $r\geq r_\delta=\frac{N\sigma+\delta^{1+\alpha}}{\pi c_W^2\delta^2}$ and $v\in\mathscr{V}_r^1$,  \eqref{E-Vsharp} yields \begin{equation}
J_r(v)\geq N\sigma+\delta^{1+\alpha}.
\label{thesis}
\end{equation}
Since $\delta$ is a small number  \eqref{E-Vsharp2} and \eqref{Jr4} imply that \eqref{thesis} is satisfied also for $v\in \mathscr{V}_r^2\cup\mathscr{V}_r^4$. From this and \eqref{E-Vsharp3} we conclude that \eqref{thesis} holds for $v\in\mathscr{V}_r\setminus\mathscr{V}_r^{4,c}$. To complete the proof it remain to show that  $\mathscr{V}_r^{4,c}$ is a subset of $\mathscr{V}_r^*$. By inspecting the expressions of $\mathscr{V}_r^{j,c}$, $j=1,2,3$ and the definition of $\mathscr{V}_r^4$ we find that, after a suitable translation,  $v\in\mathscr{V}_r^{4,c}\subset\mathscr{V}_r$ satisfies
\begin{equation}
\begin{split}
&v(s)\in B_{c\delta^\alpha}(a),\;s\in[0,s_v^-),\quad v(s)\in B_{c\delta^\alpha}(\omega a),\;s\in(s_v^+,\frac{2\pi r}{N}],\\
&\;s_v^+-s_v^-<\frac{4\sigma}{c_W^2\delta^2}.
\end{split}
\label{befor-tran}
\end{equation}
By equivariance we have
\[v(s)\in B_{c\delta^\alpha}(\omega a),\;s\in(s_v^+,\frac{2\pi r}{N}]\;\Leftrightarrow\;
v(s)\in B_{c\delta^\alpha}(a),\;s\in(s_v^+-\frac{2\pi r}{N},0]\]
That together with \eqref{befor-tran}imply
\[v(s)\in B_{c\delta^\alpha}(a),\;s\in(s_v^+-\frac{2\pi r}{N},s_v^-).\]
The translation $s\rightarrow s+s_0$, $s_0=\frac{2\pi r}{N}-s_v^+$ transforms this equation into \eqref{V*} with $s_{v,\delta}=s_v^+-s_v^-$.  The proof is complete.
\end{proof}

Next we show that the measure of the set of the fibers where the profile of a minimizer $u^R$ is not in $\mathscr{V}_r^*$ is bounded independently of $R$
\begin{lemma}
\label{Sigma} Let $\delta\in(0,\delta_W]$, $\alpha\in(0,1)$, $r_\delta$ and $\mathscr{V}_r^*$ as before.
Set $\Sigma=\{r\in[r_\delta,R):u^R(x(r,\frac{\cdot}{r}))\in\mathscr{V}_r\setminus\mathscr{V}_r^*\}$. Then
\[\vert\Sigma\vert\leq\frac{C_2}{\delta^{1+\alpha}},\]
\end{lemma}
where $C_2=2C_0+C_1+N\sigma r_0$.
\begin{proof}
From \eqref{UB} we have
\begin{equation}
\begin{split}
&\int_{(r_0,R)}J_r(u^R(x(r,\frac{\cdot}{r})))dr
=\int_{(r_0,r_\delta)}J_r(u^R(x(r,\frac{\cdot}{r})))dr\\
&+\int_{(r_\delta,R)\setminus\Sigma}J_r(u^R(x(r,\frac{\cdot}{r})))dr+\int_\Sigma J_r(u^R(x(r,\frac{\cdot}{r})))dr\leq N\sigma R+C_1.
\end{split}
\label{UBN}
\end{equation}
From \eqref{LB} and Proposition \ref{away} we have
\[\begin{split}
&\int_{(r_0,r_\delta)}J_r(u^R(x(r,\frac{\cdot}{r})))dr\leq N\sigma(r_\delta-r_0)-C_0,\\
&\int_{(r_\delta,R)\setminus\Sigma}J_r(u^R(x(r,\frac{\cdot}{r})))dr
\geq N\sigma (R-r_\delta-\vert\Sigma\vert)-C_0
,\\
&\int_{\Sigma}J_r(u^R(x(r,\frac{\cdot}{r})))dr
\geq (N\sigma +\delta^{1+\alpha})\vert\Sigma\vert.
\end{split}\]
This and \eqref{UBN} conclude the proof.
\end{proof}

We set
\begin{equation}
\tilde{\Sigma}=\{x\in\R^2: r(x)\in\Sigma\},
\label{tildeSigma}
\end{equation}
where we have used the notation  $(r(x),\theta(x))$ for the radial coordinates of $x$.

\section{A pointwise estimate for $u^R$}\label{pointwise}
Let $u^R:B_R\rightarrow\R^2$ be a minimizer of \eqref{energy}. Then the smoothness of $W$ and elliptic theory imply
that we can assume
\begin{equation}
\|u^R\|_{L^\infty}+\|\nabla u^R\|_{L^\infty}\leq M,
\label{uR-bounds}
\end{equation}
for some $M>0$ independent of $R$.

Observe that the properties of $W$ imply
\begin{lemma}
\label{W-mon}
Given $\delta^*\in(0,\delta_W]$ there is $\delta_0>0$ such that,
 for each $n\in\SF^1$, $\delta\in(0,\delta_0]$ and $d>\delta$ that satisfy
 \[a+d n\in B_M\setminus\cup_{j=2}^{N}B_{\delta^*}(\omega^{j-1}a),\]
 we have
\begin{equation}
W(a+\delta n)<W(a+d n).
\label{W-}
\end{equation}
\end{lemma}
\noindent Next we present a result which is essential for our analysis and may be interesting in itself.

If $v:B_R\rightarrow\R^2$ is a $H^1$ map we use the polar form of $v$:
\begin{equation}
\begin{split}
&v=a+q^vn^v,\;\;\text{ on }\;\{v\neq a\},\\
&q^v=\vert v-a\vert,\,n^v=\frac{v-a}{\vert v-a\vert}\in\SF^1,
\end{split}
\label{polar}
\end{equation}
and the identity
\begin{equation}
\vert\nabla v\vert^2=\vert\nabla q^v\vert^2+({q^v})^2\vert\nabla n^v\vert^2.
\label{polar-d}
\end{equation}
We now choose the number $\delta\in(0,\delta_W]$ introduced in Lemma \ref{ur} in such a way that the inequality \eqref{W-} can be applied. We fix a number $\delta^*\in(0,\delta_W]$ and assume $\delta\in(0,\delta_0]$, $\delta_0$ as in Lemma \ref{W-mon}.

\begin{theorem}\label{uR<}
Assume that a minimizer $u^R:B_R\rightarrow\R^2$ of \eqref{energy}, a ball $B_l(x_0)\subset B_R$ and a set $\Sigma\subset(0,R)$  satisfy, for some constants $c>0$ and $C>0$,

\begin{equation}
\vert u^R(x)-a\vert\leq c\delta^\alpha,\;\;x\in B_l(x_0)\setminus\tilde{\Sigma},\;\;\text{for some}\:\alpha\in(0,1)
\label{Cond}
\end{equation}
where $\tilde{\Sigma}=\{x\in B_R:\vert x\vert\in\Sigma\}$, and
\begin{equation}
\vert\Sigma\vert\leq\frac{C}{\delta^{1+\alpha^\prime}},\;\;\text{for some}\;\alpha^\prime\in(0,1).
\label{SigmaSize}
\end{equation}

Assume that
\[2\alpha+\alpha^\prime<1.\]
 Then there is a constant $D>0$ such that for each $\delta\leq D^\frac{1}{1-(2\alpha+\alpha^\prime)}$ there exist $l_\delta>0$ and $r_\delta>0$ which are independent of $R$ and such that
\begin{equation}
\begin{split}
&l\geq l_\delta,\\
&r(x_0)\geq r_\delta,
\end{split}
\label{l-r-cond}
\end{equation}
imply
\begin{equation}
\vert u^R(x_0)-a\vert\leq 2\delta.
\label{onSigma}
\end{equation}
\end{theorem}

Before presenting the proof some comments are in order. The point of the Theorem is that, if \eqref{Cond} holds for sufficiently large $r(x_0)$ and $l$, then, in spite of the fact that nothing is a priori known on the behavior of
$u^R$ on $\tilde{\Sigma}$, the inequality \eqref{onSigma} is satisfied also when $x_0\in\tilde{\Sigma}$, that is when the center of the ball is in $\tilde{\Sigma}$ . Now a comment on the proof. Suppose we knew that
\begin{equation}
\min_{j=2}^N\vert u^R(x)-\omega^{j-1}a\vert\geq\delta^*,\;\;x\in\tilde{\Sigma}\cap B_l(x_0).
\label{Cond*}
\end{equation}
Then we could directly invoke Theorem 5.3 in \cite{afs} or Theorem 1.2 in \cite{f0} and conclude \eqref{onSigma} provided $l>0$ is sufficiently large. But we can not assume \eqref{Cond*} since we can not exclude that, at some point $\bar{x}\in\tilde{\Sigma}\cap B_l(x_0)$ it results
$u^R(\bar{x})=\omega^{j-1}a$ for some $j>1$. In spite of this we will see that the bound on the size of $\Sigma$ in \eqref{SigmaSize} together with \eqref{Cond} allow to overcome this difficulty and to extend the proof of the quoted theorems in order to cover the case at hand.

\begin{proof}
In this proof, to simplify the notation, we simply write $B_l, B_{l\pm\eta}$ instead of $B_l(x_0), B_{l\pm\eta}(x_0)$. By
inspecting the proof of Theorem 5.3 given in Section 5.5 of \cite{afs} we see that the following result, analogous of Lemma 5.3 in \cite{afs}, can be established by exactly the same arguments used in the proof of that lemma.
\begin{lemma}
\label{lemma53}
Let $\tilde{u}:B_R\rightarrow\R^2$ a $C^{0,1}$, $C_N$-equivariant map (not necessarily a minimizer) that satisfies \eqref{uR-bounds}.

Given $l_0>0$, $\eta>0$, assume that, for some $l\geq l_0+\eta$,
\[\vert\tilde{u}-a\vert\leq\delta,\;\;\text{ on }\;B_l(x_0),\]
where $B_l(x_0)\subset B_R$ satisfies $B_l(x_0)\cap B_l(\omega x_0)=\emptyset$.
Then there exists a $C^{0,1}$, $C_N$-equivariant map $v:B_R\rightarrow\R^2$ that coincides with $\tilde{u}$ on $B_R\setminus\cup_{j=1}^N B_l(\omega^{j-1}x_0)$ and satisfies
\begin{equation}
J_{B_l(x_0)}(\tilde{u})-J_{B_l(x_0)}(v)\geq k\vert B_{l-\eta}(x_0)\cap\{q^{\tilde{u}}=\delta\}\vert,
\label{energy+}
\end{equation}
where $k=k(W,l_0,\eta,\delta)$ is a constant that does not depend on $l>l_0+\eta$ and $R$.
 \end{lemma}
 In the following we will assume
 \begin{equation}
 l_0\geq \frac{\delta}{M}.
 \label{l-0}
 \end{equation}
 Next we define a deformation of $u^R$ into a map $\tilde{u}$ that satisfies the assumptions of Lemma \ref{lemma53} and derive a quantitative estimate for the energy $J_{B_{l+\eta}}(\tilde{u})-J_{B_{l+\eta}}(u^R)$ spent in the deformation. As in the proof of Lemma 5.4 \cite{afs}, we set $p^{u^R}(x)=q^{u^R}(x)-(q^{u^R}(x)-\delta)^+$ and define $\tilde{u}=a+q^{\tilde{u}}n^{u^R}$ by, see Figure \ref{q}
 \begin{equation}
 \begin{split}
 &q^{\tilde{u}}=\left\{\begin{array}{l}
 q^{u^R}(x),\;\;\text{ for }\;x\in B_R\setminus\cup_{j=1}^N\omega^{j-1}B_{l+\eta},\\
 p^{u^R}(x)+g(x)(q^{u^R}(x)-\delta)^+,\;\text{ for }\;x\in B_{l+\eta},
 \end{array}\right.\\\\
 &g(x)=\left\{\begin{array}{l}
 0,\;\text{ for }\;x\in B_l,\\
 \frac{\vert x-x_0\vert-l}{\eta},\;\text{ for }\;x\in \overline{B}_{l+\eta}\setminus B_l.
 \end{array}\right.
 \end{split}
 \label{utilde}
 \end{equation}

\begin{figure}
  \begin{center}
\begin{tikzpicture}[scale=.85]
\draw [fill] (0,0) circle [radius=0.035];
\draw [fill] (0,2) circle [radius=0.035];
\draw [fill] (0,-2) circle [radius=0.035];
\draw [fill] (0,-4) circle [radius=0.035];

\draw [thin] (-3,0) -- (3,0);
\draw [dotted] (-3,.7) -- (3,.7);

\draw [thin] (-3,2) -- (3,2);
\draw [dotted] (-3,2.7) -- (3,2.7);

\draw [thin] (-3,-2) -- (3,-2);
\draw [blue] (-3,-1) -- (-2,-2);
\draw [blue] (-2,-2) -- (2,-2);
\draw [blue] ( 2,-2) -- (3,-1);

\draw [thin] (-3,-4) -- (3,-4);
\draw [dotted] (-3,-3.3) -- (3,-3.3);


\draw [thin] (-3,1.8) -- (-3,3.3);
\draw [thin] (3,1.8) -- (3,3.3);

\draw [thin] (-3,-.2) -- (-3,1.3);
\draw [thin] (3,-.2) -- (3,1.3);

\draw [thin] (-3,-2.2) -- (-3,-.7);
\draw [thin] (3,-2.2) -- (3,-.7);

\draw [thin] (-3,-2.7) -- (-3,-4.2);
\draw [thin] (3,-2.7) -- (3,-4.2);


\draw [dotted] (-2,1.9) -- (-2,3.);
\draw [dotted] (2,1.9) -- (2,3.);

\draw [dotted] (-2,-.1) -- (-2,1.);
\draw [dotted] (2,-.1) -- (2,1.);

\draw [dotted] (-2,-2.1) -- (-2,-1.2);
\draw [dotted] (2,-2.1) -- (2,-1.2);

\draw [dotted] (-2,-3) -- (-2,-4.1);
\draw [dotted] (2,-3) -- (2,-4.1);


\draw[blue, thick, domain=-3.3:3.3] plot (\x,{.4*((\x+2.3)*(\x+1.5)*(\x+.5)*(\x-1.5))
/(1+(\x+1)*(\x+1)*(\x+1)*(\x+1)))+2.7});


\draw[blue, thick, domain=-3.3:-3.] plot (\x,{.4*((\x+2.3)*(\x+1.5)*(\x+.5)*(\x-1.5))
/(1+(\x+1)*(\x+1)*(\x+1)*(\x+1)))+.7});

\draw[blue, thick, domain=-2.3:-1.5] plot (\x,{.4*((\x+2.3)*(\x+1.5)*(\x+.5)*(\x-1.5))
/(1+(\x+1)*(\x+1)*(\x+1)*(\x+1)))+.7});

\draw[blue, thick, domain=-.5:1.5] plot (\x,{.4*((\x+2.3)*(\x+1.5)*(\x+.5)*(\x-1.5))
/(1+(\x+1)*(\x+1)*(\x+1)*(\x+1)))+.7});

\draw[blue, thick, domain=3:3.3] plot (\x,{.4*((\x+2.3)*(\x+1.5)*(\x+.5)*(\x-1.5))
/(1+(\x+1)*(\x+1)*(\x+1)*(\x+1)))+.7});

\draw [blue, thick] (-3,.7) -- (-2.3,.7);
\draw [blue, thick] (-1.5,.7) -- (-.5,.7);
\draw [blue, thick] (1.5,.7) -- (3,.7);


\draw[blue, thick, domain=-3.:-2.3] plot (\x,{-.4*((\x+2.)*(\x+2.3)*(\x+1.5)*(\x+.5)*(\x-1.5))
/(1+(\x+1)*(\x+1)*(\x+1)*(\x+1)))-3.3});

\draw[blue, thick, domain=2.:3.] plot (\x,{.4*((\x-2.)*(\x+2.3)*(\x+1.5)*(\x+.5)*(\x-1.5))
/(1+(\x+1)*(\x+1)*(\x+1)*(\x+1)))-3.3});

\draw[blue, thick, domain=-3.3:-3.] plot (\x,{.4*((\x+2.3)*(\x+1.5)*(\x+.5)*(\x-1.5))
/(1+(\x+1)*(\x+1)*(\x+1)*(\x+1)))-3.3});

\draw[blue, thick, domain=-2.3:-1.5] plot (\x,{.4*((\x+2.3)*(\x+1.5)*(\x+.5)*(\x-1.5))
/(1+(\x+1)*(\x+1)*(\x+1)*(\x+1)))-3.3});

\draw[blue, thick, domain=-.5:1.5] plot (\x,{.4*((\x+2.3)*(\x+1.5)*(\x+.5)*(\x-1.5))
/(1+(\x+1)*(\x+1)*(\x+1)*(\x+1)))-3.3});

\draw[blue, thick, domain=3:3.3] plot (\x,{.4*((\x+2.3)*(\x+1.5)*(\x+.5)*(\x-1.5))
/(1+(\x+1)*(\x+1)*(\x+1)*(\x+1)))-3.3});

\draw [blue, thick] (-1.5,-3.3) -- (-.5,-3.3);
\draw [blue, thick] (1.5,-3.3) -- (2,-3.3);

\node [below] at (.1,0) {$x_0$};
\node [below] at (1.1,0) {$l$};
\node [below] at (-1.1,0) {$l$};
\node [below] at (2.5,0) {$\eta$};
\node [below] at (-2.5,0) {$\eta$};
\node [left] at (-3,.5) {$\delta$};
\node [left] at (-3,-1.3) {$1$};

\node [above] at (-2.3,2.7) {$q^{u^R}$};
\node [above] at (-2.3,.7) {$p^{u^R}$};
\node [above] at (-2.4,-1.4) {$g$};
\node [above] at (-2.3,-3.3) {$q^{\tilde{u}}$};
\end{tikzpicture}
\end{center}
\caption{The maps $q^{u^R}$, $p^{u^R}$, $q^{\tilde{u}}$ and $g$.}
\label{q}
\end{figure}
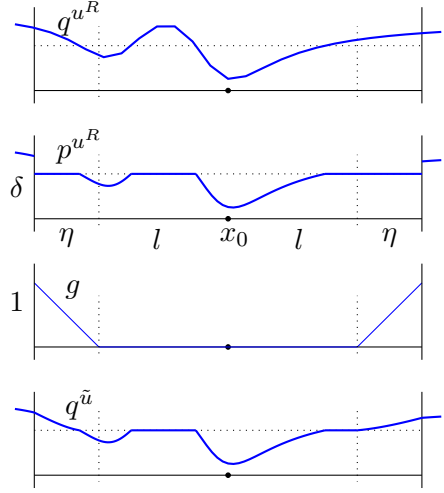

We remark explicitly that in the definition of $\tilde{u}$ we have not changed the direction vector: $n^{\tilde{u}}=n^{u^R}$.

The difficulty now is that, as observed before, we can not assume \eqref{Cond*} and by consequence we can not use Lemma \ref{W-mon} to deduce, as in the proof of Lemma 5.4 \cite{afs}, that $J_{B_l}(\tilde{u})\leq J_{B_l}(u^R)$. To overcome this difficulty we need to treat differently the part of $\tilde{\Sigma}$ where \eqref{Cond*} holds from the rest.
Note that, by definition $\tilde{u}\neq u^R$ only in the subset where $q^{u^R}>\delta$. It follows

\begin{equation}
J_{B_{l+\eta}}(\tilde{u})-J_{B_{l+\eta}}(u^R)
=J_{B_{l+\eta}\cap\{q^{u^R}>\delta\}}(\tilde{u})-J_{B_{l+\eta}\cap\{q^{u^R}>\delta\}}(u^R).
\label{E-diff-in}
\end{equation}
Let $L_\theta=\{x(r,\theta): r\in(0,+\infty)\}$ be the ray through $x(r,\theta)$ and define $\Theta$ by
\[\Theta=\{\theta: L_\theta\cap\{x\in B_l:\min_{\tilde{a}\in A\setminus\{a\}}\vert u^R(x)-\omega\tilde{a}\vert<\delta^*\}\neq\emptyset\}.\]
 Set $l^\prime=(l^2+\vert\Sigma\vert^2)^\frac{1}{2}$  and define
\begin{equation}
U=\cup_{\theta\in\Theta} L_\theta\cap\{x\in B_{l^\prime}:q^{u^R}\geq 2c\delta^\alpha\}.
\label{UU}
\end{equation}
 We divide the set $B_{l+\eta}\cap\{q^{u^R}>\delta\}$ in three parts:
 \[U,\;\;V_1=(B_l\cap\{q^{u^R}>\delta\})\setminus U,\;\;V_2=((B_{l+\eta}\setminus B_l)\cap\{q^{u^R}>\delta\})\setminus U\]
 and estimate separately the difference of energy of $\tilde{u}$ and $u^R$ for the three sets.
  See Figure \ref{klematis} for an illustration of the set $U$.
  \begin{figure}
  \begin{center}
\begin{tikzpicture}[scale=.7]

\draw [] (-.5,.5) to [out=10,in=230] (2.5,2.5);
\node at (.5,1.5) {$q^{u^R}<\delta$};

\draw [] (2.5,-7.5) to [out=40,in=190] (7.5,-5.7);
\node at (5,-6.7) {$q^{u^R}<\delta$};

\path [fill=lightgray] (2,-2) to [out=90,in=90] (4.8,-2)
to [out=270,in=270] (2,-2) ;

\path [fill=gray] (2.2,-2) to [out=90,in=180] (3.35,-1.6)
to [out=0,in=90] (4.5,-2) to [out=270, in=0] (3.35,-2.4) to [out=180, in=270] (2.2,-2);

\path [fill=lightgray] (1.7,-4.5) to [out=90,in=180] (2.51,-4.2)
to [out=0,in=90] (3.5,-4.5) to [out=270, in=0] (2.51,-4.8) to [out=180, in=270] (1.7,-4.5);

\draw [] (0,1) arc [radius=3, start angle=90, end angle=-90];
\draw [] (0,2) arc [radius=4, start angle=90, end angle=-90];

\draw [] (4,2) arc [radius=5.657, start angle=45, end angle=-45];

\draw [blue] (0,-.6) arc [radius=18, start angle=90, end angle=68];
\draw [blue] (0,-2.7) arc [radius=15.9, start angle=90, end angle=68];

\draw [blue] (0,-4.8) arc [radius=13.8, start angle=90, end angle=68];
\draw [blue] (0,-3.5) arc [radius=15.1, start angle=90, end angle=68];

\draw [<-] (2.235,-6.5054) arc [radius=12.3, start angle=73.557, end angle=71.557];
\draw [<-] (1.61,-6.406) arc [radius=12.3, start angle=82.48, end angle=84.48];

\draw [] (2.146,-7) -- (3.811,2);

\draw [] (1.531,-7) -- (2.72,2);

\draw [<-] (5,-5) -- (6,-4.7);

\draw [->] (6.2,-4.2) -- (5.8,-3);

\node at (6.2,-4.5) {$\tilde{\Sigma}$};

\draw [->] (2,-3) -- (2.5,-2.5);

\draw [->] (2,-3.5) -- (2.4,-4.5);

\node at (1,-3.25) {$\{q^{u^R}\geq 2c\delta^\alpha\}$};

\draw [->] (4,-.5) -- (3.5,-2);

\node at (6.3,-.5) {$\{\min_{j=2}^N\vert u^R-\omega^{j-1}a\vert<\delta^*\}$};

\draw [->] (1,-.27) -- (1.5,.598);
\node at (1.25,.4) {$l$};

\draw [->] (1.431,.0475) -- (2.2944,1.276);
\node at (1.9544,1.076) {$l^\prime$};

\draw [->] (3.461,1.1155) -- (4.204,1.784);
\draw [->] (2.75,.4753) -- (2.2296,.007);

\node at (3.1,.9) {$\eta$};

\draw[fill=black] (0,-2) circle [radius=.02];
\node at (-.25,-2.1) {$x_0$};

\draw[fill=red,red] (2.275,-1.5) circle [radius=.03];
\draw[fill=red,red] (2.167,-2.37) circle [radius=.03];

\draw[fill=red,red] (2.92,-2.78) circle [radius=.03];
\draw[fill=red,red] (3.23,-1.2) circle [radius=.03];

\draw[fill=red,red] (1.88,-4.28) circle [radius=.03];
\draw[fill=red,red] (2.64,-4.2) circle [radius=.03];

\draw[fill=red,red] (1.82,-4.7) circle [radius=.03];
\draw[fill=red,red] (2.55,-4.78) circle [radius=.03];


\node at (1.93,-6.4){$\Theta$};

\node at (2.5,-1.7){$1$};
\node at (2.1,-4.5){$2$};

\end{tikzpicture}
\end{center}
\caption{The construction of $U$. $U$ is the union of the two curvilinear quadrilatera $1$ and $2$
marked with red dots.}
\label{klematis}
\end{figure}
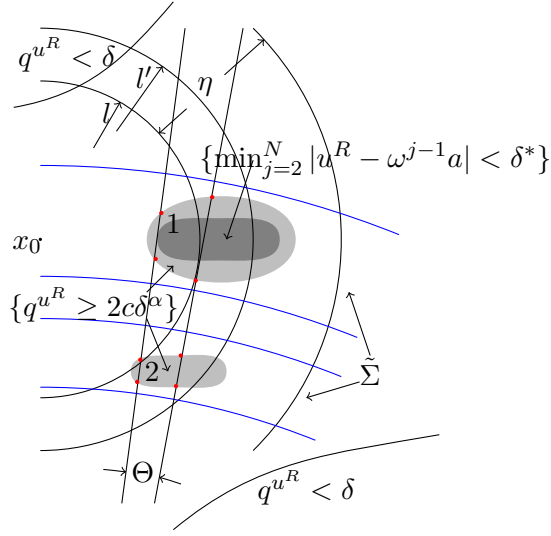

\begin{lemma}
\label{Jtildeu-JuR}
Assume $2\alpha+\alpha^\prime<1$. Then there exists $D>0$ such that, given $\delta\in(0,D^{\frac{1}{1-(2\alpha+\alpha^\prime)}}]$, there are $\eta_\delta>0$ and $\bar{r}=\bar{r}_\delta>0$ independent of $l$ and $R$ which, provided
\[r(x_0)\geq \frac{l}{\sin{\frac{\pi}{N}}}+\bar{r}_\delta,\quad\;l\geq l_0=\frac{\delta}{M},\]
 imply
\begin{equation}
J_U(\tilde{u})-J_U(u^R)\leq 0.
\label{negative}
\end{equation}
\end{lemma}
\begin{proof}
The set $U$ contains all points in $B_l(x_0)$ where condition \eqref{Cond*} is not satisfied and by \eqref{Cond} is  a subset of $\tilde{\Sigma}$.
Moreover the choice of $l^\prime$ ensures that, if $\xi\in L_{\theta(\xi)}\cap B_l(x_0)$ is one of such points, then $L_{\theta(\xi)}\cap B_{l^\prime}(x_0)$ contains a point $\xi^\prime\neq\xi$ that satisfies
\[\vert u^R(\xi^\prime)-a\vert=2c\delta^\alpha,\]
 and is such that the interval with extremes $\xi$ and $\xi^\prime$ is contained in $U$.
This, the smallness of $\delta$ and Lemma \ref{lower-sigma} imply
\begin{equation}
J(u_{\mathrm{restr}},(0,\vert\xi^\prime-\xi\vert))\geq\sigma-\frac{C_W}{2}(4c^2\delta^{2\alpha}+{\delta^*}^2)
\geq\sigma-C_W{\delta^*}^2,
\label{trace-E}
\end{equation}
where $u_{\mathrm{restr}}(s)=u^R(\xi+s\frac{\xi^\prime-\xi}{\vert\xi^\prime-\xi\vert})$, $s\in(0,\vert\xi^\prime-\xi\vert)$, is the restriction of $u^R$ to the interval with extreme at $\xi$ and $\xi^\prime$.
From \eqref{trace-E} we have
\begin{equation}
\begin{split}
&J_U(u^R)\geq\int_{\{x(r,\theta)\in U\}}(\frac{1}{2}\vert\frac{\partial}{\partial r} u^R\vert^2+W(u^R))rdrd\theta\\
&\geq(r(x_0)-l^\prime)\int_\Theta J(u_{\mathrm{restr}},(0,\vert\xi^\prime-\xi\vert))d\theta\geq(r(x_0)-l^\prime)(\sigma-C_W{\delta^*}^2)\vert\Theta\vert.
\end{split}
\label{JuR}
\end{equation}

We now derive an upper bound for $J_U(\tilde{u})$ and show that, if $\delta$ is sufficiently small and if $\eta$ is sufficiently large, then $J_U(\tilde{u})\leq J_U(u^R)$. Since $U\subset\{q^{u^R}>\delta\}$, from \eqref{utilde} we have
\begin{equation}
q^{\tilde{u}}=\left\{\begin{array}{l}\delta,\;\;x\in U\cap B_l,\\
\delta+\frac{\vert x-x_0\vert-l}{\eta}(q^{u^R}-\delta),\;\;x\in U\setminus B_l.
\end{array}\right.
\label{utildeU}
\end{equation}
This \eqref{polar-d}, \eqref{uR-bounds} and $q^{u^R}\geq 2c\delta^\alpha$ on $U$ imply ($E(v)=\frac{\vert\nabla v\vert^2}{2}+W(v)$)
\begin{equation}
\begin{split}
&E(\tilde{u})=\frac{1}{2}\delta^2\vert\nabla n^{u^R}\vert^2+W(\tilde{u})
\leq\frac{1}{2}(\frac{\delta}{q^{u^R}})^2({q^{u^R}})^2\vert\nabla n^{u^R}\vert^2+\frac{1}{2}C_W^2\delta^2\\
&\leq\frac{1}{8c^2}\delta^{2(1-\alpha)}{M}^2+\frac{1}{2}C_W^2\delta^2\leq C_3\delta^{2(1-\alpha)},\;\;x\in U\cap B_l.
\end{split}
\label{bo}
\end{equation}
To estimate $E(\tilde{u})$ on $U\setminus B_l$ we note that \eqref{SigmaSize} and \eqref{l-0} imply
\begin{equation}
l^\prime-l=\frac{\vert\Sigma\vert^2}{l^\prime+l}\leq\frac{C^2}{\delta^{2(1+\alpha^\prime)}(l^\prime+l)}
\leq\frac{C^2}{\delta^{2(1+\alpha^\prime)}2l_0}\leq\frac{C_2^2M}{2\delta^{3+2\alpha^\prime}}=\hat{r}_\delta.
\label{lprime-l}
\end{equation}
From this and \eqref{utildeU} we have
\begin{equation}
\begin{split}
&q^{\tilde{u}}\leq\delta+\frac{l^\prime-l}{\eta}(q^{u^R}-\delta)\leq\delta+\frac{M}{\eta}\hat{r}_\delta,\\
&\vert\nabla q^{\tilde{u}}\vert\leq\frac{1}{\eta}((q^{u^R}-\delta)+\hat{r}_\delta
\vert\nabla q^{u^R}\vert)\leq\frac{M}{\eta}(1+\hat{r}_\delta).
\end{split}
\label{bo-out}
\end{equation}
These inequalities and
\begin{equation}
\eta=\eta_\delta=\frac{M}{\delta}(1+\hat{r}_\delta)
\label{assumeeta}
\end{equation}
imply $q^{\tilde{u}}\leq 2\delta$ and $\vert\nabla q^{\tilde{u}}\vert\leq\delta$ and in turn, proceeding as in the derivation of \eqref{bo}
\begin{equation}
\begin{split}
&E(\tilde{u})\leq\frac{1}{2}(\delta^2+\frac{1}{c^2}\delta^{2(1-\alpha)}M^2+4C_W^2\delta^2)\\
&\leq C_3\delta^{2(1-\alpha)},\;\;x\in U\setminus B_l.
\end{split}
\label{boout}
\end{equation}
From \eqref{bo}, \eqref{boout} and \eqref{SigmaSize} that implies
\[\vert U\vert=\int_{\{x(r,\theta)\in U\}}rdrd\theta\leq(r(x_0)+l^\prime)\vert\Sigma\vert\vert\Theta\vert
\leq(r(x_0)+l^\prime)\frac{C_2}{\delta^{1+\alpha^\prime}}\vert\Theta\vert,\]
il follows
\begin{equation}
J_U(\tilde{u})\leq(r(x_0)+l^\prime)C_4\vert\Theta\vert\delta^{1-(2\alpha+\alpha^\prime)}.
\label{Jutilde}
\end{equation}
From this and \eqref{JuR}we have
\[J_U(\tilde{u})-J_U(u^R)\leq(r(x_0)+l^\prime)C_4\vert\Theta\vert\delta^{1-(2\alpha+\alpha^\prime)}-
(r(x_0)-l^\prime)(\sigma-C_W{\delta^*}^2)\vert\Theta\vert,\]
which implies \eqref{negative} provided

\[\delta^{1-(2\alpha+\alpha^\prime)}\leq\frac{r(x_0)-l^\prime}{r(x_0)+l^\prime}\frac{\sigma-C_W{\delta^*}^2}{C_4}.\]
Note that
\begin{equation}
r(x_0)\geq\frac{l+\eta_\delta}{\sin{\frac{\pi}{N}}}=\frac{l}{\sin{\frac{\pi}{N}}}+\bar{r}_\delta
,\;\;\bar{r}_\delta=\frac{\eta_\delta}{\sin{\frac{\pi}{N}}}.
\label{rx0>}
\end{equation}
Otherwise we have $B_{l+\eta_\delta}(x_0)\cap B_{l+\eta_\delta}(\omega x_0)\neq\emptyset$.

Since $l^\prime<l+\eta_\delta$, \eqref{rx0>}  implies $\frac{r(x_0)-l^\prime}{r(x_0)+l^\prime}\geq\frac{1-\sin{\frac{\pi}{N}}}{1+\sin{\frac{\pi}{N}}}$. Therefore, under that condition, we see that  we can secure \eqref{negative} if we fix $\delta\in(0,\delta_0]$ so that
\begin{equation}
\delta\leq D^{\frac{1}{1-(2\alpha+\alpha^\prime)}},\;\;
D=\frac{\sigma(1-\sin{\frac{\pi}{N}})}{2C_4(1+\sin{\frac{\pi}{N}})}.
\label{fixdelta}
\end{equation}
By increasing the value of $C_4$ if necessary we can also assume that $D^{\frac{1}{1-(2\alpha+\alpha^\prime)}}\leq\delta_0$.
Once $\delta>0$ is fixed, \eqref{lprime-l} implies that $4(l^\prime-l)\leq \bar{r}=\bar{r}_\delta=4\hat{r}_\delta$ for some constant $\bar{r}>0$ independent of $l$ and $R$. The same is true for $\eta=\eta_\delta$ defined in \eqref{assumeeta}. This concludes the proof.
\end{proof}

\begin{remark}
Observe that with $\delta>0$ fixed also the bound for $\vert\Sigma\vert$ given by Lemma \ref{Sigma} is a constant independent of $l$ and $R$.
\end{remark}
To estimate $J_{V_i}(\tilde{u})-J_{V_i}(u^R)$, $i=1,2$ we proceed
  as in the proof of Lemma 5.4 \cite{afs}. The definition \eqref{utilde} of $q^{\tilde{u}}$ implies
 \begin{equation}
 q^{\tilde{u}}=\delta,\quad \nabla q^{\tilde{u}}=0,\;\;\text{ on }\;V_1.
 \label{qtildeinV1}
 \end{equation}
 From this \eqref{polar-d} and Lemma \ref{W-mon} we deduce
 \[\begin{split}
 &J_{V_1}(\tilde{u})-J_{V_1}(u^R)\\
 &=\int_{V_1}\Big(\frac{1}{2}(-\vert\nabla q^{u^R}\vert^2+(\delta^2-(q^{u^R})^2)\vert n^{u^R}\vert^2)+W(a+q^{\tilde{u}}n^{u^R})-W(a+q^{u^R}n^{u^R})\Big)dx\\
 &\leq 0,
 \end{split}\]
 which yields
 \begin{equation}
  J_{V_1}(\tilde{u})-J_{V_1}(u^R)\leq 0.
 \label{V1}
 \end{equation}
 It remains to evaluate the difference of energy on $V_2\subset B_{l+\eta}\setminus B_l$. From  \eqref{utilde} we have
 \[\vert\nabla q^{\tilde{u}}\vert\leq\vert\nabla g\vert(q^{u^R}-\delta)+\vert g\vert\nabla q^{u^R}\vert\leq(\frac{1}{\eta}+1)M,\;\;\text{a.e. on}\;V_2\]

 From this and \eqref{utilde} that implies $q^{\tilde{u}}\leq q^{u^R}$ on $B_{l+\eta}\setminus B_l$ it follows
 \begin{equation}
 \vert\nabla\tilde{u}\vert^2-\vert\nabla u^R\vert^2\leq(\frac{1}{\eta}+1)^2M^2+((q^{\tilde{u}})^2- (q^{u^R})^2)\vert n^{u^R}\vert^2\leq(\frac{1}{\eta}+1)^2M^2,\;\;\text{a.e. on}\;V_2.
 \label{nabla-nabla}
 \end{equation}
 From \eqref{uR-bounds} we have $q^{\tilde{u}}\leq q^{u^R}\leq M$ on $V_2$ and in turn
\[W(a+q^{\tilde{u}}n^{u^R})-W(a+q^{u^R}n^{u^R})\leq W(a+q^{\tilde{u}}n^{u^R})\leq C_M,\;\;\text{a.e. on}\;V_2\]
for some constant $C_M>0$.
This and \eqref{nabla-nabla} imply
\begin{equation}
J_{V_2}(\tilde{u})-J_{V_2}(u^R)\leq(\frac{1}{2}(\frac{1}{\eta}+1)^2M^2+W_M)\vert V_2\vert.
\label{V2}
\end{equation}
From the previous analysis we conclude
\begin{lemma}
\label{5.4}
There exists $\bar{\delta}>0$ such that, given $\delta\in(0,\bar{\delta}]$ there are $\eta=\eta_\delta>0$, $K=K_\delta>0$ and $\bar{r}=\bar{r}_\delta>0$ which, provided
\begin{equation}
\begin{split}
& l\geq l_0+\eta,\quad l_0=\frac{\delta}{M},\\
&r(x_0)\geq 2l+\bar{r},
\end{split}
\label{r>2l}
\end{equation}
imply
\begin{equation}
J_{B_{l+\eta}}(\tilde{u})-J_{B_{l+\eta}}(u^R)\leq K\vert(B_{l+\eta}\setminus B_{l-\eta})\cap\{q^{u^R}>\delta\}\vert,
\label{w}
\end{equation}
\end{lemma}
\begin{proof}
From \eqref{V1}, \eqref{V2}, \eqref{negative} in Lemma \ref{Jtildeu-JuR}, and \eqref{E-diff-in} we obtain
\[J_{B_{l+\eta}}(\tilde{u})-J_{B_{l+\eta}}(u^R)\leq K\vert V_2\vert,\]
where $K_\delta=\frac{1}{2}(\frac{1}{\eta_\delta}+1)^2M^2+W_M$. This and $V_2\subset B_{l+\eta}\setminus B_{l}\subset B_{l+\eta}\setminus B_{l-\eta}$ conclude the proof.
\end{proof}

The estimate \eqref{w} corresponds exactly to the statement of Lemma 5.4 in \cite{afs}. Once  \eqref{w} is established the proof proceeds
exactly as the proof of Theorem 5.3 (pag.165 in \cite{afs}):
set $l_h=l_0+(2h-1)\eta$ for $h=1,\ldots$ and let ${\tilde{u}}_h$ the map ${\tilde{u}}$ given by Lemma \ref{5.4} for $l=l_h$, $h=1,\ldots$. Let $v_h$ the map $v$ given by Lemma \ref{lemma53} with ${\tilde{u}}={\tilde{u}}_h$ and $l=l_h$. Then,
the minimality of $u^R$ implies
\[\begin{split}
&0\geq J_{B_{l_h+\eta}}(u^R)-J_{B_{l_h+\eta}}(v_h)\\
&=J_{B_{l_h+\eta}}(u^R)-J_{B_{l_h+\eta}}({\tilde{u}}_h)+J_{B_{l_h+\eta}}({\tilde{u}}_h)-J_{B_{l_h+\eta}}(v_h)\\
&=J_{B_{l_h+\eta}}(u^R)-J_{B_{l_h+\eta}}({\tilde{u}}_h)+J_{B_{l_h}}({\tilde{u}}_h)-J_{B_{l_h}}(v_h).
\end{split}\]
\noindent
This (\ref{w}) and (\ref{energy+}) yield
\begin{equation}
k\vert B_{l_h-\eta}\cap\{q^{u^R}>\delta\}\vert\leq K\vert B_{l_h+\eta}\setminus B_{l_h-\eta})\cap\{q^{u^R}>\delta\}\vert,\quad\text{ for }\;\;h=1,\ldots
\label{measure-rel}
\end{equation}
\noindent
It is rather clear that, if $\mu_0:=\vert B_{l_0}\cap\{q^{u^R}>\delta\}\vert\neq 0$, this inequality cannot hold for large $h$. Indeed, if we set $\mu_h:=\vert B_{r_h+\eta}\cap\{q^{u^R}>\delta\}\vert$, for $h=1,\ldots$ then (\ref{measure-rel})
yields
\begin{equation}
\frac{k}{K}\mu_{h-1}\leq\mu_h-\mu_{h-1},\quad\text{ for }\;\;h=1,\ldots
\label{iterative}
\end{equation}
and therefore
\[
\mu_0\Big(1+\frac{k}{K}\Big)^h\leq\mu_h,\quad\text{ for }\;\;h=1,\ldots.
\]
From this and (\ref{iterative}) we obtain
\begin{equation}
\mu_0\frac{k}{K}\Big(1+\frac{k}{K}\Big)^{h-1}\leq\mu_h-\mu_{h-1}
,\quad\text{ for }\;\;h=1,\ldots.
\label{ITer11}
\end{equation}

Now assume that
\[q^{u^R}(x_0)=\vert u^R(x_0)-a\vert\geq 2\delta.\]
 Then  \eqref{uR-bounds}implies
\[\delta<q^{u^R},\quad\text{ for }\;\;x\in B_{\frac{\delta}{M}}(x_0).\]
It follows $\mu_0\geq\vert B_{\frac{\delta}{M}}(x_0)\vert$.
Since:
\[\mu_h-\mu_{h-1}\leq\vert B_{l_h+\eta}\setminus\overline{B_{l_h-\eta}}\vert
\leq4\eta l_h=4\eta(l_0+h\eta),\]
the right end side of \eqref{ITer11} grows linearly in $h$.  On the other hand the left hand side grows exponentially in $h$. Hence there exists a minimum value $h_m$ of $h$ such that \eqref{ITer11} is violated for $h\geq h_m$  in contradiction with the minimality of $u^R$. It follows that $l\geq l_{h_m}$ together with the condition $r(x_0)\geq 2l+\bar{r}$ required from Lemma \ref{5.4} imply
\[q^{u^R}(x_0)=\vert u^R(x_0)-a\vert<2\delta.\]
Finally we observe that $l_{h_m}$ is actually a function of $\delta$ this is a consequence of the fact that $\eta$, $\mu_0$ and $h_m$ are all function of $\delta$.
This concludes the proof Theorem \ref{uR<}.
\end{proof}

From now on we assume that $\delta\in(0,\bar{\delta}]$ is fixed and treat $\bar{l}$ and $\bar{r}$ as fixed constants.
\begin{corollary}
\label{cor}
Let $\delta\in(0,\bar{\delta}]$, $\bar{l}$ and $\bar{r}$ be as in Theorem \ref{uR<} and let $u^R$ be a minimizer of \eqref{energy}. Assume  $x_0\in B_R$ and $d>0$ are such that

\begin{equation}
\begin{split}
& B_{\bar{l}+d}(x_0)\subset B_R,\\
& r(x_0)\geq 2\bar{l}+\bar{r}+d,\\
&\vert u^R(x)-a\vert\leq c\delta^\alpha,\quad x\in B_{\bar{l}+d}(x_0)\setminus\tilde{\Sigma}.
\end{split}
\label{cor-cond}
\end{equation}
Then
\begin{equation}
\vert u^R(x_0)-a\vert\leq 2\delta e^{-\bar{k}d},
\label{<exp}
\end{equation}
where $\bar{k}>0$ is independent of $R$.
\end{corollary}
\begin{proof}
 From \eqref{cor-cond} we have that any ball of radius $\bar{l}$ contained in $B_{\bar{l}+d}(x_0)$ satisfies the  assumptions of Theorem \ref{uR<}. It follows
\[\vert u^R(x)-a\vert\leq 2\delta\;\;x\in B_d(x_0).\]
This and a standard argument, see the proof of Lemma 4.4 in \cite{afs},  imply the result.
\end{proof}

Theorem \ref{uR<} is tailored for application to the problem at hand. We state a version of the theorem more appropriate when rectangular coordinates are used. We let $\Omega_R\subset\R^2$ a bounded smooth domain that depends on a parameter $R>0$. For $\nu\in\SF$ we set $\Omega_R\cdot\nu=\{x\cdot\nu:x\in\Omega_R\}$.
\begin{theorem}
\label{r-uR<}
Assume that a minimizer $u^R:\Omega_R\rightarrow\R^2$ of $J_{\Omega_R}$, a ball $B_l(x_0)$, $\nu\in\SF$ and a set
$\Sigma\subset\Omega_R\cdot\nu$ satisfy, for some constants $c>0$ and $C>0$,
\begin{equation}
\vert u^R(x)-a\vert\leq c\delta^\alpha,\;\;x\in B_l(x_0)\setminus\tilde{\Sigma},\;\;\text{for some}\:\alpha\in(0,1)
\label{r-Cond}
\end{equation}
where $\tilde{\Sigma}=\{x\in \Omega_R: x\cdot\nu\in\Sigma\}$, and
\begin{equation}
\vert\Sigma\vert\leq\frac{C}{\delta^{1+\alpha^\prime}},\;\;\text{for some}\;\alpha^\prime\in(0,1).
\label{r-SigmaSize}
\end{equation}

Assume that
\[3\alpha+\alpha^\prime<1.\]
 Then there is a constant $D>0$ such that for each $\delta\leq D^\frac{1}{1-(3\alpha+\alpha^\prime)}$ there exists $l_\delta>0$  independent of $R$ and such that
\begin{equation}
l\geq l_\delta,
\label{l--cond}
\end{equation}
implies
\begin{equation}
\vert u^R(x_0)-a\vert\leq 2\delta.
\label{r-onSigma}
\end{equation}
\end{theorem}

\section{Analysis of the geometric structure of $u^R$.}\label{structure}

From Lemma \ref{Sigma} $r\in[r_\delta,R)\setminus\Sigma$ implies that, modulo a suitable translation, it results $u^R(x(r,\frac{\cdot}{r}))\in\mathscr{V}^*$,  $\mathscr{V}^*$  the set defined by \eqref{V*} in Proposition \ref{away}.  This and equivariance imply in particular the existence of $\theta_r\in[0,2\pi)$
 such that
 \begin{equation}
 \begin{split}
 & u^R(x(r,\theta))\in B_{c\delta^\alpha}(a),\;\;\theta\in(\theta_r^+,\frac{2\pi}{N}+\theta_r^-),\;\, r\in[r_\delta,R)\setminus\Sigma,\\
 & u^R(x(r,\theta))\in B_{c\delta^\alpha}(\omega^{j-1}a),\;\;\theta\in(\theta_r^++\frac{2\pi(j-1)}{N},\frac{2\pi j}{N}+\theta_r^-),\;\,j=2,\ldots,N
 \end{split}
 \label{uR-in-W*}
 \end{equation}
 where
 \[\begin{split}
 &\theta_r^\pm=\theta_r\pm\frac{\nu}{2r},\\
 &\nu=\frac{4\sigma}{c_W^2\delta^2},\,( \text{ see Proposition } \ref{away}).
 \end{split}\]

\subsection{The pseudo regularity of the map $r\rightarrow\theta_r$.}\label{Lip}

We plan to show that the minimality of $u^R$ implies a relationship between the numbers $\theta_{r_*}$ and $\theta_{r}$ corresponding to different radii $r_*$ and $r\in[r_\delta,R)\setminus\Sigma$. Actually we will show that the map $r\rightarrow\theta_r$ has a kind of Lipschitz behavior, see Figure \ref{holderF}. If $\vert\theta_r-\theta_{r_*}\vert\leq\frac{\nu}{\min\{r_*,r\}}$ nothing can be said on the actual value of the difference $\theta_r-\theta_{r_*}$. The situation is different if $\theta_r^--\theta_{r_*}^+>0$ or $\theta_{r_*}^--\theta_r^+>0$. We have indeed

\begin{lemma}
\label{rrprime}
There are $\hat{c}>0$, $\beta\in(0,1)$ and $\bar{r}>r_\delta$ such that, if $R>\bar{r}$, $r_*\in[\bar{r},R)$ and
$r^*=\min\{(1+\beta)r_*,R\}$, then it results
\begin{equation}
\begin{split}
&\theta_{r_*}^++\frac{\nu}{r}+\hat{c}\ln{\frac{r}{r_*}}\geq\theta_r^+>\theta_r^-\geq
\theta_{r_*}^--\frac{\nu}{r}-\hat{c}\ln{\frac{r}{r_*}},\;\;r\in[r_*,r^*)\setminus\Sigma,\\\\
&\theta_{r_*}^++\frac{\nu}{r}+\hat{c}\ln{\frac{r_*}{r}}\geq\theta_r^+>\theta_r^-\geq
\theta_{r_*}^--\frac{\nu}{r}-\hat{c}\ln{\frac{r_*}{r}},\;\;r\in(r_*(1-\beta),r_*]\setminus\Sigma.
\end{split}
\label{srprime-sr}
\end{equation}
and
\begin{equation}
\theta_r^+<\theta_r^-+\frac{2\pi}{N}.
\label{no-cap}
\end{equation}
Moreover if $\theta_r^\pm$ does not satisfies \eqref{srprime-sr} then $r\not\in(r_*(1-\beta),r^*)$.
\end{lemma}
\begin{proof}
1. Let $v:[r_*,r]\rightarrow\R^2$, $r_*<r$, be a smooth function. Then it results
\begin{equation}
\int_{r_*}^{r}\vert v^\prime\vert^2\rho d\rho\geq\frac{\vert v(r)-v(r_*)\vert^2}{\ln{\frac{r}{r_*}}}.
\label{D_r-Kin}
\end{equation}
\noindent
2. Consider first the case $\theta_{r}^--\theta_{r_*}^+\in(0,\frac{2\pi}{N}-\bar{\theta}]$, where $\bar{\theta}>0$ is fixed and small. In this case we have
\begin{equation}
\left.
\begin{array}{l}
u^R(x(r_*,\theta_{r_*}^++\phi))\in B_{c\delta^\alpha}(a),\\
u^R(x(r,\theta_{r_*}^++\phi))\in B_{c\delta^\alpha}(\omega^{N-1}a),
\end{array}
\right.\;\text{for}\;\phi\in(0,\theta_{r}^--\theta_{r_*}^+).
\label{BC-r}
\end{equation}
It follows
\[\vert u^R(x(r,\theta_{r_*}^++\phi))-u^R(x(r_*,\theta_{r_*}^++\phi))\vert\geq\vert(\omega^{N-1}-I)a\vert-2c\delta^\alpha,
\;\text{for}\;\phi\in(0,\theta_{r}^--\theta_{r_*}^+).\]
This and \eqref{D_r-Kin} imply
\[\int_{\theta_{r_*}^+}^{\theta_{r}^-}\int_{r_*}^{r}\vert\frac{\partial}{\partial\rho}u^R\vert^2\rho d\rho\theta
\geq\frac{(\vert(\omega^{N-1}-I)a\vert-2c\delta^\alpha)^2}{\ln{\frac{r}{r_*}}}(\theta_{r}^--\theta_{r_*}^+),\]
which together with the bound \eqref{du-dr-bound} yields
\begin{equation}
\theta_{r}^--\theta_{r_*}^+\leq\hat{c}\ln{\frac{r}{r_*}}\;\Leftrightarrow\;
\theta_{r}^+-\theta_{r_*}^+\leq\frac{\nu}{r}+\hat{c}\ln{\frac{r}{r_*}},
\label{te+-te+}
\end{equation}
where we have set $\hat{c}=\frac{2(C_0+C_1)}{N(\vert(\omega^{N-1}-I)a\vert-2c\delta^\alpha)^2}$. In case $\theta_{r_*}^--\theta_r^+\in(0,\frac{2\pi}{N}-\bar{\theta}]$, the same argument developed above and $\vert(\omega^{N-1}-I)a\vert=\vert\omega-I)a\vert$ leads to
\begin{equation}
\theta_{r_*}^--\theta_{r}^+\leq\hat{c}\ln{\frac{r}{r_*}}\;\Leftrightarrow\;
\theta_{r}^--\theta_{r_*}^-\geq-\frac{\nu}{r}-\hat{c}\ln{\frac{r}{r_*}}.
\label{te--te-}
\end{equation}
Equation \eqref{srprime-sr}$_1$  follows from  \eqref{te+-te+} and \eqref{te--te-}. To prove  \eqref{srprime-sr}$_2$ we observe that proceeding as before we see that $\theta_{r}^--\theta_{r_*}^+\in(0,\frac{2\pi}{N}-\bar{\theta}]$ and $r<r_*$ imply
\[\theta_{r}^--\theta_{r_*}^+\leq\hat{c}\ln{\frac{r_*}{r}}.\]
This and the corresponding statement valid for $\theta_{r_*}^--\theta_r^+\in(0,\frac{2\pi}{N}-\bar{\theta}]$ prove \eqref{srprime-sr}$_2$.

3. From \eqref{srprime-sr}, if $r>r_*$ and $r_*$ is sufficiently large,
\[2\hat{c}\ln{\frac{r}{r_*}}\leq 2\hat{c}\ln{(1+\beta)}\leq\frac{\pi}{N},\]
 is a sufficient condition for \eqref{no-cap}.
From this and the analogous condition for the case $r_*>r$ it follows that $\beta\leq 1-e^{-\frac{\pi}{2\hat{c}N}}$ ensures that \eqref{no-cap} holds.
\vskip.2cm
4. If $\theta_r^--\theta_{r_*}^+>\frac{2\pi}{N}-\bar{\theta}$ and  $r>r_*$ the same reasoning that proves \eqref{te+-te+} leads to
\[\ln{\frac{r}{r_*}}\geq C_5,\]
for some $C_5>0$. Analogous conclusion holds in the cases $\theta_{r_*}^--\theta_r^+>\frac{2\pi}{N}-\bar{\theta}$
 etc. The last statement of the lemma follows from this after a reduction of the value of $\beta$ if necessary. The proof is complete.
\end{proof}

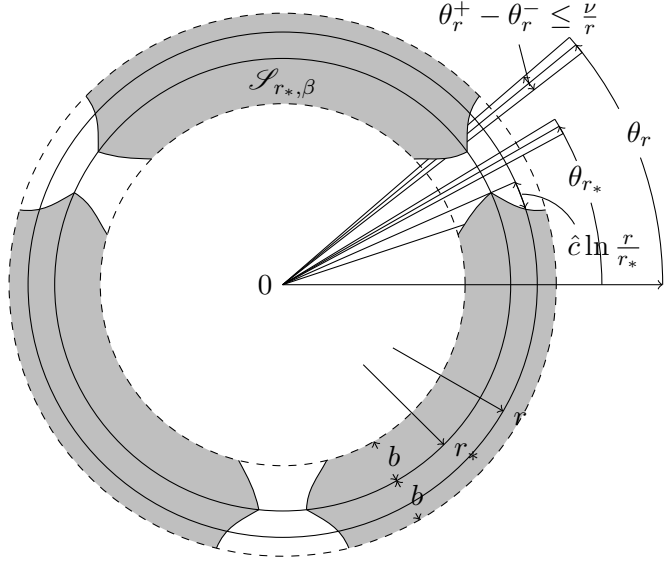
\begin{figure}
  \begin{center}
\begin{tikzpicture}[xscale=1]
\draw (0,0)--(3.637,2.1);
\draw [thin] (0,0)--(3.691,2.004);
\draw [thin] (3.691,2.004) arc [radius=4.2, start angle=28.5, end angle=31.5];
\draw [thin] (3.581,2.194)--(0,0);
\draw (0,0)--(3.858,3.18);
\draw [thin] (0,0)--(3.94,3.078);
\draw [thin] (3.94,3.078) arc [radius=5, start angle=38, end angle=41];
\draw [thin] (3.773,3.28)--(0,0);
\draw [thin] (0,0) -- (3.043,1.015);
\draw [<->] (3.31,2.585) arc [radius=4.2, start angle=38, end angle=41];

\draw [] (4.2,0) arc [radius=4.2, start angle=0, end angle=15];
\draw [->] (3.866,1.641) arc [radius=4.2, start angle=23, end angle=30];
\draw [] (5,0) arc [radius=5, start angle=0, end angle=19.7];
\draw [->] (4.423,2.332) arc [radius=5, start angle=27.8, end angle=39.5];

\path [fill=lightgray] (-2.741,1.220)  to [out=210,in=352] (-3.460,.992)
arc [radius=3.6, start angle=164, end angle=256] to
 [out=68,in=210] (-.3136,-2.983) to
 [out=102,in=298] (-.58,-2.329)
arc [radius=2.4, start angle=256, end angle=164] to
 [out=122,in=328] (-2.741,1.220);
\path [fill=lightgray] (.3136,-2.983) to [out=330,in=112] (.871,-3.493)
arc [radius=3.6, start angle=-76, end angle=16] to
 [out=188,in=330] (2.741,1.220) to
 [out=222,in=58] (2.307,.661)
arc [radius=2.4, start angle=16, end angle=-76] to
 [out=242,in=88] (.3136,-2.983);
\path [fill=lightgray] (2.427,1.763) to [out=90,in=232] (2.589,2.501)
arc [radius=3.6, start angle=44, end angle=136] to
 [out=308,in=90] (-2.427,1.763) to
 [out=342,in=178] (-1.726,1.667)
arc [radius=2.4, start angle=136, end angle=44] to
 [out=2,in=208] (2.427,1.763);


\draw [dashed](0,0) circle [radius=2.4];
\draw (0,0) circle [radius=3];
\draw [dashed](0,0) circle [radius=3.6];
\draw [thin] (0,0) circle [radius=3.35];
\draw [->] (0,0)--(5,0);


\draw[] (.3136,-2.983) to [out=88,in=242] (.58,-2.329);
\draw[] (.3136,-2.983) to [out=330,in=112] (.871,-3.493);
\draw[] (-.3136,-2.983) to [out=102,in=298] (-.58,-2.329);
\draw[] (-.3136,-2.983) to [out=210,in=68] (-.871,-3.493);
\draw[] (2.427,1.763) to [out=208,in=2] (1.726,1.667);
\draw[] (2.427,1.763) to [out=90,in=232] (2.589,2.501);
\draw[] (2.741,1.220) to [out=222,in=58] (2.307,.661);
\draw[] (2.741,1.220) to [out=330,in=188] (3.460,.992);
\draw[] (-2.427,1.763) to [out=342,in=178] (-1.726,1.667);
\draw[] (-2.427,1.763) to [out=90,in=308] (-2.589,2.501);
\draw[] (-2.741,1.220) to [out=328,in=122] (-2.307,.661);
\draw[] (-2.741,1.220) to [out=210,in=352] (-3.460,.992);
\draw[->] (1.06,-1.06) -- (2.121,-2.121);
\draw[->] (1.45,-.8375) -- (2.901,-1.675);
\draw[->] (1.25,-2.174) -- (1.2,-2.078);
\draw[->] (1.45,-2.512) -- (1.5,-2.598);
\draw[->] (1.55,-2.684) -- (1.5,-2.598);
\draw[->] (1.75,-3.03) -- (1.8,-3.117);
\node [left] at (0,0) {$0$};
\node [right] at (2.121,-2.221) {$r_*$};
\node [right] at (2.901,-1.775) {$r$};
\node [right] at (1.25,-2.274) {$b$};
\node [right] at (1.55,-2.784) {$b$};
\node [] at (4,1.4) {$\theta_{r_*}$};
\node [] at (4.7,2) {$\theta_r$};
\draw [<->] (3.06,1.362) arc [radius=3.35, start angle=24, end angle=17.63];
\draw [] (0,0) -- (3.06,1.362);
\draw [thin] (3.25,2.635)--(3.11,3.185);
\node [above] at (3.11,3.185) {$\theta_r^+-\theta_r^-\leq\frac{\nu}{r}$};
\draw [thin] (3.84,.8) arc [radius=.8, start angle=30, end angle=90];
\node [right, below] at (4.24,.8) {$\hat{c}\ln{\frac{r}{r_*}}$};
\node [below] at (0,3.) {$\mathscr{S}_{r_*,\beta}$};
\end{tikzpicture}
\end{center}
\caption{The Lipschitz like property of the map $r\rightarrow\theta_r$ and the set $\mathscr{S}_{r_*,\beta}$\;(N=3).}
\label{holderF}
\end{figure}
\begin{corollary}
\label{esseScr}
Let $r_*\in(\bar{r},R)$, $r^*$ and $\beta\in(0,1)$ be as in Lemma \ref{rrprime}. Then, assuming that $\bar{r}$ has been chosen sufficiently large, the sets
\begin{equation}
\begin{split}
&\check{\mathscr{S}}_{r_*,\beta}=\{x(r,\theta):\\
&\frac{2\pi}{N}-\frac{3\nu}{2r_*}-\hat{c}\ln{\frac{r}{r_*}})>\theta-\theta_{r_*}
>\frac{3\nu}{2r_*}+\hat{c}\ln{\frac{r}{r_*}},
\;\;r\in[r_*,r^*)\}\\\\
&\hat{\mathscr{S}}_{r_*,\beta}=\{x(r,\theta):\\
&\frac{2\pi}{N}-\frac{\nu}{2r_*}-\frac{\nu}{r}-\hat{c}\ln{\frac{r_*}{r}}>\theta-\theta_{r_*}
>\frac{\nu}{2r_*}+\frac{\nu}{r}+\hat{c}\ln{\frac{r_*}{r}},
\;\;r\in(r_*(1-\beta),r_*]\}.
\end{split}
\label{ArcNear-1}
\end{equation}
are well defined, nonempty and
 it results
\begin{equation}
\vert u^R(x)-\omega^{j-1} a\vert\leq c\delta^\alpha\;\;,\text{for}\;x\in\omega^{j-1} \mathscr{S}_{r_*,\beta}\setminus\tilde{\Sigma},\;\;j=1,\ldots,N
\label{near-om a}
\end{equation}
where
\[\mathscr{S}_{r_*,\beta}=\hat{\mathscr{S}}_{r_*,\beta}\cup\check{\mathscr{S}}_{r_*,\beta}.\]

\end{corollary}
\begin{proof}
The assumption on  $\beta$ implies that, provided $r_*\geq\bar{r}$ is sufficiently large, $\theta_r^+\leq\theta_r^-+\frac{\pi}{N}$ for $r\in(r_*(1-\beta),r^*)$. This means that the set $\hat{\mathscr{S}}_{r_*,\beta}$ and $\check{\mathscr{S}}_{r_*,\beta}$ are well defined and nonempty. From Lemma \ref{Sigma} and Proposition \ref{away}, $r\in(r_*(1-\beta),r^*)\setminus\Sigma$ implies $u^R(x(r,\frac{\cdot}{r}))\in\mathscr{V}_r^*$ and the inequalities \eqref{srprime-sr} in Lemma \ref{rrprime} and equivariance yield \eqref{near-om a}. The proof is complete.
\end{proof}
 The set $\mathscr{S}_{r_*,\beta}$ is illustrated in Figure \ref{holderF}. Denote $C_r$ the circumference of radius $r\in(0,R)$ and set
\begin{equation}
\hat{p}_{r_*}^\pm=x(r_*,\theta_{r_*}\pm\frac{3\nu}{2 r_*}).
 \label{hatp}
 \end{equation}
 Observe that $\hat{p}_{r_*}^+$ and $\omega\hat{p}_{r_*}^-$ are the extreme of the arc $\hat{\mathscr{S}}_{r_*,\beta}\cap\check{\mathscr{S}}_{r_*,\beta}$.

 Next we introduce a representation of $\partial\mathscr{S}_{r_*,\beta}$. We focus on the
 connected component of $\partial\mathscr{S}_{r_*,\beta}$ that contains $\hat{p}_{r_*}^+$. Analogous definitions
 apply to the component that contains $\omega\hat{p}_{r_*}^-$.

 Set
 \begin{equation}
 \begin{split}
 &\check{g}_{r_*}(r)=x(r,\check{\vartheta}_{r_*}(r)),\;\;\;\check{\vartheta}_{r_*}(r)
 =\theta_{r_*}+\frac{3\nu}{2r_*}+\hat{c}\ln{\frac{r}{r_*}},\;\;r\in[r_*,r^*),\\
 &\hat{g}_{r_*}(r)=x(r,\hat{\vartheta}_{r_*}(r)),\;\;\;\hat{\vartheta}_{r_*}(r)
 =\theta_{r_*}+\frac{\nu}{2r_*}+\frac{\nu}{r}+\hat{c}\ln{\frac{r_*}{r}},\;\;r\in((1-\beta)r_*,r^*].
 \end{split}
 \label{partial S}
 \end{equation}

 Note that $\check{g}_{r_*}(r_*)=\hat{g}_{r_*}(r_*)=\hat{p}_{r_*}^+$ and that $\check{g}_{r_*}([r_*,r^*))$ coincides with the connected component of $\check{\mathscr{S}}_{r_*,\beta}\cap\mathscr{S}_{r_*,\beta}$ that contains $\hat{p}_{r_*}^+$. Similarly $\hat{g}_{r_*}((1-\beta)r_*,r_*])$ coincides with the connected component of $\hat{\mathscr{S}}_{r_*,\beta}\cap\mathscr{S}_{r_*,\beta}$ that contains $\hat{p}_{r_*}^+$.
\begin{definition}\label{def}
  Let $\ell_{r_*}^+$ be the half line which has origin at $\hat{p}_{r_*}^+$, forms equal angles with the two tangents to $\partial\hat{\mathscr{S}}_{r_*,\beta}$ at $\hat{p}_{r_*}^+$ and points inside $\hat{\mathscr{S}}_{r_*,\beta}$. We let $\ell_{r_*}^-$ the analogous half line which has origin in $\hat{p}_{r_*}^-$ and points inside $\omega^{-1}\hat{\mathscr{S}}_{r_*,\beta}$.
\end{definition}

 For $x,y\in\R^2$ we let $[x,y]$ be the segment with extreme at $x,y$.
 If $x$ and $y$ satisfy $r(x)=r(y)$, that is: $x$ and $y$ belong to the same circumference, we let $\mathrm{arc}(x,y)$ denote the shortest arc of $C_r$ with extreme $x$ and $y$ and $d^\circ(x,y)$ be the length of
$\mathrm{arc}(x,y)$. If $S\subset\R^2$ and $S\cap C_{r(x)}\neq\emptyset$ we set
\begin{equation}
d^\circ(x,S)=\inf_{y\in S\cap C_{r(x)}}d^\circ(x,y).
\label{d-zero}
\end{equation}

\begin{lemma}
\label{ball-exists}
Let $\bar{r}$, $r_*\geq\bar{r}$ and $\beta$ as in Corollary \ref{esseScr}. Then there exist $\epsilon\in(0,1)$, $\eta\in(0,1)$ such that, if $\bar{r}$ is sufficiently large,
\begin{enumerate}
\item
\[\begin{split}
&x\in\check{\mathscr{S}}_{r_*,\beta},\;\;\;l\in(0,\epsilon r_*]\;\;\;\text{and}\;\;\;d(x,\partial\mathscr{S}_{r_*,\beta})\geq l,\\\\
&\Rightarrow\;\;B_{\eta l}(x)\subset\mathscr{S}_{r_*,\beta}.
\end{split}\]
The same conclusion is valid if the condition $d(x,\partial\mathscr{S}_{r_*,\beta})\geq l$ is replaced by
\[r^*-r(x)\geq l,\;\;\text{and}\;\;d^\circ(x,\partial\mathscr{S}_{r_*,\beta})\geq l,\]
where, as before, $r^*=\min\{(1+\beta)r_*,R\}$.
\item
\[\begin{split}
&\rho\in(0,\epsilon r_*]\;\;\;\Rightarrow\\\\
&B_{\frac{\rho}{2}}(\tilde{p}^+(\rho))\subset\hat{\mathscr{S}}_{r_*,\beta},\\
&B_{\frac{\rho}{2}}(\tilde{p}^-(\rho))\subset\omega^{-1}\hat{\mathscr{S}}_{r_*,\beta}.
\end{split}\]
where $\tilde{p}^\pm(\rho)\in\ell_{r_*}^\pm\cap C_{r_*-\rho}$  satisfies $\lim_{\rho\rightarrow 0}\tilde{p}^\pm(\rho)=\hat{p}_{r_*}^\pm$.
\end{enumerate}
\end{lemma}
\begin{proof}
To prove (ii) we observe that, since,  by definition, $\ell_{r_*}^+$ forms equal angles with the two tangents to
$\partial\hat{\mathscr{S}}_{r_*,\beta}$ at $\hat{p}_{r_*}^+$, we have approximately
\[d(\tilde{p}^+(\rho),C_{r_*})\simeq d(\tilde{p}^+(\rho),\partial\hat{\mathscr{S}}_{r_*,\beta}\setminus C_{r_*})\simeq\rho.\]
This and the analogous statement for $\tilde{p}^-(\rho)$, under the standing assumption that $\epsilon>0$ is small, imply (ii). The proof is complete.
\end{proof}

\begin{remark}
\label{remark}
Let $\varphi_*$ the angle between the two lines tangent to $\hat{\mathscr{S}}_{r_*,\beta}$ at  $\hat{p}_{r_*}^+$. Then
\begin{equation}
\sin{\varphi_*}=\frac{1}{\sqrt{{1+(\hat{c}+\frac{\nu}{r_*})^2}}}.
\label{phi*}
\end{equation}
Fix a small number $\epsilon>0$ and, for $r\in((1-\epsilon)r_*,r_*]$ set $\vartheta(r)=\theta-\theta_{r_*}=\frac{\nu}{2r_*}+\frac{\nu}{r}+\hat{c}\ln{\frac{r_*}{r}}$. Let $0\xi_1\xi_2$ the positive frame determined by the assumption that the $\xi_1$ axis coincides with the ray $L_{\theta_{r_*}}$. Then the definition of $\hat{\mathscr{S}}_{r_*,\beta}$ in Corollary \ref{esseScr} implies that
$(\xi_1(r),\xi_2(r))=(r\cos{\vartheta},r\sin{\vartheta})$, $r\in((1-\epsilon)r_*,r_*]$ is a local representation of
$\partial\hat{\mathscr{S}}_{r_*,\beta}\setminus C_{r_*}$ and $(\xi_1(r_*),\xi_2(r_*))=\hat{p}_{r_*}^+$. Let $\varphi(r)$ the angle that the vector $(\xi_1^\prime(r),\xi_2^\prime(r))$ tangent to $\hat{\mathscr{S}}_{r_*,\beta}$ at $(\xi_1(r),\xi_2(r))$ forms with the vector $(\sin{\vartheta},-\cos{\vartheta})$ tangent to $C_r$ at $(\xi_1(r),\xi_2(r))$. It results
\[\sin{\varphi(r)}=\frac{1}{\sqrt{{1+(r\vartheta^\prime(r))^2}}}.\]
This, $\vartheta^\prime(r)=-\frac{\hat{c}}{r}-\frac{\nu}{r^2}$ and  $r=r_*$ imply \eqref{phi*}.
\end{remark}
\subsection{Construction of the diffuse interface}\label{diff-int}
The properties of the map $r\rightarrow\theta_r$ discussed above allow for the construction of  a set that can be regarded as the diffuse interface that divide $B_R$ in regions where $u^R$ is near to one or another of the zeros of $W$. We will be able to define the length of the interface and by means of Theorem \ref{uR<} show that the energy of $u^R$ is essentially contained in the interface and proportional to its length. This and the upper bound \eqref{UB} for the energy of $u^R$ allow to control the length and by consequence the shape of the interface leading to the proof of Theorem \ref{Th2}.

We now consider a sequences of balls $B_{R_n}$ defined by
\[R_n=(n+1+c_1)^2,\]
where $c_1>0$ is a constant and let $u^{R_n}$ a $C_N$-equivariant minimizer of $J_{B_{R_n}}$.
Let $\tilde{c}>0$ and $\mu_j\in[0,\vert\Sigma\vert]$, $j=1,\ldots,n+1$ be constants to be chosen later and define
\begin{equation}
\left.\begin{array}{l}
r_j=(j+c_1)^2-\mu_j,\\\\
\lambda_j=\tilde{c}\ln{r_j},
\end{array}\right.
j=1,\ldots,n+1,
\label{lambda rj}
\end{equation}
where $c_1$ is chosen sufficiently large to ensure
\[r_1\geq\max\{r_\delta,2\bar{l}+\bar{r}\},\]
$r_\delta$ as in Lemma \ref{Sigma} and  $\bar{l}$ and $\bar{r}$ as in Theorem \ref{uR<}. From Lemma \ref{Sigma} we can assume that, for each $j$, the constant $\mu_j\in[0,\vert\Sigma\vert]$ is such that
\[r_j\in(0,R_n)\setminus\Sigma,\;\;j=1,\ldots,n+1.\]

\begin{lemma}\label{1:2}
The sequences $\{r_j\}_{j=1}^{n+1}$ and $\{\lambda_j\}_{j=1}^{n+1}$ satisfy
\begin{enumerate}
\item
\begin{equation}
\begin{split}
&\lim_{j\rightarrow+\infty}r_{j+1}-r_j=+\infty,\quad\quad\quad\quad\\\\
&\lim_{j\rightarrow+\infty}\frac{r_{j+1}-r_j}{r_j}=0,\quad\quad\quad\quad\\\\
&\lim_{r\rightarrow+\infty}\frac{\lambda_j}{r_{j+1}-r_j}=0.\quad\quad\quad\quad
\end{split}
\label{sequence}
\end{equation}
\item
If $c_1>0$ is sufficiently large there exist constants $c_0,C^0>0$ such that
\begin{equation}
 c_0r_j^\frac{1}{2}\leq r_{j+1}-r_j\leq C^0r_j^\frac{1}{2},\;\;j=1,\ldots,n+1,
 \label{.5}
 \end{equation}
 and we can assume
\begin{equation}
\left.\begin{array}{l}
\beta_j=\frac{r_{j+1}-r_j}{r_j}\leq\frac{\beta}{2},\\
\frac{2\lambda_j}{r_{j+1}-r_j}<1.
\end{array}\right.
j=1,\ldots,n+1,
\label{r-r<beta}
\end{equation}
where $\beta$ is as in Lemma \ref{rrprime}.
 \end{enumerate}
\end{lemma}
\begin{proof}
The proof is a an elementary computation.
\end{proof}

\begin{figure}
  \begin{center}
\begin{tikzpicture}[scale=.75]

\path [fill=lightgray] (0,0) arc [radius=10, start angle=90, end angle= 82]
to (4.788,3.155)
arc [radius=14, start angle=70, end angle=110]
to (-1.392,-.0973)
arc [radius=10, start angle=98, end angle=90];

\draw [thin] (0,0) arc [radius=10, start angle=90, end angle= 65];
\draw [thin] (0,4) arc [radius=14, start angle=90, end angle= 65];
\draw [thin] (0,0) arc [radius=10, start angle=90, end angle= 115];
\draw [thin] (0,4) arc [radius=14, start angle=90, end angle= 115];

\draw[] (1.392,-.0973)--(4.788,3.155);
\draw[] (-1.392,-.0973)--(-4.788,3.155);

\draw[->] (.523,-.0137) to [out=67,in=212] (3.149,3.641);
\draw[->] (-.523,-.0137) to [out=113,in=328] (-3.149,3.641);

\draw[thin] (0,-1)--(0,5.5);
\draw[dashed] (.523,-.0137)--(.7327,3.981);
\draw[dashed] (2.249,-.2563)--(3.149,3.641);
\draw[dashed] (2.588,-.341)--(3.623,3.523);
\draw[dashed] (3.42,-.603)--(4.788,3.155);

\draw[thin] (0,0)--(0,4);
\draw[dashed] (-.523,-.0137)--(-.7327,3.981);
\draw[dashed] (-2.249,-.2563)--(-3.149,3.641);
\draw[dashed] (-2.588,-.341)--(-3.623,3.523);
\draw[dashed] (-3.42,-.603)--(-4.788,3.155);

\draw[dashed] (1.392,-.0973)--(1.531,.893);
\draw[dashed] (-1.392,-.0973)--(-1.531,.893);


\path[fill=red] (.523,-.0137) circle [radius=.045];
\path[fill=red] (1.392,-.0973) circle [radius=.045];

\path[fill=red] (-.523,-.0137) circle [radius=.045];
\path[fill=red] (-1.392,-.0973) circle [radius=.045];

\path[fill=red] (.7327,3.981) circle [radius=.045];
\path[fill=red] (3.149,3.641) circle [radius=.045];
\path[fill=red] (3.623,3.523) circle [radius=.045];
\path[fill=red] (4.788,3.155) circle [radius=.045];

\path[fill=red] (-.7327,3.981) circle [radius=.045];
\path[fill=red]  (-3.149,3.641) circle [radius=.045];
\path[fill=red] (-3.623,3.523)  circle [radius=.045];
\path[fill=red] (-4.788,3.155) circle [radius=.045];


\node[right] at (0,5.5) {$L_{\theta_{r_j}}$};
\node[right] at (4.2,-1.2) {$r_j$};
\node[right] at (5.8,2.5) {$r_{j+1}$};
\node[left] at (-4.788,3.155) {$q_{j+1}^+$};
\node[right] at (4.788,3.155) {$q_{j+1}^-$};
\node[left] at (-4.788,3.155) {$q_{j+1}^+$};
\node[right] at (4.788,3.155) {$q_{j+1}^-$};
\node[right] at (1.422,-.1) {$p_j^-$};
\node[left] at (-1.422,-.1) {$p_j^+$};
\node[below] at (.523,.1537) {$\hat{p}_j^-$};
\node[below] at (-.523,.1537) {$\hat{p}_j^+$};
\node[above]at (-1.3,2) {$\mathscr{I}_j$};

\draw[->] (.5,4.2)--(.35,4.);
\node[above] at (.6,4.) {$\frac{3\nu}{2r_j}$};
\draw[->] (2.,4.1)--(2.,3.85);
\node[above] at (2.,3.9) {$\hat{c}\ln{\frac{r_{j+1}}{r_j}}$};
\draw[->] (3.4,3.8)--(3.4,3.6);
\node[above] at (3.4,3.7) {$\frac{\nu}{r_{j+1}}$};
\draw[->] (4.3,3.6)--(4.2,3.4);
\node[right] at (4.2,3.7) {$\frac{\lambda_{j+1}}{r_{j+1}}$};
\draw[->] (1.,-.35)--(1.,-.1);
\node[below] at (1.27,.0) {$\frac{\lambda_j}{r_j}$};
\end{tikzpicture}
\end{center}
\caption{The points $\hat{p}_j^\pm, p_j^\pm, q_{j+1}^\pm$ and the set $\mathscr{I}_j$.}
\label{Ij}
\end{figure}
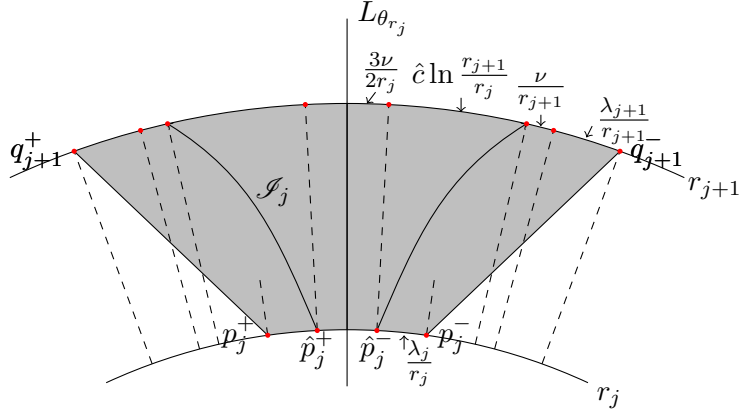

Define (see Figure \ref{Ij}):
\begin{equation}
\left.\begin{array}{l}
p_j^\pm=x(r_j,\theta_{r_j}\pm\frac{3\nu}{2r_j}\pm\frac{\lambda_j}{r_j}),\\\\
q_j^\pm=x(r_j,\theta_{r_{j-1}}\pm\frac{3\nu}{2r_{j-1}}
\pm\frac{\nu}{r_j}\pm\hat{c}\ln{\frac{r_j}{r_{j-1}}}\pm\frac{\lambda_j}{r_j}),
\end{array}\right.
\;\;j=1,\ldots,n+1,
\label{PeQ}
\end{equation}
and set $\hat{p}_j^\pm=\hat{p}_{r_j}^\pm$ where $\hat{p}_{r_j}^\pm$ is defined in \eqref{hatp}.
We have
\begin{equation}
\begin{split}
&\theta(p_j^+)-\theta(\hat{p}_j^+)=\frac{\lambda_j}{r_j},\\
&\theta(q_j^+)-\theta(p_j^+)\geq\frac{\nu}{r_{j-1}}-\frac{\nu}{r_j}>0.
\end{split}
\label{pq-in}
\end{equation}
The first equation is obvious, \eqref{pq-in}$_2$ follows from the definition \eqref{PeQ} and \eqref{srprime-sr}$_1$ in Lemma \ref{rrprime} with $r_*=r_{j-1}$ and $r=r_j$ which implies
\[\theta_{r_j}+\frac{\nu}{2r_j}
\leq\theta_{r_{j-1}}+\frac{\nu}{2r_{j-1}}+\frac{\nu}{r_j}+\hat{c}\ln{\frac{r_j}{r_{j-1}}}.\]

\begin{lemma}
\label{balls-}
Set $a_-=\omega^{-1}a$ and $a_+=a$. Then
\begin{enumerate}
\item
There is a constant $\eta\in(0,1)$ independent of $n$ such that
\begin{equation}
\left.
\begin{array}{l}
x\in\mathrm{arc}[p_j^\pm,q_j^\pm]\;\;\Rightarrow\\\\
\vert u^{R_n}-a_\pm\vert\leq c\delta^\alpha,\;\;\text{on}\;\;
B_{\eta\lambda_j}(x)\setminus\tilde{\Sigma},
\end{array}
\right.
j=2,\ldots,n,
\label{arc-in}
\end{equation}
\item
\begin{equation}
\left.
\begin{array}{l}
x\in[p_j^\pm,q_{j+1}^\pm]\;\;\;\Rightarrow\\\\
\vert u^{R_n}-a_\pm\vert\leq c\delta^\alpha,\;\;\text{on}\;\;
B_{\eta\lambda_j}(x)\setminus\tilde{\Sigma},
\end{array}
\right.
j=1,\ldots,n-1.
\label{segm-in}
\end{equation}
\end{enumerate}
\end{lemma}
\begin{proof}
From \eqref{pq-in} and \eqref{d-zero} it follows
\begin{equation}
d^\circ(x,\partial\mathscr{S}_{r_j,\beta})\geq\lambda_j,\;\;x\in\mathrm{arc}(p_j^+,q_j^+),\;j=2,\ldots,n.
\label{pq-in1}
\end{equation}
Set $r_j^*=\min\{(1+\beta)r_j,R_n\}$. Then we have
\begin{equation}
r_j^*\geq r_{j+1},\;\;j=2,\ldots,n.
\label{r*large}
\end{equation}
This is obvious if $r_j^*=R_n$ and follows from assumption \eqref{r-r<beta}$_1$ if $r_j^*=(1+\beta)r_j$.
From \eqref{r*large} and \eqref{r-r<beta}$_2$ it follows
\begin{equation}
r_j^*-r(x)\geq r_{j+1}-r_j\geq\lambda_j,\;\;j=2,\ldots,n.
\label{r*-rj}
\end{equation}
This and \eqref{pq-in1} imply that we can apply Lemma \ref{ball-exists} and conclude that $B_{\eta\lambda_j}(x)\subset\mathscr{S}_{r_j,\beta}$ and \eqref{arc-in}$_1$ follows from \eqref{near-om a} in Corollary \ref{esseScr}. The proof of \eqref{arc-in}$_2$ is similar.

Proof of (ii). Let $\check{g}_{r_j}$ and $\check{\vartheta}_{r_j}$ be defined as in \eqref{partial S} with $r_*=r_j$. The curve $[r_j,r_{j+1}]\ni r\rightarrow \check{g}_{r_j}(r)$ is a parametrization of the connected component of
$\partial\mathscr{S}_{r_j,\beta}\cap\{x(r):r\in[r_j,r_{j+1}]\}$ that contains $\hat{p}_j^+$. The curve $[r_j,r_{j+1}]\ni r\rightarrow \tilde{g}(r)=x(r,\check{\vartheta}_{r_j}(r)+\frac{\lambda_j}{r_j})$ lies on the left of $\check{g}_{r_j}([r_j,r_{j+1}])$ and satisfies
\begin{equation}
d^\circ(\tilde{g}(r),\check{g}_{r_j}(r))=\frac{r}{r_j}\lambda_j\geq\lambda_j.
\label{d0-diff}
\end{equation}
From \eqref{pq-in} we see that $\tilde{g}(r_j)=p_j^+$ and $\theta(q_{j+1}^+)-\theta(\tilde{g}(r_{j+1}))>0$. This and the fact that the curve $\tilde{g}$ turns its concavity toward increasing $\theta$ imply that the whole curve lies on the right of the segment $[p_j^+,q_{j+1}^+]$. This and \eqref{d0-diff} yield
\[d^\circ(x,\partial\mathscr{S}_{r_j,\beta})\geq\lambda_j,\;\;x\in[p_j^+,q_{j+1}^+],\]
that together with \eqref{r*-rj} allow to complete the proof as in case (i). This concludes the proof of \eqref{segm-in}$_1$. The same argument proves \eqref{segm-in}$_2$.
\end{proof}
From \eqref{PeQ} and \eqref{sequence} we obtain
\begin{equation}
\begin{split}
&\lim_{\stackrel{j\rightarrow+\infty}{j\leq n}}\frac{\vert p_j^+-p_j^-\vert}{r_{j+1}-r_j}
=0,\\
&\lim_{\stackrel{j\rightarrow+\infty}{j\leq n}}\frac{\vert q_{j+1}^+-q_{j+1}^-\vert}{r_{j+1}-r_j}=2\hat{c}.
\end{split}
\label{lengths}
\end{equation}
This implies that, by choosing $c_1>0$ sufficiently large in \eqref{lambda rj}, we can assume
\begin{equation}
\frac{\vert p_j^+-p_j^-\vert}{r_{j+1}-r_j}\leq\epsilon,\quad\frac{\vert q_{j+1}^+-q_{j+1}^-\vert}{r_{j+1}-r_j}\leq3\hat{c},\;\;j=1,\ldots,n.
\label{near-infty}
\end{equation}
From \eqref{lengths} it follows
\begin{equation}
\lim_{\stackrel{j\rightarrow+\infty}{j\leq n}}\tan{\psi_j}=\hat{c},
\label{varpsi}
\end{equation}
where $\psi_j$ is the angle that the vector $q_{j+1}^+-p_j^+$ forms with $L_{\theta_{r_j}}$ (recall that $L_\theta$ is the ray determined by $\theta$). Let $[x_j,x_{j+1}]$ be a segment that connects a point $x_j\in\mathrm{arc}(p_j^-,p_j^+)$ with a point $x_{j+1}\in\mathrm{arc}(q_{j+1}^-,q_{j+1}^+)$. From \eqref{near-infty}, since $\epsilon>0$ is a small fixed number, it follows
\begin{equation}
[x_j,x_{j+1}]\subset\mathscr{I}_j,\;\;\;j=j_0,\ldots,n,
\label{in}
\end{equation}
where $\mathscr{I}_j\subset\{x: r(x)\in[r_j,r_{j+1}]\}$ is the set bounded by
 the union of the segments $[p_j^+,q_{j+1}^+]$, $[p_j^-,q_{j+1}^-]$ and of the arcs $\mathrm{arc}(p_j^-,p_j^+)$, $\mathrm{arc}(q_{j+1}^-,q_{j+1}^+)$, see Figure \ref{Ij}.  From \eqref{lengths} and \eqref{varpsi} we see that we can also assume
 \begin{equation}
\tilde{\psi}_j\leq\tilde{\psi}_0,\;\;\;j=1,\ldots,n,
\label{varpsi1}
\end{equation}
 where $\tilde{\psi}_j$ is the angle between $[x_j,x_{j+1}]$ and the ray $L_{\theta(x_j)}$ and $\tilde{\psi}_0$ a constant independent of $n$ and $1\leq j\leq n$. By neglecting the first $j_0$ terms of the sequences defined in \eqref{sequence} and by renumbering via $j=i+j_0$ we can assume that \eqref{in} and \eqref{varpsi1} hold and $\mathscr{I}_j$ is well defined for $j=1,\ldots,n$.

  Lemma \ref{balls-} (i) does not apply for $j=n+1$ and  Lemma \ref{balls-} (ii) does not apply for $j=n$. This together with the fact that the length of $\mathrm{arc}(p_{n+1}^-,p_{n+1}^+)$ diverges to $+\infty$ with $n$ can cause a significant error in the derivation of a lower bound for the energy of $u^{R_n}$ in the the diffuse interface that we define later. Therefore we need a new definition for the set $\mathscr{I}_n$.

\begin{lemma}
\label{ball-exists-n}
Assume that $n\in\N$ is sufficiently large. Let $\ell_{r_{n+1}}^\pm$, $\hat{p}_{n+1}^\pm=\hat{p}_{r_{n+1}}^\pm $, be as in Definition \ref{def} with $r_*=r_{n+1}$. Then
\begin{enumerate}
\item $[p_n^\pm,q_{n+1}^\pm]\cap\ell_{r_{n+1}}^\pm=\{q^\pm\}$ for some $q^\pm$ and there exist $\eta\in(0,1)$, $k_i\in(0,1)$, $i=1,2$,  independent of $n$ such that
\begin{equation}
k_1\lambda_n\leq r_{n+1}-r(q^\pm)\leq k_2(r_{n+1}-r_n).
\label{qpm-bounds}
\end{equation}
\item
\begin{equation}
\begin{split}
& x\in[p_n^\pm,q^\pm]\;\;\;\Rightarrow\\\\
&\vert u^{R_n}-a_\pm\vert\leq c\delta^\alpha\;\;\text{on}\;\;B_{\eta\lambda_n}(x)\setminus\tilde{\Sigma}.
\end{split}
\label{ball+}
\end{equation}
\item
\begin{equation}
\begin{split}
& x\in[q^\pm,\hat{p}_{n+1}^\pm]\;\;\text{and}\;\;r_{n+1}-r(x)=\rho\;\;\Rightarrow\\\\
&\vert u^{R_n}-a_\pm\vert\leq c\delta^\alpha\;\;\text{on}\;\;B_{\frac{\rho}{2}}(x)\setminus\tilde{\Sigma}.
\end{split}
\label{ball++}
\end{equation}
\end{enumerate}
\end{lemma}
\begin{proof}
Let $\varphi_{n+1}=\varphi_*$ with $\varphi_*$ defined as in Remark \ref{remark} for $r_*=r_{n+1}$, $\varphi_{n+1}$ is the angle between the two tangents to $\hat{\mathscr{S}}_{r_{n+1},\beta}$ at $\hat{p}_{n+1}^+$. From \eqref{phi*} $\varphi=\lim_{n\rightarrow+\infty}\varphi_{n+1}$ has $\sin{\varphi}=\frac{1}{\sqrt{1+\hat{c}^2}}$. Note that this and \eqref{varpsi} imply $\varphi=\frac{\pi}{2}-\psi$, $\psi= \lim_{n\rightarrow+\infty}\psi_n$. Note also that
\eqref{lengths}$_1$ and $\mathrm{arc}(\hat{p}_{n+1}^-,\hat{p}_{n+1}^+)\subset\mathrm{arc}(p_{n+1}^-,p_{n+1}^+)$ imply $\lim_{n\rightarrow+\infty}\frac{\vert\hat{p}_{n+1}^+-\hat{p}_{n+1}^-\vert}{r_{n+1}-r_n}=0$.

These observations yield that, for large $n$,
 $p_n^+\simeq p_n^-$, $ q_{n+1}^+$ and $q_{n+1}^-$ are approximately the vertices of an isosceles triangle $T_n$ with basis $2\hat{c}(r_{n+1}-r_n)$ and height $r_{n+1}-r_n$ as indicated in Figure \ref{tn}. In the same approximation $\hat{p}_{n+1}^+\simeq\hat{p}_{n+1}^-$ lies on the basis of $T_n$ and $\ell_{r_{n+1}}^\pm$  forms an angle of size $\frac{\varphi}{2}$ ($\varphi=\frac{\pi}{2}-\psi$) with the basis of $T_n$. It follows that the points $q^\pm$ claimed in (i) exist and satisfy the right inequality in \eqref{qpm-bounds} for some $k_2\in(0,1)$ independent of $n$.

 To complete the proof of (i) we note that  $\hat{p}^+=x(r_{n+1},\check{\vartheta}_{r_n}(r_{n+1}))$ is the extreme position allowed to $\hat{p}_{n+1}^+$ on $\mathrm{arc}( q_{n+1}^-, q_{n+1}^+)$. Similarly we define $\hat{p}^-$, the extreme possible position of $\hat{p}_{n+1}^-$ on $\mathrm{arc}( q_{n+1}^-, q_{n+1}^+)$. This follows from  Lemma \ref{rrprime} with $r_*=r_n$. The left inequality \eqref{qpm-bounds} is determined by the limit position $\hat{q}^\pm$  assumed by $q^\pm$ when $\hat{p}_{n+1}^\pm=\hat{p}^\pm$. For $n$ large $\hat{q}^+,\hat{p}_{n+1}^+=\hat{p}^+$ and $q_{n+1}^+$ are the vertices of a triangle which, as illustrated in Figure \ref{tn},  has $\vert q_{n+1}^+-\hat{p}^+\vert\simeq\lambda_{n+1}$ and the angles in $q_{n+1}^+$ and in $\hat{p}^+$ approximately equal to $\varphi$ and $\frac{\varphi}{2}$ respectively. A similar argument applies to $\hat{q}^-,\hat{p}_{n+1}^-$ and $q_{n+1}^-$. The lower bound for $r_{n+1}-r(q^\pm)$ is a consequence of the geometric properties of the above triangle.
 \hskip.2cm

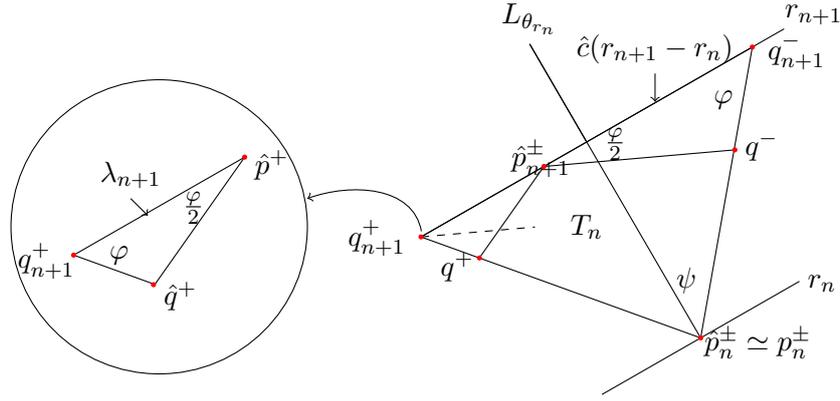
\begin{figure}
  \begin{center}
\begin{tikzpicture}[scale=.75]
\draw[thin] (0,0)--(3.464,2);
\draw [thin](0,0)--(-2.9066,-1.6782);
\draw[thin] (0,0)--(-1,1.732);
\draw[thin] (0,0)--(-1,1.732);

\draw[thin] (2.,-3.464)--(3.732,-2.464);
\draw[thin] (2.,-3.464)--(.268,-4.464);
\draw[thin] (0,0)--(2,-3.464);L

\draw[] (2.9066,1.6782)--(-2.9066,-1.6782);
\draw[] (2.9066,1.6782)--(2,-3.464);
\draw[] (-2.9066,-1.6782)--(2,-3.464);

\draw[thin] (-.75,-.433)--(2.63,-.1369);
\draw[thin] (-.75,-.433)--(-1.882,-2.05);

\draw[dashed] (-2.9066,-1.6782)--(-.9066,-1.503);
\draw[<-] (-4.9,-1) to [out=20,in=100] (-2.90,-1.5782);

\path[fill=red] (2.60,-.1395) circle [radius=.045];
\path[fill=red] (-.75,-.433) circle [radius=.045];
\path[fill=red] (-1.882,-2.05) circle [radius=.045];

\path[fill=red] (2,-3.464) circle [radius=.045];
\path[fill=red] (2.9066,1.6782) circle [radius=.045];
\path[fill=red] (-2.9066,-1.6782) circle [radius=.045];


\node[right] at (3.3,2.25) {$r_{n+1}$};
\node[right] at (3.7,-2.5) {$r_n$};
\node[above] at (.5,-.5) {$\frac{\varphi}{2}$};
\node[above] at (2.4,.4) {$\varphi$};
\node[left] at (-1.8,-2.2) {$q^+$};
\node[right] at (2.6,-.07) {$q^-$};
\node[above] at (1.75,-2.85) {$\psi$};
\node[above] at (-1,1.732) {$L_{\theta_{r_n}}$};
\node[above] at (3,-4) {$\hat{p}_n^\pm\simeq p_n^\pm$};
\draw[->] (1.2,1.211)--(1.2,.712);
\node[above] at (1.2,1.211) {$\hat{c}(r_{n+1}-r_n)$};
\node[right] at (3.,1.6) {$q_{n+1}^-$};
\node[left] at (-3.,-1.7) {$q_{n+1}^+$};
\node[] at (-.8,-.3) {$\hat{p}_{n+1}^\pm$};
\node[] at (0,-1.5) {$T_n$};
\draw[thin] (-9,-2)--(-6,-.268);
\draw[thin] (-9,-2)--(-7.6,-2.5096);
\draw[thin] (-6,-.268)--(-7.6,-2.553);

\path[fill=red] (-9,-2) circle [radius=.045];
\path[fill=red] (-6,-.268) circle [radius=.045];
\path[fill=red] (-7.6,-2.523) circle [radius=.045];

\draw [thin] (-7.5,-1.5) circle [radius=2.6];;

\node[below] at (-6.9,-.7) {$\frac{\varphi}{2}$};
\node[above] at (-8.2,-2.35) {$\varphi$};
\node[above] at (-8,-1) {$\lambda_{n+1}$};
\draw[->] (-8,-1)--(-7.7,-1.3);
\node[left] at (-8.8,-2.1) {$q_{n+1}^+$};
\node[right] at (-6,-.4) {$\hat{p}^+$};
\node[right] at (-7.6,-2.723) {$\hat{q}^+$};

\end{tikzpicture}
\end{center}
\caption{The triangle $T_n$ and the points $q^\pm$. The triangle $\hat{q}^+, \hat{p}^+, q_{n+1}^+$.}
\label{tn}
\end{figure}

From the proof of (ii) in Lemma \ref{balls-} we have that $[p_n^+,q_{n+1}^+]$ is contained in $\bar{\mathscr{S}}_{r_n,\beta}$. Moreover $x\in[p_n^+,q_{n+1}^+]$ implies $B_{\eta\lambda_n}(x)\subset\mathscr{S}_{r_n,\beta}$ provided $r_{n+1}-r(x)\geq\eta\lambda_n$. From \eqref{qpm-bounds} it follows that, by reducing the value of $\eta$ if necessary, we can make sure that this condition is satisfied for every $x\in[p_n^+,q^+]$. This proves the first part of (ii). The proof of the second part is similar.
\hskip.2cm

Statement (iii) is a plain consequence of part (ii) of Lemma \ref{ball-exists} with $r_*=r_{n+1}$. The proof is complete.
\end{proof}
 We are now in the position to give a suitable definition of the set $\mathscr{I}_n$. We define $\mathscr{I}_n$ to be the subset of $\{x: r(x)\in[r_n,r_{n+1}]\}$ bounded by the union of $[p_n^-,q^-]$, $[p_n^+,q^+]$, $[q^-,\hat{p}_{n+1}^-]$,  $[q^+,\hat{p}_{n+1}^+]$, $\mathrm{arc}( p_n^-, p_n^+)$ and $\mathrm{arc}(\hat{p}_{n+1}^-, \hat{p}_{n+1}^+)$.
\vskip.2cm

  Define  (see Figure \ref{I})
  \[\mathscr{I}=\cup_{j=1}^n\mathscr{I}_j.\]

 The set $\mathscr{I}$  is a kind of diffuse interface that separate regions where $u^{R_n}$ is near to $a$ or to $\omega^{-1}a$. Indeed let $\mathscr{R}$ be defined by
 \[B_{r_{n+1}}\setminus B_{r_1}\cup_{j=1}^n\omega^{j-1}\mathscr{I}=\cup_{j=1}^n\omega^{j-1}\mathscr{R}.\]
Then
 \begin{equation}
 x\in\omega^{j-1}\mathscr{R}\;\;\;\Rightarrow\;\;\vert u^{R_n}(x)-\omega^{j-1}a\vert\leq c\delta^\alpha,\;\;r(x)\not\in\Sigma,\;\;j=1,\ldots,N-1.
 \label{ok-R}
 \end{equation}

 \subsection{An upper bound for the length of $\mathscr{I}$.}\label{lenghtUB}
From what we know up to now the structure of $\mathscr{I}$ can be quite complex, for example we can not exclude that $\mathscr{I}$ revolves several times around $B_{r_1}$. We will show that, instead, the shape of $\mathscr{I}$ can be controlled. We will associate to $\mathscr{I}$ a kind of length and show that most of the energy of $u^{R_n}$ is contained in $\mathscr{I}$ (and in its images under $\omega$) and proportional to its length. Then from the upper bound \eqref{UB} we obtain that the the difference from the length of $\mathscr{I}$ and $R_n$ is bounded by a constant independent of $n$. This implies a strong restriction on the geometry of $\mathscr{I}$ and eventually leads to Theorem \ref{Th2}.

If $(u,v)$, $u,v\in\R^2$, is an ordered pair of linearly independent vectors we say that $(u,v)$ is positive if the rotation from $u$ to $v$ through an angle $<\pi$ is counterclockwise, negative otherwise.

 Let $\Gamma$ be the family of rectifiable curves that connect $x(r_2,\theta_{r_2})$ to $x(r_{n+1},\theta_{r_{n+1}})$ and are contained in $\mathscr{I}$. We let $\vert\gamma\vert$ the length of $\gamma\in\Gamma$.
\begin{proposition}\label{Gamma}
\begin{enumerate}
\item
There exists $\gamma^m\in\Gamma$ and $K>0$ independent of $n$ such that
\begin{equation}
\vert\gamma^m\vert=\min_{\gamma\in\Gamma}\vert\gamma\vert,
\label{Min}
\end{equation}
and
 \begin{equation}
 \frac{d s}{d r}\leq K,\;\;r\in[r_2,r_{n+1}],
 \label{Kstar}
 \end{equation}
 where $s:[0,\vert\gamma^m\vert]\rightarrow\R$ is the curvilinear abscissa along $\gamma^m$.

 Moreover
 \begin{equation}
 \gamma^m([s_j,s_{j+1}])=[\gamma_j^m,\gamma_{j+1}^m],\;\;j=2,\ldots,n
 \label{poligonal}
 \end{equation}
 where $s_j$ and $\gamma_j^m$ are defined by $\gamma_j^m=\gamma^m(s_j)\in\mathrm{arc}(p_j^-,p_j^+)$.
\item
 Set $\tau_j=\frac{\gamma_{j+1}^m-\gamma_j^m}{\vert\gamma_{j+1}^m-\gamma_j^m\vert},\;\;j=2,\ldots,n$. Then
 \[\tau_{j-1}\neq\tau_j\;\;\Rightarrow\;\;\gamma_j^m\in\{p_j^-,p_j^+\}\]
 and
 \[\begin{split}
 &\gamma_j^m=p_j^-\;\;\Rightarrow\;\;(\tau_{j-1},\tau_j)\;\;\text{is negative},\\
 &\gamma_j^m=p_j^+\;\;\Rightarrow\;\;(\tau_{j-1},\tau_j)\;\;\text{is positive}.
 \end{split}\]
 \end{enumerate}
\end{proposition}

\begin{figure}
  \begin{center}
\begin{tikzpicture}[scale=.75]
\path[fill=lightgray] (0,0) circle [radius=1.1];

\draw [dotted] (0,2) arc [radius=2, start angle=90, end angle= 60];
\draw [dotted] (0,3) arc [radius=3, start angle=90, end angle= 60];
\draw [dotted] (0,4.1) arc [radius=4.1, start angle=90, end angle= 60];
\draw [dotted] (0,5.3) arc [radius=5.3, start angle=90, end angle= 60];
\draw [dotted] (0,6.6) arc [radius=6.6, start angle=90, end angle= 60];
\draw [dotted] (0,8) arc [radius=8, start angle=90, end angle= 60];
\draw [thin] (0,8.2) arc [radius=8.2, start angle=90, end angle= 60];

\draw [dotted] (0,2) arc [radius=2, start angle=90, end angle= 120];
\draw [dotted] (0,3) arc [radius=3, start angle=90, end angle= 120];
\draw [dotted] (0,4.1) arc [radius=4.1, start angle=90, end angle= 120];
\draw [dotted] (0,5.3) arc [radius=5.3, start angle=90, end angle= 120];
\draw [dotted] (0,6.6) arc [radius=6.6, start angle=90, end angle= 120];
\draw [dotted] (0,8) arc [radius=8, start angle=90, end angle= 120];
\draw [thin] (0,8.2) arc [radius=8.2, start angle=90, end angle= 120];


\path[fill=lightgray] (-.21,1.08)
 arc[radius=1.1, start angle=101, end angle=79]
 to (.908,1.782)
  arc[radius=2, start angle=63, end angle=117]
 to (-.21,1.08);


\path[fill=lightgray] (.249,1.9845) arc[radius=2, start angle=82.85, end angle=97.15]
to (-1.0751,2.801)
arc[radius=3, start angle=111, end angle=69]
to (.249,1.9845);


\path[fill=lightgray] (.584,2.943) arc[radius=3, start angle=78.775, end angle=69.225]
to (2.173,3.477)
arc[radius=4.1, start angle=58, end angle=90]
to (.584,2.943);


\path[fill=lightgray] (.2146,4.094) arc[radius=4.1, start angle=87, end angle=80.013]
to (1.7686,4.996)
arc[radius=5.3, start angle=70.506, end angle=96.5065]
to (.2146,4.094);


\path[fill=lightgray] (.8834,5.226)  arc[radius=5.3, start angle=80.405, end angle=75]
to (2.642,6.0481)
arc[radius=6.6, start angle=66.4025, end angle=89.0025]
to (.8834,5.226) ;


\path[fill=lightgray] (1.4846,6.4308)
 arc[radius=6.6, start angle=77, end angle=72.659]
 to (2.654,6.75) -- (1.595,7.84)-- (.8682,7.6654) -- (1.4846,6.4308);


\path[fill=red] (-.21,1.08) circle [radius=.045];
\path[fill=red] (.21,1.08) circle [radius=.045];

\path[fill=red] (-.908,1.782) circle [radius=.045];
\path[fill=red] (.908,1.782) circle [radius=.045];


\path[fill=red] (.249,1.9845) circle [radius=.045];
\path[fill=red] (-.249,1.9845) circle [radius=.045];

\path[fill=red] (1.0751,2.801) circle [radius=.045];
\path[fill=red] (-1.0751,2.801) circle [radius=.045];


\path[fill=red] (.584,2.943) circle [radius=.045];
\path[fill=red] (1.0641,2.8049) circle [radius=.045];

\path[fill=red] (0,4.1) circle [radius=.045];
\path[fill=red] (2.173,3.477) circle [radius=.045];


\path[fill=red] (.2146,4.094) circle [radius=.045];
\path[fill=red] (.7111,4.038) circle [radius=.045];

\path[fill=red] (1.7686,4.996) circle [radius=.045];
\path[fill=red] (-.6,5.266) circle [radius=.045];


\path[fill=red] (.8834,5.226) circle [radius=.045];
\path[fill=red] (1.372,5.119) circle [radius=.045];

\path[fill=red] (.1149,6.598) circle [radius=.045];
\path[fill=red] (2.642,6.0481) circle [radius=.045];


\path[fill=red] (1.4846,6.4308) circle [radius=.045];
\path[fill=red] (1.9672,6.3) circle [radius=.045];

\path[fill=red] (.8682,7.6654) circle [radius=.045];
\path[fill=red] (2.654,6.75) circle [radius=.045];

\path[fill=red] (1.554,7.848) circle [radius=.045];
\path[fill=red] (1.636,7.831) circle [radius=.045];



\draw [thin] (0,2) -- (.584,2.943) -- (.7111,4.038) -- (1.4846,6.4308) -- (1.595,7.84);

\path[fill=black] (0,0) circle [radius=.035];
\node[right] at (0,0) {$0$};

\node[left] at (-.55,.952) {$r_1$};
\node[left] at (-1,1.732) {$r_2$};
\node[left] at (-2.05,3.551) {$r_j$};
\node[left] at (-3.3,5.716) {$r_n$};
\node[left] at (-4.,6.7682) {$r_{n+1}$};
\node[left] at (-4.,7.1514) {$R_n$};
\node[right] at (0,4.8) {$\mathscr{I}$};
\node[right] at (.5,6) {$\gamma^m$};
\draw[->] (1.1,6.) -- (1.25,5.75);

\node [left] at (.8682,7.6654)  {$q^+$};
\node [right] at (2.654,6.75)  {$q^-$};

\node [left] at (.05,4.)  {$q_j^+$};
\node [right] at (2.25,3.5)  {$q_j^-$};

\node [left] at (-.3,4.5)  {$p_j^+$};
\draw[->] (-.35,4.4) -- (.15,4.12);
\node [right] at (1.02,4.2)  {$p_j^-$};
\draw[->] (1.1,4.2) -- (.75,4.05);

\node [above] at (1.65,7.8) {$\hat{p}_{n+1}^\pm$};

\draw (.292,2.4715) -- (.55,2.3114);
\draw (.292,2.4715) -- (-.5,2.9616);

\node [right] at (.5,2.1514)  {$\xi_r^-$};
\node [left] at (-.25,3.2116)  {$\xi_r^+$};

\draw [thin] (0,2.4887) arc [radius=2.4887, start angle=90, end angle= 60];
\draw [thin] (0,2.4887) arc [radius=2.4887, start angle=90, end angle= 120];
\node [right] at (1.2443,2.1552)  {$r$};

\path[fill=black] (.55,2.3114) circle [radius=.04];
\path[fill=black] (-.5,2.9616) circle [radius=.04];
\end{tikzpicture}
\end{center}
\caption{The diffuse interface $\mathscr{I}$, the curve $\gamma^m$ and the points $\xi_r^\pm$.}
\label{I}
\end{figure}
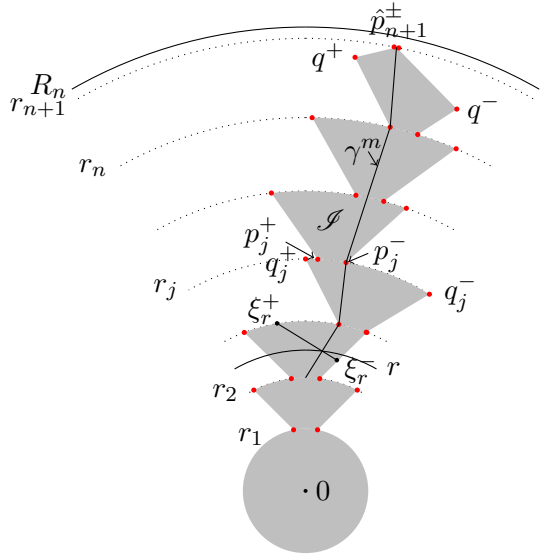

\begin{proof}
1. Given $\gamma\in\Gamma$ set $s_2=0$, $s_{n+1}=\vert\gamma\vert$  ($s$ the curvilinear abscissa along $\gamma$). For $j=3,\ldots,n$ there exists $s_j\in(0,\vert\gamma\vert)$ such that $\gamma(s_j)\in\mathrm{arc}(p_j^-,p_j^+)$. We can assume that $s_j$ is chosen so that $s_j<s_{j+1}$, $j=2,\ldots,n$. From \eqref{in} we have
\[[\gamma(s_j),\gamma(s_{j+1})]\subset\mathscr{I}_j,\;\;j=2,\ldots,n.\]
Therefore the curve $\hat{\gamma}=\cup_{j=2}^n[\gamma(s_j),\gamma(s_{j+1}]$ belongs to $\Gamma$ and
\begin{equation}
\vert\hat{\gamma}\vert\leq\vert\gamma\vert,\;\;\gamma\in\Gamma.
\label{^is better}
\end{equation}
The length $\vert\hat{\gamma}\vert$ of $\hat{\gamma}$ is a continuous function of the $n-2$ points $\gamma(s_j)\in\mathrm{arc}(p_j^-,p_j^+)$, $j=3,\ldots,n-1$. This implies the existence of $\gamma^m\in\Gamma$ that satisfies \eqref{Min}. The bound \eqref{Kstar} follows from \eqref{varpsi1} with $K=\frac{1}{\cos{\tilde{\psi}_0}}$. This completes the proof of (i). To prove (ii) we observe that, if $\gamma_j^m\not\in\{p_j^-,p_j^+\}$ and $\tau_{j-1}\neq\tau_j$, there exist $x,y\neq\gamma_j^m$, $x\in[\gamma_{j-1}^m,\gamma_j^m]$ and $y\in[\gamma_j^m,\gamma_{j+1}^m]$ such that $[x,y]\subset\mathscr{I}$. This contradicts the minimality of $\gamma^m$ since $\vert y-x\vert<\vert\gamma_j^m-x\vert+\vert y-\gamma_j^m\vert$. This contradiction proves that $\gamma_j^m\not\in\{p_j^-,p_j^+\}$ implies $\tau_{j-1}=\tau_j$. The same argument applies to the case $\gamma_j^m\in\{p_j^-,p_j^+\}$.
The proof is complete.
\end{proof}

Let $\upsilon_j\in\R^2$, $j=2,\ldots,n$ a unit vector orthogonal to $\tau_j$ and such that $(\tau_j,\upsilon_j)$ is positive. Given $r\in[r_j,r_{j+1}]$ let $N_{j,r}=\{x=\gamma^m(r)+t\upsilon_j: t\in\R\}$ be the line through $\gamma^m(r)$ orthogonal to $\tau_j$.
For $r\neq r_j$, $j=3,\ldots,n$ and $j$ such that $r\in(r_j,r_{j+1})$ let $[\xi_r^-,\xi_r^+]$ be the closure of the connected component of $N_{j,r}\cap\mathring{\mathscr{I}}$ that contains $\gamma^m(r)$. $\mathring{E}$ denotes the interior of $E$. If $r=r_j$ and $\tau_{j-1}=\tau_j$ we have $N_{j-1,r_j}=N_{j,r_j}$ and we define $[\xi_r^-,\xi_r^+]$ as before. If instead, $\tau_{j-1}\neq\tau_j$, we denote $[\xi_{r_j^-}^+,\xi_{r_j^-}^-]$ and $[\xi_{r_j^+}^+,\xi_{r_j^+}^-]$ the two segments that the previous definition yields with respect to $N_{j-1,r_j}$ and $N_{j,r_j}$ respectively. From Proposition \ref{Gamma} (ii) it follows that $r=r_j$ and $\tau_{j-1}\neq\tau_j$ imply that one of the following alternatives holds
\begin{description}
\item[a)]$\gamma_j^m=p_j^-=\xi_{r_j^\pm}^-$\quad and \quad $(\xi_{r_j^-}^+-\xi_{r_j^-}^-,\xi_{r_j^+}^+-\xi_{r_j^+}^-)$\quad is negative,
\item[b)]$\gamma_j^m=p_j^+=\xi_{r_j^\pm}^+$\quad and \quad $(\xi_{r_j^-}^--\xi_{r_j^-}^+,\xi_{r_j^+}^--\xi_{r_j^+}^+)$\quad is positive.
\end{description}

A natural consequence of a) and b) is Proposition \ref{no-inters} (ii) below that we prove in detail in Section \ref{app}

 Define
\[\begin{split}
&\partial\mathscr{I}_j^\pm=[p_j^\pm,q_{j+1}^\pm]\cup\mathrm{arc}[q_{j+1}^\pm,p_{j+1}^\pm],\;\;j=1,\ldots,n-1,\\\\
&\partial\mathscr{I}_n^\pm=[p_n^\pm,q^\pm]\cup\mathrm{arc}[q^\pm,\hat{p}_{n+1}^\pm].
\end{split}\]
For $j=1,\ldots,n-1$,  $\partial\mathscr{I}_j^-$ and $\partial\mathscr{I}_j^+$ are the connected components of
$\partial\mathscr{I}_j\setminus\cup_{i=0,1}\mathrm{arc}[p_{j+i}^-,p_{j+i}^+]$.
 $\partial\mathscr{I}_n^-$ and $\partial\mathscr{I}_n^+$ are the connected components of
$\partial\mathscr{I}_n\setminus(\mathrm{arc}[p_n^-,p_n^+]\cup\mathrm{arc}[\hat{p}_{n+1}^-,\hat{p}_{n+1}^+])$.

\begin{proposition}\label{no-inters} It results
\begin{equation}
\xi_r^\pm\in\cup_{i=j-1}^{j+1}\partial\mathscr{I}_i^\pm,\;\;\;r\in(r_j,r_{j+1}),\;\;j=2,\ldots,n-1.
\label{xi-in dpm}
\end{equation}
and
\begin{equation}
[\xi_r^-,\xi_r^+]\cap[\xi_{r^\prime}^-,\xi_{r^\prime}^+]=\emptyset,\;\;r\neq r^\prime\in[r_2,r_{n+1}]\setminus\{r_j\}_{j=2}^{n+1}.
\label{no-inters1}
\end{equation}
\end{proposition}
\begin{proof}
See Section \ref{app}.
\end{proof}

\begin{proposition}
\label{-1}
Set $a_-=\omega^{-1}a$, $a_+=a$. Then
\begin{enumerate}
\item
\[\begin{split}
&r\in(r_j,r_{j+1}),\;\;j=2,\ldots,n-1,\\
&\Rightarrow\\
&\vert u^{R_n}-a_\pm\vert\leq c\delta^\alpha,\;\;x\in B_{\eta\lambda_{j-1}}(\xi_r^\pm)\setminus\tilde{\Sigma}.
\end{split}\]
\item There exists $\bar{\rho}>0$, $\bar{\eta}\in(0,1)$, $0<k<k^\prime$ independent of $n$ and $\tilde{r}$ such that $r_{n+1}-k^\prime\lambda_{n-1}\leq\tilde{r}\leq r_{n+1}-k\lambda_{n-1}$ and
\begin{equation}
\begin{split}
&r\in(r_n,\tilde{r})\;\;\;\Rightarrow\\\\
&\vert u^{R_n}-a_\pm\vert\leq c\delta^\alpha,\;\;x\in B_{\eta\lambda_{n-1}}(\xi_r^\pm)\setminus\tilde{\Sigma},
\end{split}
\label{A}
\end{equation}
and
\begin{equation}
\begin{split}
&\rho\in(\bar{\rho},r_{n+1}-\tilde{r})\;\;\;\Rightarrow\\\\
&\vert u^{R_n}-a_\pm\vert\leq c\delta^\alpha,\;\;x\in B_{\bar{\eta}\rho}(\xi_{r_{n+1}-\rho}^\pm)\setminus\tilde{\Sigma}.
\end{split}
\label{B}
\end{equation}
\end{enumerate}
\end{proposition}
\begin{proof}
See Section \ref{app}.
\end{proof}

We are now in the position to derive a sharp upper bound for the length $\vert\gamma^m\vert$ of $\gamma^m$.
For each $r\in(r_2,r_{n+1}-\bar{\rho})\setminus\{r_j\}_{j=3}^n$ let $J^*(r)$ be the one dimensional energy of the
restriction of $u^{R_n}$ to the segment $[\xi_r^-,\xi_r^+]$. From Proposition \ref{no-inters} and the fact that, for $r\in(r_j,r_{j+1})$, $[\xi_r^-,\xi_r^+]$ remains orthogonal to $[\gamma_j^m,\gamma_{j+1}^m]$, it follows via Lemma \ref{basic}
\begin{equation}
J_{B_{R_n}}(u^{R_n})\geq NJ_{\mathscr{I}}(u^{R_n})\geq N\int_0^{\vert\gamma^m\vert}J^*(r(s))ds,
\label{LBmascr}
\end{equation}
where $s\in(0,\vert\gamma^m\vert)$ is the curvilinear abscissa along $\gamma^m$ and $s\rightarrow r(s)$ the inverse of $r\rightarrow s(r)$ which exists by \eqref{Kstar}.
Set
\begin{equation}
\tilde{c}=\frac{1}{\bar{k}\eta},
\label{tildec}
\end{equation}
where $\tilde{c}$ is the constant in \eqref{lambda rj} and $\eta$ is defined in Proposition \ref{-1}. We assume
\begin{equation}
\tilde{c}\ln{r_1}=\lambda_1\geq\frac{\bar{l}}{\eta},
\label{>barl}
\end{equation}
which is equivalent to the assumption that the constant $c_1$  in \eqref{lambda rj} be sufficiently large. From  \eqref{tildec}, \eqref{>barl}, Corollary \ref{cor} and Proposition \ref{-1} we obtain
\begin{equation}
\begin{split}
&\left.\begin{array}{l}
\vert u^{R_n}(\xi_r^\pm)-a_\pm\vert\leq 2\delta e^{-\bar{k}(\eta\lambda_{j-1}-\bar{l})}\\
=2\delta e^{\bar{k}\bar{l}} e^{-\bar{k}\eta\tilde{c}\ln{r_{j-1}}}=\frac{2\delta e^{\bar{k}\bar{l}}}{r_{j-1}},
\end{array}\right.
r\in(r_j,r_{j+1}),\;\;j=2,\ldots,n-1,\\
&\;\;\vert u^{R_n}(\xi_r^\pm)-a_\pm\vert\leq\frac{2\delta e^{\bar{k}\bar{l}}}{r_{n-1}},
\;\;r\in(r_n,\tilde{r}).
\end{split}
\label{j..}
\end{equation}
Instead, in the interval $(\tilde{r},r_{n+1}-\bar{\rho})$, assuming also that $\bar{\rho}\geq\frac{\bar{l}}{\bar{\eta}}$, we obtain
\begin{equation}
\vert u^{R_n}(\xi_r^\pm)-a_\pm\vert\leq 2\delta e^{\bar{k}\bar{l}}e^{-\bar{k}\bar{\eta}(r_{n+1}-r)},\;\;r\in(\tilde{r},r_{n+1}-\bar{\rho}).
\label{near-n+1}
\end{equation}
From \eqref{j..} and Lemma \ref{lower-sigma} it follows
\begin{equation}
\begin{split}
&J^*(r)\geq\sigma-C_W\frac{4\delta^2 e^{2\bar{k}\bar{l}}}{r_{j-1}^2},\;\;r\in(r_j,r_{j+1}),\\
&J^*(r)\geq\sigma-C_W\frac{4\delta^2 e^{2\bar{k}\bar{l}}}{r_{n-1}^2},\;\;r\in(r_n,\tilde{r}),
\end{split}
\label{integ}
\end{equation}
and
\begin{equation}
J^*(r)\geq\sigma-C_W4\delta^2 e^{2\bar{k}\bar{l}}e^{-2\bar{k}\bar{\eta}(r_{n+1}-r)},\;\;r\in(\tilde{r},r_{n+1}-\bar{\rho}).
\label{integn}
\end{equation}
Recall that from Proposition \ref{Gamma} and Lemma \ref{1:2} we have $\frac{ds}{dr}\leq K$, $\frac{r_j}{r_{j-1}}\leq(1+\frac{\beta}{2})$ and $r_{j+1}-r_j\leq C^0r_j^\frac{1}{2}$. This and \eqref{integ} imply \[\begin{split}
&\frac{1}{r_{j-1}^2}\int_{s(r_j)}^{s(r_{j+1})}\frac{ds}{dr}dr\leq K\frac{r_{j+1}-r_j}{r_{j-1}^2}
\leq\frac{K(1+\frac{\beta}{2})^2C^0}{r_j^\frac{3}{2}},\\
&\frac{1}{r_{n-1}^2}\int_{s(r_j)}^{s(\tilde{r})}\frac{ds}{dr}dr\leq\frac{K(1+\frac{\beta}{2})^2C^0}{r_n^\frac{3}{2}}.
\end{split}\]
which, via \eqref{integ}, yields
\begin{equation}
\begin{split}
&\int_{s(r_j)}^{s(r_{j+1})}J^*(r(s))ds\geq\sigma(s(r_{j+1})-s(r_j))
-\frac{C^*}{r_j^\frac{3}{2}},\;\;j=2,\ldots,n-1,\\
&\int_{s(r_j)}^{s(\tilde{r})}J^*(r(s))ds\geq\sigma(s(\tilde{r})-s(r_n))-\frac{C^*}{r_n^\frac{3}{2}}.
\end{split}
\label{int-int}
\end{equation}
where we have set $C^*=4\delta^2C_WKC^0(1+\frac{\beta}{2})^2e^{2\bar{k}\bar{l}}$. Finally, in a similar way, from \eqref{integn} we get
\begin{equation}
\begin{split}
&\int_{s(\tilde{r})}^{s(r_{n+1}-\bar{\rho})}J^*(r(s))ds\geq\sigma(s(r_{n+1}-\bar{\rho})-s(\tilde{r}))-
\frac{4\delta^2C_WKe^{2\bar{k}\bar{l}}}{2\bar{k}\bar{\eta}}\\
&\geq\sigma(s(r_{n+1})-s(\tilde{r}))-K\bar{\rho}-
\frac{4\delta^2C_WKe^{2\bar{k}\bar{l}}}{2\bar{k}\bar{\eta}}.
\end{split}
\label{J*n}
\end{equation}
By adding this estimate with the estimates \eqref{int-int} from $j=1$ to $j=n$ we obtain
\begin{equation}
\begin{split}
&\int_0^{\vert\gamma^m\vert}J^*(r(s))ds\geq\int_{s(r_2)}^{s(r_{n+1}-\bar{\rho})}J^*(r(s))ds\\
&\geq\sigma(s(r_{n+1})-s(r_2))-C_1^*=\sigma\vert\gamma^m\vert-C_1^*,
\end{split}
\label{intJ*}
\end{equation}
where $C_1^*=C^*\sum_{j=1}^\infty\frac{1}{r_j^\frac{3}{2}}+K\bar{\rho}+
\frac{4\delta^2C_WKe^{2\bar{k}\bar{l}}}{2\bar{k}\bar{\eta}}$ is a constant independent of $n$. From \eqref{intJ*}, \eqref{LBmascr} and \eqref{UB} we get
\begin{equation}
\vert\gamma^m\vert\leq R_n+\frac{C_1^*+\frac{C_1}{N}}{\sigma}
\leq r_{n+1}+C_2^*,\;\;C_2^*=\vert\Sigma\vert+\frac{C_1^*+\frac{C_1}{N}}{\sigma}.
\label{LBgamma}
\end{equation}
This is the announced upper bound for the length of the interface.

\subsection{Existence of $C_N$-equivariant $N$-junctions}
We can now complete the proof of Theorem \ref{Th2}.

1. We have
\begin{equation}
\begin{split}
&\vert\gamma^m\vert\geq\vert\gamma^m(s(r))-\gamma_2^m\vert+\vert\gamma_{n+1}^m-\gamma^m(s(r))\vert,\\
&\vert\gamma^m(s(r))-\gamma_2^m\vert\geq r-r_2,\\
&\vert\gamma_{n+1}^m-\gamma^m(s(r))\vert=((r_{n+1}-r\cos{\vartheta})^2+r^2{\sin^2{\vartheta}})^\frac{1}{2}\\
&=((r_{n+1}-r)^2+4rr_{n+1}{\sin^2{\frac{\vartheta}{2}}})^\frac{1}{2},
\end{split}
\label{gamma>}
\end{equation}
where
\begin{equation}
\vartheta=\theta(\gamma^m(s(r)))-\theta(\gamma_{n+1}^m).
\label{varthetadef}
\end{equation}
 From \eqref{gamma>} and \eqref{LBgamma}, after some manipulation, we get
\[4\sin^2{\frac{\vartheta}{2}}\leq2C_3^*\frac{r_{n+1}-r}{rr_{n+1}}+\frac{{C_3^*}^2}{rr_{n+1}}
\leq\frac{C_3^*}{r}(2+\frac{{C_3^*}}{r_{n+1}}),\;\;C_3^*=C_2^*+r_2.\]
It follows $4\sin^2{\frac{\vartheta}{2}}\leq\frac{4C_3^*}{r}$ for $n$ sufficiently large. From this we conclude
\begin{equation}
\vartheta\leq\frac{C_4^*}{r^\frac{1}{2}},\;\;r\in[\hat{r},r_{n+1}],\;n\geq\bar{n},
\label{varteta-b}
\end{equation}
where $C_4^*>0$ and $\hat{r}\geq r_2$ are constants independent of $n\geq\bar{n}$, for some $\bar{n}$.
\vskip.2cm

2. The estimate \eqref{varteta-b} gives some control of the shape of $\gamma^m$, the \emph{spine} of the diffuse interface $\mathscr{I}$, and allows to show that $\mathscr{I}$ lie in a neighborhood of the ray $L_{\theta(\gamma_{n+1}^m)}$ in the sense that
\begin{equation}
\begin{split}
&\mathscr{I}\subset B_{\mathring{r}}\cup D,\\
&D=\{x(r,\theta): \vert\theta-\theta(\gamma_{n+1}^m)\vert\leq\frac{\mathring{C}}{r^\frac{1}{2}},\;r\in(\mathring{r},r_{n+1}),
\;\mathring{r}=\frac{(\mathring{C}N)^2}{\pi^2}\},
\end{split}
\label{I-geom}
\end{equation}
for some constants $\mathring{C}>0$ and $r_0\geq\hat{r}$ independent of $n\geq\bar{n}$. The condition $r\geq
\frac{(\mathring{C}N)^2}{\pi^2}$ in \eqref{I-geom} ensures that
\begin{equation}
\omega^jD\cap D=\emptyset,\;\;j=1,\ldots,N-1.
\label{DomD0}
\end{equation}
To prove \eqref{I-geom} we estimate the thickness of $\mathscr{I}$. For $r\in(r_j,r_{j+1}]$ let $x^\pm\in\mathscr{I}_j\cap C_r$. Then the definition of $\mathscr{I}_j$ in the proof of Lemma \ref{balls-} and \eqref{near-infty} imply
\begin{equation}\begin{split}
&\max_{x^\pm}\frac{d^\circ(x_-,x^+)}{r}\leq\frac{d^\circ(q_{j+1}^-,q_{j+1}^+)}{r_j}
\leq3\hat{c}\frac{r_{j+1}-r_j}{r_j}\\
&\leq\frac{3\hat{c}C^0}{r_j^\frac{1}{2}}
\leq\frac{C_5^*}{r^\frac{1}{2}},\;\;r\in(r_j,r_{j+1}],\;j=1,\ldots,n,
\end{split}
\label{thickness}
\end{equation}
where we have also used \eqref{.5} and $\frac{r}{r_j}\leq 1+\frac{\beta}{2}$ and set $C_5^*=3\hat{c}C^0(1+\frac{\beta}{2})^\frac{1}{2}$.
This and \eqref{varteta-b} imply \eqref{I-geom} with $\mathring{C}=C_4^*+C_5^*$.
\vskip.2cm

3. If $u^{R_n}$ is a minimizer of $J_{B_{R_n}}$ the map $\varrho u^{R_n}(\varrho\cdot)$ is also a minimizer for each rotation $\varrho:\R^2\rightarrow\R^2$. Therefore we can assume that $\theta(\gamma_{n+1}^m)=0$, for $n\geq\bar{n}$ and define
\[Q_n=\{x(r,\theta): \theta\in(\frac{\mathring{C}}{r^\frac{1}{2}},\frac{2\pi}{N}-\frac{\mathring{C}}{r^\frac{1}{2}}),\;r\in(\mathring{r},r_{n+1}),
\;\mathring{r}=\frac{(\mathring{C}N)^2}{\pi^2}\},\]
Since $Q_n\subset\mathscr{R}$, \eqref{ok-R} implies
 \begin{equation}
\vert u^{R_n}(x)-\omega^{j-1}a\vert\leq c\delta^\alpha,\;\;x\in\omega^{j-1}Q_n\setminus\tilde{\Sigma},\;\;j=1,\ldots,N-1.
 \label{ok-R1}
 \end{equation}
 This and Corollary \ref{cor} yield
 \begin{equation}
 \vert u^{R_n}(x)-\omega^{j-1}a\vert\leq\bar{K}e^{-\bar{k}d(x,\partial\omega^{j-1}Q_n)},\;\;x\in\omega^{j-1}Q_n,\;n\geq\bar{n},
 \label{decay-n}
 \end{equation}
 where $\bar{k}$ is as in Corollary \ref{cor} and $\bar{K}>0$ some constant independent of $n$.
 \vskip.2cm

 4. The family of minimizers $\{u^{R_n}\}_n$ is uniformly bounded in $C^{2+\alpha}(B_{R_n};\R)$, for some $\alpha\in(0,1)$. It follows the existence of a subsequence still denoted $\{u^{R_n}\}_n$ that converges in compact in the $C^2$ sense to a map $U:\R^2\rightarrow\R^2$ which is a solution of \ref{elliptic}. $U$ is $C_N$-equivariant and, since the estimate \eqref{decay-n} passes to the limit for $n\rightarrow+\infty$, satisfies \eqref{exp-infty}. The proof is complete.
\subsection{Appendix}\label{app}

\begin{proof}(of Proposition \ref{no-inters})
\vskip.2cm

1. We have
\begin{equation}
\left.\begin{array}{l}
 N_{j,r_j}\cap\mathrm{arc}[p_{j+1}^-,p_{j+1}^+]=\emptyset,\\\\
 N_{j,r_{j+1}}\cap\mathrm{arc}[p_j^-,p_j^+]=\emptyset,
\end{array}\right.
j=1,\ldots,n.
\label{no-intsc}
\end{equation}

If $N_{j,r_{j+1}}\cap C_{r_j}=\emptyset$ \eqref{no-intsc}$_2$ is trivially true. Assume instead that there is $\xi\in N_{j,r_{j+1}}\cap C_{r_j}$. Then $\gamma_j^m$, $\gamma_{j+1}^m$ and $\xi$ are the vertices of a triangle rectangle in $\gamma_{j+1}^m$ and it follows
\begin{equation}
d^\circ(\gamma_j^m,\xi)\geq\vert\gamma_j^m-\xi\vert\geq\vert\gamma_{j+1}^m-\gamma_j^m\vert\geq\vert r_{j+1}-r_j\vert.
\label{fist-step}
\end{equation}
Since
$d^\circ(\gamma_j^m,p_j^\pm)\leq d^\circ(p_j^-,p_j^+)$ and \eqref{r-r<beta} yields
 \[\vert r_{j+1}-r_j\vert>d^\circ(p_j^-,p_j^+)=3\nu+2\lambda_j,\;\;\;j=1,\ldots,n,\]
\eqref{no-intsc}$_2$ follows from \eqref{fist-step}. The same argument applies to \eqref{no-intsc}$_1$.
\vskip.2cm

2. It results
\begin{equation}
[\xi_{r_j^\pm}^+,\xi_{r_j^\pm}^-]\cap(\mathrm{arc}[p_{j-1}^-,p_{j-1}^+]\cup\mathrm{arc}[p_{j+1}^-,p_{j+1}^+]).
\label{no-int+-}
\end{equation}

Since $[\xi_{r_j^-}^+,\xi_{r_j^-}^-]\subset N_{j-1,r_j}$ and $[\xi_{r_j^+}^+,\xi_{r_j^+}^-]\subset N_{j+1,r_j}$
 \eqref{no-intsc}$_2$ implies
\begin{equation}
\begin{split}
&[\xi_{r_j^-}^+,\xi_{r_j^-}^-]\cap\mathrm{arc}[p_{j-1}^-,p_{j-1}^+]=\emptyset,\\
&[\xi_{r_j^+}^+,\xi_{r_j^+}^-]\cap\mathrm{arc}[p_{j+1}^-,p_{j+1}^+]=\emptyset.
\end{split}
\label{sec-steps}
\end{equation}
If $\tau_{j+1}=\tau_j$ we have $[\xi_{r_j^-}^+,\xi_{r_j^-}^-]=[\xi_{r_j^+}^+,\xi_{r_j^+}^-]$ and \eqref{no-int+-} follows from \eqref{sec-steps}. If $\tau_{j+1}\neq\tau_j$ both in case a) and b) it results that $[\xi_{r_j^-}^+,\xi_{r_j^-}^-]$ lies ($[\xi_{r_j^+}^+,\xi_{r_j^+}^-]$ lies) on the half plane determined by $N_{j+1,r_j}$ (by $N_{j-1,r_j}$) that does not contain $\mathrm{arc}[p_{j+1}^-,p_{j+1}^+]$ ($\mathrm{arc}[p_{j-1}^-,p_{j-1}^+]$). This and \eqref{sec-steps} imply \eqref{no-int+-}.
\vskip.2cm

3. From 2. we have that $\xi_{r_j^+}^-$ and $\xi_{r_{j+1}^-}^-$ ($\xi_{r_j^+}^+$ and $\xi_{r_{j+1}^-}^+$) are the extreme of a subarc $\mathscr{C}_j^-\subset\cup_{i=j_1}^{j+1}\partial\mathscr{I}_i^-$ (of a subarc $\mathscr{C}_j^+\subset\cup_{i=j_1}^{j+1}\partial\mathscr{I}_i^+)$. This, since, by definition $[\xi_r^+,\xi_r^-]$,
$[\xi_{r_j^+}^+,\xi_{r_j^+}^-]$ and $[\xi_{r_{j+1}^-}^+,\xi_{r_{j+1}^-}^-]$ have the same direction, implies
\[\xi_r^\pm\in\mathscr{C}_j^\pm\subset\cup_{i=j-1}^{j+1}\partial\mathscr{I}_i^\pm,\;\;r\in(r_j,r_{j+1}),\;j=2,\ldots,n-1,\]
that concludes the proof of \eqref{xi-in dpm}.
\vskip.2cm

4. We have
\begin{equation}
[\xi_r^+,\xi_r^-]\cap[\xi_{r^\prime}^+,\xi_{r^\prime}^-]=\emptyset,\;\;r,r^\prime\in (r_j,r_{j+2})\setminus\{r_{j+1}\},\;\;j=2,\ldots,n-1.
\label{loc-no-int}
\end{equation}
This is obvious if $\tau_{j+1}=\tau_j$. If $\tau_{j+1}\neq\tau_j$ and a) holds, \eqref{loc-no-int} follows from the fact that $\mathscr{C}_j^-$ and $\mathscr{C}_{j+1}^-$ have the extreme $\gamma_{j+1}^m=p_{j+1}^-=\xi_{r_{j+1}^\pm}^-$ in common. The same argument applies if b) holds.
\vskip.2cm

5. Assume that there are $r\in(r_i,r_{i+1})$, $r^\prime\in(r_j,r_{j+1})$ and $\xi$ such that
\begin{equation}
\{\xi\}=[\xi_r^+,\xi_r^-]\cap[\xi_{r^\prime}^+,\xi_{r^\prime}^-].
\label{xi-exists}
\end{equation}
Without loss of generality we can assume that $j\geq i$. From \eqref{xi-in dpm} it follows $j\leq i+2$. On the other hand 4. implies $j>i+1$ and we conclude that $j=i+2$. This and \eqref{xi-in dpm} imply
\begin{equation}
r(\xi)\in[r_{j-1},r_j].
\label{r-interv}
\end{equation}
From \eqref{xi-exists} it follows that $\tau_{j-2}\neq\tau_j$ and therefore that at least one of the following two possibilities holds:
\[\tau_j\neq\tau_{j-1},\quad \tau_{j-1}\neq\tau_{j-2}.\]
We discuss the case $(\tau_{j-1},\tau_j)$ negative. The analysis of the other possibilities is analogous. We have \[(\tau_{j-1},\tau_j)\;\;\text{negative}\;\;\Rightarrow\;\;\gamma_j^m=p_j^-=\xi_{r_j^\pm}^-.\]
This and $r^\prime>r_j$ imply $r(\xi_{r^\prime}^-)>r_j$ which, since $\xi\in[\xi_{r^\prime}^+,\xi_{r^\prime}^-]$, is in contradiction with \eqref{r-interv} and therefore with the existence of $\xi$. The proof is complete.
\end{proof}
\vskip1cm
\begin{proof}(of Proposition \ref{-1})
\vskip.2cm

1. From  \eqref{xi-in dpm} and Lemma \ref{balls-} it follows that if $j=2,\ldots,n-2$, for each $x\in\cup_{i=j-1}^{j+1}\partial\mathscr{I}_i^\pm$ and, in particular for $\xi_r^\pm$, it results
\[\vert u^{R_n}-a_\pm\vert\leq c\delta^\alpha,\;\;\text{on}\;\:B_{\eta\lambda_{j-1}}\setminus\tilde{\Sigma}.\]
This concludes the proof of (i) for $j=2,\ldots,n-2$.
\vskip.2cm

2. For $r$ near $r_{n+1}$, one or both the extreme of $[\xi_r^-,\xi_r^+]$ may lie on $\mathrm{arc}[\hat{p}_{n+1}^-,\hat{p}_{n+1}^+]$. Since $\vert\hat{p}_{n+1}^+-\hat{p}_{n+1}^-\vert\leq 3\nu$ A sufficient condition to exclude this is $r\leq r_{n+1}-\bar{\rho}$,  $\bar{\rho}=3\nu$.
\vskip.2cm

3. Assume that $\rho=\rho(x)=r_{n+1}-r(x)$ satisfies
\begin{equation}
\rho\geq 4\eta\lambda_{n-1}.
\label{trade}
\end{equation}
Then we have $B_{\eta\lambda_{n-1}}(x)\subset B_{\frac{\rho}{2}}(x)$ and from Lemma \ref{ball-exists-n} it follows
\begin{equation}
\begin{split}
&x\in\cup_{j=n-1,n}\partial\mathscr{I}_j^\pm\;\;\text{and}\;\;\rho\geq 4\eta\lambda_{n-1}\\
&\Rightarrow\;\;\;\vert u^{R_n}-a_\pm\vert\leq c\delta^\alpha,\;\;\text{on}\;\;B_{\eta\lambda_{n-1}}(x)\setminus\tilde{\Sigma}.
\end{split}
\label{trade-}
\end{equation}
\noindent
Fix $\bar{\eta}=\frac{1}{4}$. Then
\[\rho< 4\eta\lambda_{n-1},\]
implies $B_{\bar{\eta}\rho}(x)\subset B_{\eta\lambda_{n-1}}(x)$. From this, $B_{\bar{\eta}\rho}(x)\subset B_{\frac{\rho}{2}}(x)$  and Lemma \ref{ball-exists-n} we have
\begin{equation}
\begin{split}
&x\in\cup_{j=n-1,n}\partial\mathscr{I}_j^\pm\;\;\text{and}\;\;\rho<4\eta\lambda_{n-1}\\
&\Rightarrow\;\;
\vert u^{R_n}-a_\pm\vert\leq c\delta^\alpha,\;\;\text{on}\;\;B_{\bar{\eta}\rho}(x)\setminus\tilde{\Sigma}.
\end{split}
\label{trade+}
\end{equation}
\vskip.2cm

4. Let $w_n$ the direction vector of the ray $L_{\theta(\gamma_{n+1}^m)}$ through $\gamma_{n+1}^m$ and let $\chi_n$ the angle between $w_n$ and $\tau_n$ positive if $(w_n,\tau_n)$ is positive. We assume $\chi_n\geq 0$. The same argument with obvious modifications applies to the case $\chi_n<0$. Recall the definition of $\ell_{r_{n+1}}^\pm$ in Lemma \ref{ball-exists-n} and define $\tilde{\xi}^-$ by setting

\[\begin{split}
&\tilde{\xi}^-\in\ell_{n+1}^-,\\
&r(\tilde{\xi}^-)=r_{n+1}-4\eta\lambda_{n-1}.
\end{split}\]

Note that $\tilde{\xi}^-$ satisfies \eqref{trade} with the equality sign. We define $\tilde{r}$ as the value of $r$ such that $\tilde{\xi}^-=\xi_{\tilde{r}}^-$, that is we let $\tilde{r}$ be determined by the condition that $\gamma^m(\tilde{r})$ coincides with the intersection of $\{x=\gamma_{n+1}^m+t\tau_n,\;t\in\R\}$ with $\{x=\tilde{\xi}^-+t\upsilon_n,\;t\in\R\}$. With this choice of $\tilde{r}$ we have $N_{n,\tilde{r}}=\{x=\tilde{\xi}^-+t\upsilon_n,\;t\in\R\}$ and $[\tilde{\xi}^-,\tilde{\xi}^+]=[\xi_{\tilde{r}}^-,\xi_{\tilde{r}}^+]$ where $\tilde{\xi}^+$ is the other extreme of the connected component of $N_{n,\tilde{r}}\cap\mathscr{I}$ that contains $\gamma^m(\tilde{r})$. Since $\tilde{\xi}^-$ satisfies \eqref{trade} with the equality sign, $\xi_r^-$ satisfies \eqref{trade-} or \eqref{trade+} depending on wether  $r\leq\tilde{r}$ or $r>\tilde{r}$. This and the fact that $\chi_n\geq 0$ implies $r(\xi_r^-)\geq r(\xi_r^+)$ show the existence of $\tilde{r}$ such that \eqref{A} and \eqref{B} hold.

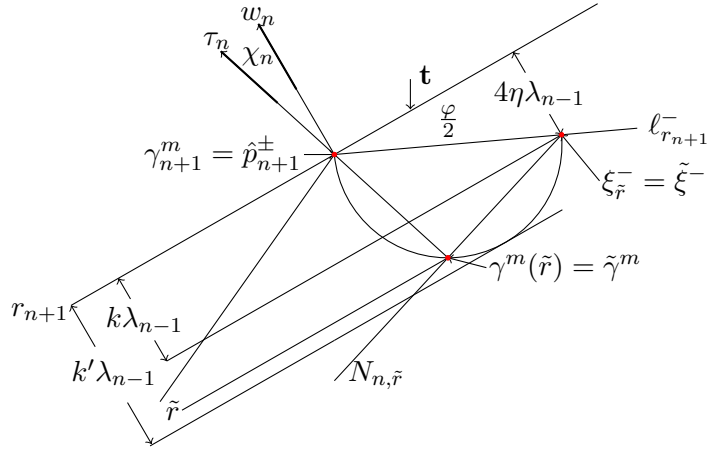
\begin{figure}
  \begin{center}
\begin{tikzpicture}[scale=1]

\path[fill=red] (0,0) circle [radius=.035];
\draw (0,0)--(3.464,2);
\draw (0,0)--(-3.464,-2);

\draw (0,0)--(3.985,.349);
\draw (0,0)--(-2.294,-3.277);

\path[fill=red] (2.989,.2615) circle [radius=.035];
\draw [] (0,0) arc [radius=1.5, start angle=-175, end angle=5];

\draw (0,0)--(-1,1.732);
\draw (1.4945,-1.369)--(0,-2.9995);
\draw [->](-.5,.866)--(-1,1.732);
\draw [thick](-.5,.866)--(-1,1.732);
\path[fill=red] (1.4945,-1.369) circle [radius=.035];
\draw (-1.4945,1.369)--(1.4945,-1.369);
\draw[->](-.74725,.6845)-- (-1.4945,1.369);
\draw[thick](-.74725,.6845)-- (-1.4945,1.369);
\draw (2.989,.2615)--(1.4945,-1.369);
\draw[<-] (2.989,.2615)--(2.8305,.536);
\draw[<-] (2.355,1.359)--(2.5135,1.0846);
\draw (0,-2.4643)--(3,-.7323);
\draw (0,-2.4643)--(-2.4248,-3.8634);

\draw[<-] (2.989,.2615)--(3.489,-.349);

\draw (2.989,.2615)--(-2.207,-2.7385);
\draw [<-](-2.207,-2.7385)--(-2.3655,-2.4641);
\draw [<-](-2.841,-1.641)--(-2.6825,-1.9154);
\draw [<-](-3.464,-2)--(-3.147,-2.5488);
\draw [->](1,1)--(1,.6);

\draw [<-](-2.405,-3.8375)--(-2.722,-3.2887);
\draw (1.4945,-1.369)--(-2.005,-3.39);
\draw [->](-.4,0)--(0,0);
\draw [<-](1.4945,-1.369)--(2,-1.5);
\node[above] at (1.5,.1) {$\frac{\varphi}{2}$};
\node [] at (-1,1.3) {$\chi_n$};
\node [right] at (3.385,-.349) {$\xi_{\tilde{r}}^-=\tilde{\xi}^-$};
\node [right] at (1.9,-1.5) {$\gamma^m(\tilde{r})=\tilde{\gamma}^m$};
\node [left] at (-.3,0) {$\gamma_{n+1}^m=\hat{p}_{n+1}^\pm$};
\node [left] at (-1.9,-3.39) {$\tilde{r}$};
\node [above] at (1.2,.8) {$\mathbf{t}$};
\node [left] at (-3.364,-2.1) {$r_{n+1}$};
\node [] at (-1,1.832) {$w_n$};
\node [] at (-1.5445,1.509) {$\tau_n$};
\node [] at (2.7,.8) {$4\eta\lambda_{n-1}$};
\node [] at (-2.5,-2.2) {$k\lambda_{n-1}$};
\node [] at (-2.9,-2.9) {$k^\prime\lambda_{n-1}$};
\node [right] at (4,.35) {$\ell_{r_{n+1}}^-$};
\node [below] at (0.55,-2.6) {$N_{n,\tilde{r}}$};
\path[fill=red] (0,0) circle [radius=.035];
\path[fill=red] (1.4945,-1.369) circle [radius=.035];
\path[fill=red] (2.989,.2615) circle [radius=.035];
\end{tikzpicture}
\end{center}
\caption{$k$ and $k^\prime$ and $\tilde{r}$.}
\label{tilde-r}
\end{figure}

To complete the proof we recall that $\lambda_{n-1}\rightarrow+\infty$ as $n\rightarrow+\infty$ while $\vert\gamma_{n+1}^m-\hat{p}_{n+1}^-\vert\leq\frac{3\nu}{2}$ with $\nu$ is independent of $n$. It follows that, by accepting an error of $\mathrm{O}(\frac{1}{n})$ we can identify $\gamma_{n+1}^m$ with $\hat{p}_{n+1}^-$ and the circumference $C_r$ with a straight line parallel to the tangent $\mathbf{t}$ to $C_{r_{n+1}}$ at $\gamma_{n+1}^m$. In the same order of approximation $\gamma^m(\tilde{r})$ can be identified with the intersection $\tilde{\gamma}^m\neq\gamma_{n+1}^m$ of $\{x=\gamma_{n+1}^m+t\tau_n,\;t\in\R\}$ with the circumference of diameter $[\gamma_{n+1}^m,\tilde{\xi}^-]$ and $N_{n,\tilde{r}}\cap\mathscr{I}$ with the line through $\tilde{\gamma}^m$ and $\tilde{\xi}^-$ (see Figure \ref{tilde-r}). Under these identifications that are equivalent to pass to the limit for $n\rightarrow+\infty$ we see that $\tilde{r}$ has an upper bound $\approx \tilde{r}\leq r_{n+1}-k\lambda_{n-1}$ ($k\approx 4\eta$) when $\chi_n=0$ ($r(\tilde{\xi}^+)=r(\tilde{\xi}^-)$) and a lower bound that corresponds to the situation where the line parallel to $\mathbf{t}$ through $\tilde{\gamma}^m$ is tangent to the circumference with diameter $[\gamma_{n+1}^m,\tilde{\xi}^-]$. If $\frac{\varphi}{2}$ is the limit value of the angle between $[\gamma_{n+1}^m,\tilde{\xi}^-]$ and $\mathbf{t}$ we find, see Figure \ref{tilde-r}, that $\tilde{r}\geq r_{n+1}-k^\prime\lambda_{n-1}$ with $k^\prime\approx 2(1+\frac{2}{\sin{\frac{\varphi}{2}}})$. This concludes the proof of (ii).
\vskip.2cm

5. It remain to show that (i) is valid also for $j=n-1$. This follows from (ii) and Proposition \ref{no-inters}.
The proof is complete.
\end{proof}

\end{document}